\documentclass{article}
\usepackage[utf8]{inputenc}
\usepackage{graphicx,xcolor}
\usepackage{amsmath,amsfonts,amsthm,stackrel,esint}
\usepackage{hyperref}
\usepackage{xfrac}
\usepackage{mathrsfs}
\newcommand\res{\mathop{\hbox{\vrule height 7pt width .5pt depth 0pt
			\vrule height .5pt width 6pt depth 0pt}}\nolimits}
\newcommand\LL{\res}
\newcommand\weakto{\rightharpoonup}
\newcommand\eps{\varepsilon}
\newcommand\Id{\mathrm{Id}}

\newcommand\loc{\mathrm{loc}}
\newcommand\Lip{\mathrm{Lip\,}}
\newcommand\supp{\mathrm{supp\,}}
\newcommand\Tr{\mathrm{Tr\,}}
\newcommand\dist{\mathrm{dist}}
\newcommand\qc{\mathrm{qc}}
\newcommand\qcinfty{{\mathrm{qc},\infty}}

\newcommand\scal{\mathrm{scal}}

\newcommand\SO{\mathrm{SO}}

\newcommand\R{\mathbb{R}}
\newcommand\N{\mathbb{N}}
\newcommand\Z{\mathbb{Z}}

\newcommand\calU{\mathcal{U}}
\newcommand\calV{\mathcal{V}}
\newcommand\calA{\mathcal{A}}
\newcommand\calB{\mathcal{B}}
\newcommand\calL{\mathcal{L}}
\newcommand\calH{\mathcal{H}}
\newcommand\calF{\mathcal{F}}
\newcommand\calM{\mathcal{M}}

\newtheorem{theorem}{Theorem}[section]

\newtheorem{proposition}[theorem]{Proposition}
\newtheorem{lemma}[theorem]{Lemma}

\newtheorem{corollary}[theorem]{Corollary}

\numberwithin{equation}{section}
\newcommand\cutoffconst{{M}}

\newcommand\Functeps{\mathcal F_\eps}
\newcommand\Functepspq{\mathcal F_{\eps,p,q}}
\newcommand\Functlim{\mathcal F_{p,q}}
\newcommand\dx{{\mathrm d}x}
\newcommand\dy{{\mathrm d}y}
\newcommand\ds{{\mathrm d}s}
\newcommand\dt{{\mathrm d}t}
\newcommand\dH{{\mathrm d}{\mathcal H}}
\newcommand\dd{{\mathrm d}}

\newcommand\Psiinfty{{\Psi_\infty}}
\newcommand\gpsup{g_{\sup}}
\newcommand\gpinf{g_{\inf}}
\newcommand\pcube{{y}}

\begin{document}
\begin{center}
  {\Large
{Superlinear {free-discontinuity} models:\\[1.5mm]
relaxation and phase-field approximation}}\\[5mm]
{\today}\\[5mm]
Sergio Conti$^{1}$, Matteo Focardi$^{2}$ and Flaviana Iurlano$^{3}$\\[2mm]
{\em $^{1}$
 Institut f\"ur Angewandte Mathematik,
Universit\"at Bonn,\\ 53115 Bonn, Germany}\\[1mm]

{\em $^{2}$ DiMaI, Universit\`a di Firenze\\ 50134 Firenze, Italy}\\[1mm]
{\em $^{3}$ DIMA, Universit\`a di Genova\\ 16146 Genova, Italy}
\\[3mm]
    \begin{minipage}[c]{0.8\textwidth}
{In this paper we develop the Direct Method {in} the Calculus of Variations for}
free-discontinuity energies {{whose} bulk and surface densities exhibit superlinear growth for large gradients and small jump amplitudes, respectively.}
{A distinctive}
{feature of this kind of models} is that the functionals
{are defined on}
$SBV$ {functions whose} jump sets {may have} infinite measure. {Establishing general lower semicontinuity and relaxation results
{in this setting}
requires new analytical techniques. {In addition, we propose} a variational approximation of {certain superlinear} energies
via phase-field models.}
\end{minipage}
\end{center}

\tableofcontents

\section{Introduction}
Variational models involving free-discontinuity problems play a fundamental role in fields such as fracture
mechanics, image processing, and materials science. These models describe systems {in which} the energy consists of
competing bulk and surface terms.
A prototypical example takes the form
\begin{equation}\label{eq0}\int_{\Omega\setminus K} \Psi(\nabla u) \dx+\int_{K}g([u],\nu)\dH^{n-1}.
\end{equation}
In the context of solid mechanics, the domain $\Omega\subset\R^n$ represents the reference configuration of a hyperelastic
body, while $u:\Omega\to\R^m$ denotes its deformation. 
The set $K\subset\Omega$ is
the unknown (sufficiently regular) set of discontinuity points of $u$; finally $\nu$ and $[u]$ denote, respectively,
the normal vector to $K$ and the difference of the traces of $u$ across $K$. The first term in the energy \eqref{eq0}
represents the stored elastic energy, while the second term 
describes the 
contribution concentrated on the fractured surface $K$, which may {account for both} energy and
dissipation. External loads and
Dirichlet boundary conditions may {also be incorporated into} the model. For a {comprehensive} overview of the subject, see
\cite{AmbrosioBraides95,Braides,FrancfortMarigo,AFP,BFM} {and references therein. While {most of} the literature {focuses on}
the case where the discontinuity set $K$ has finite $\calH^{n-1}$
measure, we {aim to extend} the framework to {settings in which} $\calH^{n-1}(K)$ may be infinite.
This generalization requires new analytical techniques to establish lower semicontinuity and relaxation
{results} under weaker regularity assumptions.}

It is nowadays clear that the correct space to relax and study problems of type \eqref{eq0} is that of functions of
bounded variation $BV(\Omega;\R^m)$ (or related spaces $SBV$, $GBV$, $GSBV$,  see \cite{AFP}). More precisely,
the weak formulation of problem \eqref{eq0} {is}
\begin{equation}\label{eq1}\int_{\Omega} \Psi(\nabla u) \dx+\int_{J_u}g([u],\nu_u)\dH^{n-1},
\end{equation}
having replaced $K$ by the jump set $J_u$ of $u$ and having labeled the corresponding normal vector $\nu_u$.

{In several models} the energy density $\Psi:\R^{m{\times} n}\to [0,\infty)$ is assumed to have $q$-growth at infinity,
$q>1$,  while the surface energy density $g:\R^m{\times} S^{n-1}\to[0,\infty)$ is commonly chosen either such that
\[g(z,\nu)\geq c>0,\]
for some $c>0$ and for all $z\in\R^m$ and $\nu\in S^{n-1}$, or such that {$g(0,\nu)=0$ and}
\begin{equation}\label{eq2a}
\lim_{|z|\to0}\frac{g(z,\nu)}{|z|}\in(0,\infty],
\end{equation}
for all $\nu\in S^{n-1}$. 
While many classical models assume surface densities that remain bounded {away from zero} for small jumps, we consider
\emph{superlinear growth}, where the surface energy density {converges to zero} as the jump amplitude vanishes. This assumption
reflects physical situations where fracture toughness increases continuously at small scales and leads to new
mathematical challenges. Specifically, {we will assume} for all $\nu\in S^{n-1}$
\begin{equation}\label{eq2}
\lim_{|z|\to0}\frac{g(z,\nu)}{|z|}=\infty.
\end{equation}
{Under the latter assumptions the functional setting of the problem is naturally given by the space of Special Functions
with Bounded Variation {$SBV$} introduced by De Giorgi and Ambrosio \cite{DeGiorgiAmbrosio} (see also
\cite[Chapters~4-8]{AFP}).}
Models of this type appear for instance in image processing
\cite{BouchitteDubsSeppecher,AlicandroBraidesGelli,AlicandroGelli,Morini}, in fracture mechanics
\cite{FokouaContiOrtiz2014,ArizaContiOrtiz2024}, in screw dislocations in single crystal plasticity
\cite{DeLucaScalaVanGoethem},
{in the theory of smectic thin films \cite{BallCanevariStroffolini}, and in the variational derivation of the Read and
Shockley formula for the energy of small angle grain boundaries in polycrystals
\cite{LauteriLuckhaus,FortunaGarroniSpadaro}.
{These models include the case of jump sets with infinite length, but the analysis of the corresponding function
space is still missing.}

In view of the closure and compactness properties of the underlying space $(G)SBV$ \cite[Theorems~4.7, 4.8 and 4.36]{AFP}, the existence of minimizers for a functional of the form \eqref{eq1}
is known using the Direct Method of the Calculus of Variations in the isotropic case. {This refers to the setting where}
the bulk energy density $\Psi(\xi)$
{depends only on the modulus of the gradient, i.e.,} $\Psi(\xi)=\hat\Psi(|\xi|)$,
and
the surface energy density $g(z,\nu)$
{depends only on the jump magnitude, i.e.,}
$g(z,\nu)=\hat g(|z|)$. {Under the assumption that}
both $\Psi$ and $g$ satisfy the {aforementioned} growth conditions, $\Psi$ is {convex},
and $g$ is concave, {existence can be established.} {See \cite{BouchitteButtazzo1990,BouchitteButtazzo1993,AFP} and} below for a more {detailed discussion of the relevant} literature.

In this paper we {develop the theory of the Direct Method for this class of problems, focusing in particular on the properties of the appropriate function spaces, as well as on the} lower semicontinuity and relaxation of energies of the form \eqref{eq1}, with more general
bulk and surface densities. 
{We also {construct} a
variational approximation via phase-field models, which is {particularly relevant for} numerical approximation.

The lower semicontinuity property for anisotropic surface energies was first studied in \cite{AmbrosioBraides1990}
in the framework of minimal partitions, i.e. $SBV$ functions 
{which take finitely many values}, obtaining as a necessary condition the $BV$-ellipticity of $g$ (see \eqref{e:BV ellipticity} for
the precise definition), see \cite[Theorem~5.14]{AFP}.

More generally, the lower semicontinuity property for anisotropic surface integrals with $BV$-elliptic, bounded and
continuous densities $g$ was first established in \cite[Theorem~3.3]{Ambrosio1990} along sequences $(u_j)_j
\subset SBV\cap L^\infty(\Omega;\R^m)$ such that
\begin{itemize}
\item[($g$-a)] $u_j\to u$ in measure on $\Omega$, $u\in SBV\cap L^\infty(\Omega;\R^m)$, 
$\sup_j\|u_j\|_{L^\infty(\Omega)}<\infty$, and $(\nabla u_j)_j$ is equiintegrable 
in $L^1(A;\R^{m\times n})$ for every {$A\subset\subset\Omega$};
\item[($g$-b)] $\sup_j\calH^{n-1}(J_{u_j})<\infty$.
\end{itemize}
We note that the statement of  \cite[Theorem~3.3]{Ambrosio1990} is formulated by requiring $g$ to be bounded from below
by a strictly positive constant, nevertheless the conclusion holds in the more general framework
as stated above with the same proof.

In what follows we extend \cite[Theorem~3.3]{Ambrosio1990} to a large class of energies compatible with Ambrosio's $SBV$ compactness theorem.
More precisely, we consider surface integrands $g:\R^m\times S^{n-1}\to[0,\infty)$ such that
\begin{equation}\label{gBVell}
{g\text{ is $BV$-elliptic,}}
\end{equation}
\begin{equation}\label{e:g consistent}
g(z,\nu)=g(-z,-\nu),
\end{equation}
\begin{equation}\label{e:g grwoth g0}
 \frac1c g_0(|z|)\le g(z,\nu) \le c g_0(|z|),
\end{equation}
\begin{equation}\label{eqgg0}
|g(z,\nu) -g(z',\nu)| \le c g_0(|z-z'|),
\end{equation}
for some $c>0$ and for all $(z,\nu)\in\R^m\times S^{n-1}$, where 
$g_0\in C^0([0,\infty);[0,\infty))$ is subadditive, nondecreasing, with {$g_0^{-1}(0) = \{0\}$},
and such that
for some $\gamma\in(0,1)$
 \begin{equation}\label{e:g0 superlinear}
  \lim_{s\to0} \frac{g_0(s)}{s^\gamma}=\ell\in(0,\infty)
 \end{equation}
{(extensive comments on the assumption in \eqref{e:g0 superlinear} will be given after the proof of Proposition~\ref{approximation}).}
{In view of \eqref{e:g grwoth g0}, $g$ is necessarily locally bounded. In addition, $g$ is continuous in the $z$-variable uniformly with respect to $\nu$, with $g_0$ as modulus of continuity, thanks to \eqref{eqgg0}.}
\begin{proposition}\label{p:lsc surface}
Let $g_0$ and $g$ {satisfy} \eqref{gBVell}-\eqref{e:g0 superlinear}, and let ${u_j}\in {(GSBV(\Omega))^m}$ {satisfy}
\begin{itemize}
\item[($g$-a')] $u_j\to u$ in {measure}, 
$u\in {(GSBV(\Omega))^m}$, and
$(\nabla u_j)_j$ is equiintegrable in  $L^1(A;\R^{m\times n})$ for every {$A\subset\subset\Omega$}.
\end{itemize}
Then
\begin{equation}\label{lscsurf}
\int_{J_u}g([u],\nu_u)\dH^{n-1}
\le\liminf_{j\to\infty} \int_{J_{u_j}}g([u_j],\nu_{u_j})\dH^{n-1}.
\end{equation}
\end{proposition}
}

Our proof {follows} the key idea introduced in \cite[Theorem~3.3]{Ambrosio1990}, replacing
the sequence $u_j$ by a sequence of piecewise constant functions via {the} coarea formula.
However, since in our setting the length of the jump sets may be infinite, the estimates are more 
{delicate} and {rely} on monotonicity, subadditivity, and $\gamma$-growth of $g_0$ at the origin. {All approximation} errors are {explicitly}
quantified in terms of $g_0$.

{Instead, the lower semicontinuity property for volume energies under quasiconvexity and $q$-growth conditions, $q>1$,
on the integrand $\Psi$ was first established by Ambrosio in \cite{Ambrosio1994} (see also
\cite[Theorem~5.29]{AFP}) along sequences $(u_j)_j \subset SBV(\Omega;\R^m)$ such that
\begin{itemize}
\item[($\Psi$-a)] $u_j\to u$ in $L^1(\Omega;\R^m)$, $u\in SBV(\Omega;\R^m)$, and 
$\sup_j\|\nabla u_j\|_{L^q(\Omega)}<\infty$, $q\in(1,\infty)$;
\item[($\Psi$-b)] $\sup_j\calH^{n-1}(J_{u_j})<\infty$.
\end{itemize}
The result was then extended by Kristensen in \cite{Kristensen99} under the assumptions:
\begin{itemize}
\item[($\Psi$-a')] the conditions in ($\Psi$-a) hold
{for an exponent in the larger range} $q\in[1,\infty]$;
\item[($\Psi$-b')] there is $g_0\in C^0([0,\infty);[0,\infty))$ concave, nondecreasing, with $g_0(0)=0$, and such that
 \begin{equation}\label{e:g0 superlinear at 0}
  \lim_{s\to0} \frac{g_0(s)}{s}=\infty\,,
 \end{equation}
for which
\begin{equation}\label{e:bound surface nrg}
\sup_j\int_{J_{u_j}}g_0(|[u_j]|)\dH^{n-1}<\infty\,;
\end{equation}
\item[($\Psi$-c)]
{$\Psi\in C^0(\R^{m\times n};\R)$,}
$(\max\{-\Psi(\nabla u_j),0\})_j$ is equiintegrable.
\end{itemize}
Kristensen's proof hinges upon the fact that Young measures generated by sequences $(\nabla u_j)_j$ satisfying
($\Psi$-a') and ($\Psi$-b') are actually gradient $q$-Young measures for $\calL^n$-almost {every point}. The use of Young
measures allows {one} to {treat} the problem for a class of normal integrands depending on the lower order variables $(x,u)$.
We {present} here a more elementary proof of the result for autonomous functionals, which still relies on some fundamental ideas introduced in \cite{Kristensen99}. Key ingredients {include} an approximation argument for
nonnegative, superlinear quasiconvex functions by quasiconvex ones
{with linear growth}
{(see Proposition~\ref{prophdeltaqc} below)}, and an extension of
the quasiconvexity inequality for quasiconvex functions
{with linear growth} to $BV$ test functions
{(see Equation~\eqref{e:BV qcvxty} below).}
For the former, rather than {relying on} the refined theory of Young measures mentioned above, we employ a truncation argument {based on the} maximal function (cf. Lemma~\ref{lemmatrunc}), which is by now elementary.
\begin{proposition}\label{p:lsc bulk}
Let $\Psi:\R^{m\times n}\to\R$ be a quasiconvex integrand satisfying
\begin{equation}\label{e:Psi gc gen}
|\Psi(\xi)|\leq C(1+|\xi|^q)
\end{equation}
for some $q\in[1,\infty)$ and $C>0$. 
Let ${u_j}\in {(GSBV(\Omega))^m}$ satisfy
\begin{itemize}
{\item[($\Psi$-a'')] $u_j\to u$ 
{in measure},
$u\in {(GSBV(\Omega))^m}$,
and
$\sup_j\|\nabla u_j\|_{L^q(\Omega)}<\infty$;}
\item[($\Psi$-b'')] \eqref{e:bound surface nrg} holds for some $g_0\in C^0([0,\infty);[0,\infty))$ subadditive,
nondecreasing, with $g_0(0)=0$, and satisfying \eqref{e:g0 superlinear at 0};
\end{itemize}
and ($\Psi$-c).
Then,
\begin{equation}\label{e:lsc volume}
\int_\Omega\Psi(\nabla u)\dx\le\liminf_{j\to\infty} \int_\Omega\Psi(\nabla u_j)\dx.
\end{equation}
\end{proposition}

As a consequence of the superlinearity assumptions of the energy densities, we combine 
Propositions~\ref{p:lsc surface} 
and~\ref{p:lsc bulk} 
to deduce a lower semicontinuity result for the full energy in
\eqref{eq1} (cf. Theorem~\ref{t:lsc}). An $L^1$ lower semicontinuity result for free-discontinuity energies along sequences satisfying ($\Psi$-a) and ($g$-b),
i.e. with gradients equi-bounded in some $L^q$, $q>1$, and with equi-bounded measure of the jump sets, has been previously
established in \cite[Theorem~4]{BFLM}.

A key challenge in the analysis of free-discontinuity energies is their relaxation under minimal regularity
assumptions. Theorem~\ref{relaxation} provides a relaxation result for functionals of type \eqref{eq1} without
requiring $\Psi$ to be quasiconvex or $g$ to be $BV$-elliptic.
The proof makes essential use of the density result obtained
in \cite[Corollary~2.3]{CFI23}.
{\begin{theorem}\label{relaxation}
Let $\Psi:\R^{m\times n}\to[0,\infty)$ be continuous and satisfy
		\begin{equation}\label{e:Psi gc}
		\Big(\frac1c |\xi|^q-c\Big)\vee0\le \Psi(\xi)\le c(|\xi|^q+1),
		\end{equation}
for all $\xi\in\R^{m\times n}$ and some $q>1$, {and let $g_0$ and $g$ satisfy \eqref{e:g consistent}-\eqref{e:g0
superlinear}.}
Let $H:L^1(\Omega;\R^{m})\to[0,\infty]$ be
\begin{equation}\label{H}
H(u):=\begin{cases}\displaystyle \int_\Omega \Psi(\nabla u)\dx+\int_{J_u}g([u],\nu)\dH^{n-1}, & \text{if }u\in
SBV(\Omega;\R^m),\\
\infty, & \text{otherwise.}
\end{cases}
\end{equation}
Then, the relaxation $\overline{H}$ with respect to the strong topology of
$L^1(\Omega;\R^m)$ is the functional
\begin{equation}\label{Hbar}
\overline H(u)=\int_\Omega \Psi^\qc(\nabla u)\dx + \int_{J_u}g_{BV}([u],\nu_u)\dH^{n-1},
\end{equation}
if $u\in
{(GSBV(\Omega))^m}$
with $\nabla u\in L^q(\Omega;\R^{m\times n})$, and
$\overline H(u)=\infty$ otherwise.
\end{theorem}
For the definition of the quasiconvex envelope $\Psi^{\qc}$ of $\Psi$ and of the $BV$-elliptic envelope $g_{BV}$ of $g$ we
refer to \eqref{e:qc envelope} and \eqref{e:BV elliptic envelope}, respectively.} {A similar statement holds by replacing the strong $L^1$-convergence by the convergence in measure. We leave the details to the interested reader.}

{A powerful approach to approximating free-discontinuity problems is through \emph{phase-field methods}, which
replace sharp {discontinuities} with diffuse interfaces.
This approach is widely used in numerical simulations of fracture
mechanics and {offers many advantages in} computational implementation. In the final part of this paper, we establish a
$\Gamma$-convergence result for {a class of} anisotropic phase-field  functionals
of Ambrosio-Tortorelli type,
{thereby} rigorously justifying their
use as approximations to functionals of {the form}}
\eqref{eq1}, {where} $\Psi$ and
$g$ {exhibit} superlinear growth in the sense
described above.

 Phase-field approximations are {especially} popular in the numerical literature {on fracture mechanics}, since functionals
defined on Sobolev spaces are {significantly more tractable from a computational standpoint}.
{In some contexts, they are also interpreted as genuine} diffuse
models and used in {place} of the {corresponding} singular limit models. {Various phase-field models have been proposed and
studied in the mechanics literature on cohesive-zone models for fracture; see, for example,
\cite{LORENTZ20111927,
Verhoosel2013_43,
Wu2017unified,
WuNguyenNguyenChiSutula2020phase,
Chen2021Deborst2021phase,
LammenContiMosler2023,
feng2023unified,
LammenContiMosler2025, ACF25-I, ACF25-II} and references therein.}
{For a broader overview of the literature on this topic,}
we refer to the introductions of
\cite{ContiFocardiIurlano2016,ContiFocardiIurlano2022}.

Our phase-field models are generalizations to the vector-valued anisotropic case of those considered in
\cite[Section~7.2]{ContiFocardiIurlano2016}. For all $\eps>0$, $p$, $q>1$, and $\ell>0$,  we consider the functionals
$\Functepspq\colon L^1(\Omega;\R^{m+1})\to[0,\infty]$ given by
\begin{equation}\label{functeps p intro}
 {\Functepspq(u,v)}:= \int_\Omega \left( f_{\eps,p,q}^q(v) \Psi(\nabla u) +
	\frac{(1-v)^{q'}}{q'q^{\sfrac{q'}q}\eps} + \eps^{q-1}|\nabla v|^q\right) \dx
\end{equation}
if $(u,v)\in W^{1,q}(\Omega;\R^m\times [0,1])$ and $\infty$ otherwise, where
$q'{=q/(q-1)}$ denotes the conjugate exponent of $q$,  and for every
$t\in[0,1)$
\begin{equation*}
f_p(t):=\frac{\ell t}{(1-t)^p},\qquad
f_{\eps,p,q}(t):= 1\wedge \eps^{{1-\sfrac1q}} f_p(t), \qquad
{f_{\eps,p,q}(1):=1}\,.
\end{equation*}
{For the precise set of assumptions on $\Psi$ we refer to Section~\ref{ss:data}. For the sake of simplicity, in this
introduction we only consider the case in which the $q$-recession function $\Psi_\infty$ of $\Psi$ satisfies the
so-called projection property (see \eqref{e:Psiinfty} for the definition of $\Psi_\infty$), namely
\begin{equation}\label{e:projectionproperty intro}
\Psi_\infty(\xi)\geq \Psi_\infty(\xi\nu\otimes\nu) \quad \text{for every $(\xi,\nu)\in\R^{m\times n}\times S^{n-1}$}.
\end{equation}
{For example, consider for $n=m=2$ and
$\alpha>0$ the functions
\begin{equation}
\begin{split}
 \psi^{(\alpha)}(\xi):= &
 (|\xi|^2-2)_+^2+\alpha (\det\xi-1)^2 ,\\
\hat\psi^{(\alpha)}(\xi):=& (|\xi|^2-2\det\xi )^2 +\alpha(\det\xi-1)^2,
\end{split}
 \end{equation}
 where $(f)_+^2:=(f\vee 0)^2$.
These functions are polyconvex, obey
the growth condition \eqref{e:Psi gc} with $q=4$, and
are minimized for $\xi\in \SO(2)$.
The $q$-recession functions are $\psi^{(\alpha)}_\infty(\xi)=|\xi|^4+\alpha(\det \xi)^2$, and $\hat\psi^{(\alpha)}_\infty(\xi)=(|\xi|^2-2\det\xi)^2+\alpha(\det \xi)^2$. The first one clearly obeys \eqref{e:projectionproperty intro}.
However,
considering 
$\nu=e_1$ and
$\xi=\mathrm{diag}(1,t)$ for $t\to0$ shows that
$\hat\psi^{(\alpha)}_\infty$
does not obey
\eqref{e:projectionproperty intro}.}

Under such an assumption we obtain the following result. 
\begin{theorem}\label{t:onedim intro}
Let $\Psi$ satisfy \eqref{e:Psi gc}, \eqref{e:projectionproperty intro}, {and \eqref{eqpsipsiinf}.}
Then for all $(u,v)\in L^1(\Omega;\R^{m+1})$ it holds
\[
\Gamma(L^1)\text{-}\lim_{\eps\to0}\Functepspq(u,v)=\Functlim(u,v),
\]
where
\begin{equation*}
\Functlim(u,v):=\int_\Omega \Psi^\qc(\nabla u)\dx + \int_{J_u}g([u],\nu_u)\dH^{n-1},
\end{equation*}
if $u\in (GSBV\cap L^1(\Omega))^m$ {with $\nabla u\in L^q(\Omega;\R^{m\times n})$}
and $v=1$ $\calL^n$-a.e. on $\Omega$, and $\Functlim(u,v):=\infty$ otherwise, where 
\begin{equation*}
g(z,\nu):=\lim_{T\uparrow\infty}\inf_{{\calV}_z^T}
\int_{-\sfrac T2}^{\sfrac T2}\Big(f_p^q(\beta(t))\Psi_\infty\big({\alpha'(t)}\otimes\nu\big)
+\frac{(1-\beta(t))^{q'}}{q'q^{\sfrac{q'}q}}+|\beta'(t)|^q\Big)\dd t
\end{equation*}
and
\begin{align*}
{\calV}^T_z:=\{(\alpha,\beta)\in & W^{1,q}((-\sfrac T2,\sfrac T2);\R^{m+1})\colon
\notag\\ &0\leq \beta\leq 1,\, \beta(\pm\sfrac T2)=1,\, \alpha(-\sfrac T2)=0,\,\alpha(\sfrac T2)=z\}.
\end{align*}
\end{theorem}
More generally, discarding the assumption in \eqref{e:projectionproperty intro}
we provide separate lower and upper bounds for the sharp-interface limit of the functionals $\Functepspq$. The bounds have the same functional form, but the surface energy densities
are defined differently
(see {discussion after} \eqref{eqdefGpsnu inf} and \eqref{eqdefGpsnu sup}). We do not know whether the latter coincide or not in general,
except under the assumption in \eqref{e:projectionproperty intro} (cf. Theorem~\ref{t:finale p}).}

The proof of the $\Gamma-\liminf$ of Theorem~\ref{t:finale p} follows the same strategy as that used to treat the linear
case \cite{ContiFocardiIurlano2022}. However, two different scales appear when one looks at the local behavior of a
sequence $1-v_\eps\to 0$ for which the energy is bounded, $\eps^{\sfrac12}\ll \eps^{\sfrac{1}{2p}}$, coming from the
competition of the first two terms of the energy
{(this will become clear in Section~\ref{s:phase field} below).} This prevents us from eliminating the truncation by $1$
(by truncating $1-v_\eps$ itself by $\eps^{\sfrac12}$) when looking at the blow-up around a jump point of $u$. We are
then led to define
{a surface energy density $\gpinf$ (cf. \eqref{eqdefGpsnu inf})} involving a truncation of the coefficient
$\eps^{q-1}f_p^q(v)$. Notice that in the linear case $p=1$ the two scales match and the truncation by $1$ can be
neglected in the blow-up around a jump point of $u$.

{It suffices to prove the $\Gamma-\limsup$ inequality for functions $u\in SBV\cap L^\infty(\Omega;\R^m)$ for
which there exists a locally finite decomposition of $\R^n$ in simplexes, such that $u$ is affine in the interior of
each of them}, thanks to the density result \cite[Corollary~2.3]{CFI23} and to the relaxation Theorem~\ref{relaxation}.
The rest of the proof is an explicit construction obtained by properly rescaling and gluing together the optimal profile of
{a surface energy density $\gpsup(z,\nu)$ (with no truncation involved, {at variance with} $\gpinf$, cf.
\eqref{eqdefGpsnu sup}) with the affine function defining $u$ in each simplex. In particular
{we show that}
$\gpinf=\gpsup$ in the case where the
projection property in \eqref{e:projectionproperty intro} holds.}

{Finally, the compactness of families $\calF_{\eps,p,q}(u_\eps,v_\eps)$ under a $L^1$ bound on $u_\eps$ and the
convergence of minimizers of functionals of the form
\[\calF_{\eps,p,q}(u,v)+\int_\Omega\Big(\eta_\eps\Psi(\nabla u)+|u-w|^r\Big)\dd x,\]
where $\sfrac{\eta_\eps}{\eps^{q-1}}\to0$ and $w\in L^r(\Omega;\R^m)$ with $r>1$, can be addressed in a standard way
thanks to {Corollary}~\ref{propuvjump}; see for example \cite[Section~6]{CFI23}.
}

The paper is organized as follows: in Section~\ref{s:lsc} we prove Propositions~\ref{p:lsc surface} and \ref{p:lsc
bulk}, and in Section~\ref{s:relaxation} {we prove} Theorem~\ref{relaxation}.
In Section~\ref{s:phase field} we introduce in detail the
phase-field model and prove Theorem~\ref{t:onedim intro}
(see Theorem~\ref{t:onedim}, and the more general results in 
Propositions~\ref{p:lb GSBV} and \ref{p:ubp}).

\subsection{Notation}
In the entire paper $\Omega\subset\R^n$ is bounded, open, Lipschitz,  $\calA(\Omega)$ denotes the family of open subsets
of $\Omega$,
{$\calB(\Omega)$ denotes the family of Borel subsets
of $\Omega$,}
and $|\cdot|$ denotes the Euclidean norm, $|{\xi}|^2:=\sum_{ij}{\xi}_{ij}^2=\Tr\big({\xi}^T{\xi}\big)$ for
${\xi}\in\R^{m\times n}$.

{We will set $Q_1:=(-\frac12,\frac12)^n$ and we will denote by {$Q^\nu$,} $\nu\in S^{n-1}$, a unit cube with one face orthogonal to $\nu$ centered in the origin; for $t>0$ we set $Q^\nu_t:=tQ^\nu$.}

We use standard notation for Sobolev and $BV$ functions, referring to \cite{AFP} when needed.

{We shall denote by $\calF_{\eps,p,q}(u,v;A)$ the functional with integration restricted to $A$; if $A=\Omega$, the dependence on the set of integration will be dropped.} This
convention will be adopted for all the functionals introduced in what follows.

{
Many arguments below are based on a truncation procedure. We recall briefly the definition of the truncation operator
acting on vector-valued functions and its main properties.
We fix a sequence $(a_k)_k\subset(0,\infty)$ such that $2a_k<a_{k+1}$, $a_k\uparrow\infty$, and such that there are
functions $\mathcal{T}_k\in {C^1_c}(\R^m;\R^m)$ satisfying
\begin{equation}\label{e:Tk}
\mathcal{T}_k(z):=\begin{cases}
z, & \text{ if }|z|\leq a_k,\cr
0, & \text{ if } |z|\geq a_{k+1}
\end{cases}
\end{equation}
and {$\Lip(\mathcal T_k)\le 1$.}

{Let $u\in (GSBV(\Omega))^m$ obey
	$\nabla u\in L^1(\Omega;\R^{m\times n})$ and
	$\int_{J_u}g_0(|[u]|)\dH^{n-1}<\infty$ with $g_0$ as in \eqref{e:g0 superlinear}.} Setting $u_k:=\mathcal{T}_k(u){\in SBV(\Omega;\R^m)}$, we have that $u_k$ converges to $u$ in {measure},
$\nabla u_k=\nabla u$ $\calL^n$-a.e. in $\Omega_k:=\{x\in\Omega:\,|u(x)|\leq a_k\}$, $J_{u_k}\subseteq J_u$,
$\nu_{u_k}=\nu_u$, {$|[u_k]|\le |[u]|$} $\calH^{n-1}$-a.e. in $J_{u_k}$, {and $u^\pm_k(x)=u^\pm(x)$  $\calH^{n-1}$-a.e. in $J_{u_k}\cap\Omega_k$.}
{Since}
\begin{equation}\label{eqjuinfty}
\calH^{n-1}(\{x\in J_u:\,|u^\pm(x)|=\infty\})=0
\end{equation}
(see \cite[Proposition~2.12, Remark~2.13]{AlFoc}), we get {$\chi_{J_{u_k}}\to \chi_{J_u}$} and
$u_k^\pm\to u^\pm$
$\calH^{n-1}$-a.e. in $J_u$.

Following De Giorgi's averaging/slicing procedure on the codomain, the family $\mathcal{T}_k$ will be used in several
instances along the paper to obtain, from a sequence converging in {measure} to a limit belonging to $L^\infty$, a sequence with the same $L^1$ limit which is, in addition, equi-bounded in $L^\infty$. Moreover, this substitution can be done up
to paying an energy error which can be made arbitrarily small.
}

}

 \section{Lower semicontinuity}\label{s:lsc}

The aim of this section is to prove a lower semicontinuity result for free-discontinuity energies
with autonomous bulk densities superlinear for large gradients, and translation-independent 
surface densities superlinear for small jump amplitudes.

More precisely, define $H_g,\,H_\Psi: L^1(\Omega;\R^m)\times\calB(\Omega)\to[0,\infty]$ by
\begin{equation}
 H_g(u;A):=
 \int_{A\cap J_u}
 g([u],\nu_u) \dH^{n-1}
\end{equation}
and 
\begin{equation}
 H_\Psi(u;A):=
 \int_{A}
 \Psi(\nabla u) \dx
\end{equation}
for $u\in {(GSBV(\Omega))^m}$, and set both functionals equal to $\infty$ otherwise in $L^1$. 
If $A=\Omega$, we drop the dependence on the
set in
$H_g$ and $H_\Psi$.

Our aim is to prove the following: 
If $u_j\to u$ {strongly in $L^1$}, $u_j$, $u\in
{(GSBV(\Omega))^m}$,
$g$ and $\Psi$ superlinear and regular enough (see below for the precise hypotheses), then
\begin{equation}
 H_g(u)+H_\Psi(u)\le\liminf_{j\to\infty} (H_g(u_j)+H_\Psi(u_j)).
\end{equation}
Necessary conditions for $L^1$-lower semicontinuity require $g$ to be $BV$-elliptic and 
$\Psi$ to be quasiconvex (cf. \cite[Theorems~5.14 and 5.26]{AFP}). We record the two definitions in what follows for the sake of
convenience.

A Borel-measurable
function $g:\R^m\times S^{n-1}\to[0,\infty)$ is said to be $BV$-elliptic (see \cite[Definition 5.13]{AFP}) if for
every $z\in \R^m$, $\nu\in S^{n-1}$, and every $u\in SBV(Q^\nu;\R^m)$
{piecewise constant, with}
$\{u\neq z\chi_{\{x\cdot \nu>0\}}\}\subset\subset Q^\nu$
one has
\begin{equation}\label{e:BV ellipticity}
 g(z,\nu)\le
 \int_{Q^\nu\cap J_u}
 g([u],\nu_u) \dH^{n-1}.
\end{equation}
We recall that a function $u\in SBV(\Omega;\R^m)$ is called piecewise constant if 
{it is constant on each element of a Caccioppoli partition of $\Omega$, in the sense that}
there are countably many sets of finite perimeter $E_j$ such that
$\sum_j\calH^{n-1}{(\Omega\cap \partial^*E_j)}<\infty$, $\calL^n(\Omega\setminus \cup_j E_j)=0$,
and $u$ is constant on each of them. Equivalently,
$\calH^{n-1}(J_u)<\infty$ and $\nabla u=0$ {$\calL^n$-a.e.}, see \cite[Theorem~4.23]{AFP}.
{We shall show below (see Lemma~\ref{lemmaBVellipiinf}) that there are equivalent variants of {the definition of $BV$-ellipticity}, using classes of test functions which are either smaller (constant on finitely many
polyhedra) or larger ($SBV$ with $\nabla u=0$ {$\calL^n$-a.e.}).}

A locally bounded {Borel-measurable} function $\Psi:\R^{m\times n}\to{\R}$ is said to be quasiconvex
(see \cite[Definition 5.25]{AFP}) if for every $\xi\in \R^{m\times n}$, and every
$\varphi\in C_c^\infty(Q_1;\R^m)$ one has
\begin{equation}\label{e:quasiconvexity}
 \Psi(\xi)\le \int_{Q_1} \Psi(\xi+\nabla\varphi) \dx.
\end{equation}
In case of superlinear growth for $g$ and $\Psi$, the problem decouples,
and one can prove two corresponding separate lower semicontinuity inequalities {(cf.
\cite[Section~5.4]{AFP})}.
Therefore, in the next two sections we will address separately semicontinuity for surface integrals and for bulk
integrals, respectively, specifying explicitly the precise set of assumptions on the corresponding energy densities, and
finally deduce semicontinuity for free-discontinuity energies in Theorem~\ref{t:lsc} at the end of the section.

\subsection{Lower semicontinuity for surface energies}

In this section we prove Proposition~\ref{p:lsc surface}. To this aim,
we first adapt to our setting an approximation argument with piecewise constant functions introduced in \cite[Theorem
3.3]{Ambrosio1990} {for the case $\calH^{n-1}(J_u)<\infty$}.
For the following result to hold, $g$ does not need to be $BV$-elliptic, {but the quantitative superlinear growth in \eqref{e:g0 superlinear} is important to treat the part where the jump amplitude is small}.
\begin{proposition}\label{approximation}
Assume that $g_0$ and $g$ are as in \eqref{e:g consistent}-\eqref{e:g0 superlinear}.
Let $u\in SBV(\Omega;\R^m)$, $\eps>0$, $\delta>4\eps$, and $\eta\in(0,\eps)$.
Then there is a piecewise constant  $u_\eps\in SBV(\Omega;\R^m)$ such that  $\|u-u_\eps\|_{L^\infty(\Omega)}\le\eps$,
$|Du_\eps|(\Omega)\le {\sqrt m}|Du|(\Omega)$,
 \begin{equation}\label{e:stima energia approx cost a tratti}
\begin{split}
H_g(u_\eps)\le &\left(1+C \left(\frac\eps\delta\right)^\gamma
+ C \left(\frac \eta\eps\right)^{1-\gamma}\right)H_g(u)
\\ & +  C \left(\frac\eps\eta\right)^\gamma H_g(u,
\{|[u]|\in [\eta,\delta)\}) + C \eps^{\gamma-1} \|\nabla u\|_{L^1(\Omega)}.
 \end{split}\end{equation}
 \end{proposition}
\begin{proof}
{\bfseries Step~1. Construction.}
The construction is done componentwise. Assume for a moment that $m=1$.
For $\eps>0$ and $\rho\in[0,1)$ we set
  \begin{equation}
  u_{\eps,\rho}(x):=\eps \left\lfloor \frac{u(x)}{\eps}+\rho\right\rfloor
  =\sum_{{k\in\N\setminus\{0\}}}\eps \chi_{\{u\ge(k-\rho)\eps\}}
  {-\sum_{k\in\N}\eps \chi_{\{u<(-k-\rho)\eps\}}}
  \end{equation}
  {where we used the fact that for $t\in\R$
  \begin{equation*}
  \lfloor t\rfloor :=
  \max\{k\in\Z: k\le t\}
  = \#\{k\ge 1: k\le t\}-\#\{k\le 0: t<k\}.
  \end{equation*}}
  For $\calL^1$-almost all $\rho\in(0,1)$, all functions $\chi_{\{u\ge(k-\rho)\eps\}}$ and
$\chi_{\{u<(-k-\rho)\eps\}}$  are in $BV(\Omega)$, with pure-jump distributional gradient.
Therefore, for $\calL^1$-almost all $\rho$ we have $u_{\eps,\rho}\in {GBV}(\Omega)$ with 
$\nabla u_{\eps,\rho}=0$ $\calL^n$-a.e. on $\Omega$. It is also clear that 
$|u_{\eps,\rho}-u|<\eps$ pointwise.

We estimate, {using  $D\chi_{\{u<(-k-\rho)\eps\}}=
-D\chi_{\{u\ge (-k-\rho)\eps\}}$,}
\begin{equation}
 |Du_{\eps,\rho}|(\Omega)\le \eps \sum_{{k\in\Z}}
|D\chi_{\{u\ge (k-\rho)\eps\}}|(\Omega)
\end{equation}
so that, averaging over $\rho$,
\begin{align*}
 \int_{[0,1)} |Du_{\eps,\rho}|(\Omega)\dd\rho &\le \eps \sum_{{k\in\Z}}
\int_{[0,1)} |D\chi_{\{u\ge (k-\rho)\eps\}}|(\Omega)\dd\rho\notag\\
&=\int_{\R}|D\chi_{\{u\ge t\}}|(\Omega)\dt=|Du|(\Omega).
\end{align*}
We can therefore pick $\rho{\in[0,1)}$ such that
$|Du_{\eps,\rho}|(\Omega)\le |Du|(\Omega)$.

We repeat this for each component, using $\eps/\sqrt m$ in place of $\eps$, and obtain
$\rho\in[0,1)^m$ such that the
resulting function $u_\eps$ obeys
\begin{equation}\label{e:stime veps}
 u_\eps\in BV(\Omega;\R^m),\, \hskip2mm
 \nabla u_\eps=0 {\text{ $\calL^{n}$-a.e. on $\Omega$}},\,  \hskip2mm
 \|u-u_\eps\|_{L^\infty(\Omega)}\le \eps,\,
\end{equation}
and such that each component $i\in\{1,\dots, m\}$ obeys
\begin{equation}\label{e:stime veps2}
 |Du_\eps^i|(\Omega)\le |Du^i|(\Omega)
\end{equation}
{which in particular implies
\begin{equation}\label{e:stimevepsdu}
 |Du_\eps|(\Omega)\le \sqrt m |Du|(\Omega).
\end{equation}}

{\bfseries Step~2. Estimates.}
From the pointwise bound {in \eqref{e:stime veps}} we obtain
\begin{equation}\label{eqjujveps}
 |[u]-[u_\eps]|\le 2\eps \hskip4mm
 \text{ $\calH^{n-1}$-a.e. on $ J_u\cup J_{u_\eps}$},
\end{equation}
where as usual we set $[u]=0$ on $\Omega\setminus J_u$, and the same for $u_\eps$.
Necessarily, $\nu_u=\pm \nu_{u_\eps}$
$\calH^{n-1}$-a.e. on the set where at least one of the jumps in \eqref{eqjujveps} has
modulus larger than $2\eps$.
For $s>0$ we define
\begin{equation}
 A_s:=\{x\in J_u: |[u](x)|\ge s\}.
\end{equation}
In the following we consider separately the four sets
\begin{equation}
A_\delta, \hskip3mm A_{\eta}\setminus A_\delta, \hskip3mm
J_u\setminus A_{\eta}, \hskip3mm\text{ and }\hskip3mm\Omega\setminus J_u\,.
\end{equation}
As $\delta>2\eps$, by \eqref{eqjujveps} one has { $\calH^{n-1}(A_\delta\setminus J_{u_\eps})=0$.}
At the same time, by monotonicity of $g_0$
and $g_0(|z|)\le C g(z,\nu)$
we have
\begin{equation}\label{eqestAdelta}
\calH^{n-1}(A_\delta)
g_0(\delta)
 \le
 C  H_g(u).
\end{equation}
Monotonicity and subadditivity of $g_0$ also imply
\begin{equation}\label{e:stima Adelta}
 \int_{A_\delta} g_0(|[u]-[u_\eps]|) \dH^{n-1}
 \le g_0(2\eps) \calH^{n-1}(A_\delta)
 \le C \frac{g_0(\eps)}{g_0(\delta)} H_g(u)
\end{equation}
and analogously
\begin{equation}
 |Du-Du_\eps|(A_\delta)\le 2\eps \calH^{n-1}(A_\delta)
 \le C \frac{\eps}{g_0(\delta)} H_g(u) .
\end{equation}
 Here and in what follows we denote by $C>0$ a constant independent of $\eps$, $\delta$, $\eta$ {and $u$}, which
may vary from line to line.

We next turn to the second set.
Since $g_0$ is nondecreasing, {similarly}
\begin{equation}\label{eqguva2ed2}
 \int_{A_{\eta}\setminus A_{\delta}} g_0(|{[u]-}[u_\eps]|) \dH^{n-1}
 \le g_0(2\eps) \calH^{n-1}(A_{\eta}\setminus A_{\delta})
 \le  C\frac{g_0(\eps)}{g_0(\eta)} H_g(u;A_\eta\setminus A_{\delta}).
\end{equation}

Finally, on the rest $|[u]|\leq\eta<\eps$, and therefore $|[u_\eps]|\le  3\eps$.
We have, using that $|[u_\eps]|$ is either 0 or at least $\eps/\sqrt{m}$,
\begin{equation}\label{eqg0omegamina2epsa0}
\begin{split}
 \int_{\Omega\setminus A_{\eta}} g_0(|[u_\eps]|) \dH^{n-1}
 \le&  g_0( 3\eps) \calH^{n-1}(J_{{u_\eps}}\setminus A_{\eta})
 \le  C\frac{g_0(\eps)}{\eps}
 |Du_\eps|(\Omega\setminus A_{\eta}).
\end{split}\end{equation}
The last term is critical, since the first factor diverges, and thus the second factor requires a careful estimate.
To this aim, using \eqref{e:stime veps2} {and \eqref{eqjujveps}} we estimate as follows
{for any component $i\in\{1,\dots, m\}$}:
\begin{equation}\begin{split}\label{eqg0omegamina2epsa}
 |Du_\eps^{i}|(\Omega\setminus A_{\eta})
 &=|Du_\eps^{i}|(\Omega)
 -|Du_\eps^{i}|(A_{\eta})\\
& \le |Du^{i}|(\Omega)-|Du^{i}|(A_{\eta})+|Du_\eps^{i}-Du^{i}|(A_{\eta})\\
& \leq \|\nabla u^{i}\|_{L^1(\Omega)} +
|Du^{i}|(J_u)-|Du^{i}|(A_{\eta})+2\eps\calH^{n-1}(A_{\eta})\\
& = \|\nabla u^{i}\|_{L^1(\Omega)} +
 |Du^{i}|(J_u\setminus A_{\eta})
 +2\eps\left[\calH^{n-1}(A_{\delta})
 +\calH^{n-1}(A_{\eta}\setminus A_\delta)\right].
 \end{split}
\end{equation}
Summing over $i$ leads to
\begin{equation}\begin{split}\label{eqg0omegamina2eps}
 |Du_\eps|(\Omega\setminus A_{\eta})\le&
m\|\nabla u\|_{L^1(\Omega)} +
 m|Du|(J_u\setminus A_{\eta})\\
& +2\eps m\left[\calH^{n-1}(A_{\delta})
 +\calH^{n-1}(A_{\eta}\setminus A_\delta)\right].
 \end{split}
\end{equation}
In order to estimate the contribution in
 $J_u\setminus A_{\eta}$ we
  define $\psi(s):=\min_{t\in(0,s]} g_0(t)/t$, and observe that
  subadditivity of $g_0$ implies that
\begin{equation}
\frac{g_0(2t)}{2t}\le \frac{2g_0(t)}{2t} =\frac{g_0(t)}{t}
 \hskip5mm
 \text{ for all }t\in (0,s].
\end{equation}
Therefore
$\psi({2s})\geq\min_{t\in[s,{2s}]} g_0(t)/t
\ge g_0(s)/(2s)$.
In  $J_u\setminus A_{\eta}$ we
 have $0<|[u]|<\eta$, so that {by subadditivity of $g_0$}
\begin{equation}\label{e:stima Du Ju meno Aeta}
\begin{split}
 |Du|(J_u\setminus A_{\eta})
&= \int_{J_u\setminus A_{\eta}}
 \frac{|[u]|}{g_0(|[u]|)}
 g_0(|[u]|)  \dH^{n-1}\\
& \le
 \frac{1}{\psi({2\eta})}
 \int_{J_u\setminus A_{\eta}}
 g_0(|[u]|)  \dH^{n-1}
 \le C
 \frac{\eta}{g_0({\eta})}
 H_g(u;\Omega\setminus A_{\eta}).
\end{split}\end{equation}
At the same time, $|[u]|\ge\eta$  in $A_{\eta}$, 
and therefore, since $g_0$ is nondecreasing,
\begin{equation}\label{e:stima salto Aeta meno Adelta}
 \calH^{n-1}(A_{\eta}\setminus A_\delta) \le \frac{1}{g_0(\eta)}
 \int_{A_{\eta}\setminus A_\delta} g_0(|[u]|) \dH^{n-1}
 \le \frac{C}{g_0(\eta)}
 H_g(u,A_{\eta}\setminus A_\delta).
\end{equation}
Finally, $A_\delta$ had been estimated in~\eqref{eqestAdelta}.
Inserting the estimates \eqref{e:stima Du Ju meno Aeta}, \eqref{e:stima salto Aeta meno Adelta}
and \eqref{eqestAdelta} in \eqref{eqg0omegamina2eps}, the last term in
\eqref{eqg0omegamina2epsa0} can be estimated by
\begin{align}\label{eqg0omegamina2epsb}
 \frac{g_0(\eps)}{\eps}&|D{u_\eps}|(\Omega\setminus A_{\eta})\notag\\
\le& C\frac{g_0(\eps)}{\eps} \left[ \|\nabla u\|_{L^1(\Omega)} +
 |Du|(J_u\setminus A_{\eta})
 +2\eps\calH^{n-1}(A_{\delta})
 +2\eps\calH^{n-1}(A_{\eta}\setminus A_\delta)\right]
\notag \\
 \le& C\frac{g_0(\eps)}{\eps} \|\nabla u\|_{L^1(\Omega)} + C
\frac{g_0(\eps)}{\eps}
 \frac{\eta }{g_0({\eta})}
 H_g(u;\Omega\setminus A_{\eta})\notag\\
& + C \frac{g_0(\eps)}{g_0(\delta)}
 H_g(u) +C\frac{g_0(\eps)}{g_0(\eta)}
 H_g(u;A_{\eta}\setminus A_\delta)\notag\\
\le& C \eps^{\gamma-1}\|\nabla u\|_{L^1(\Omega)} +
 C  \left(\frac\eta\eps\right)^{1-\gamma}H_g(u; \Omega\setminus A_{\eta})\notag\\
 &+C\left(\frac\eps\delta\right)^{\gamma}H_g(u)
 + C \left(\frac\eps\eta\right)^{\gamma}
  H_g(u;A_{\eta}\setminus A_\delta),
\end{align}
where in the last line we used $g_0(s)\sim s^\gamma$ as $s\to0$ {(cf. \eqref{e:g0 superlinear})}.
Collecting the estimates in \eqref{e:stima Adelta}, \eqref{eqguva2ed2}, \eqref{eqg0omegamina2epsa0}, and \eqref{eqg0omegamina2epsb} and recalling that
$A_\delta\subseteq A_\eta$ we finally get using {\eqref{e:g grwoth g0} and} \eqref{eqgg0},
\begin{equation}\begin{split}
H_g(u_\eps)\le & H_g(u)+
c\int_{A_\eta} g_0(|[u]-[u_\eps]|) \dH^{n-1}
+c \int_{\Omega\setminus A_\eta} g_0(|[u_\eps]|) \dH^{n-1}\\
\le & \left[1+C \left(\frac\eps\delta\right)^\gamma 
+ C \left(\frac\eta\eps\right)^{1-\gamma} 
\right] H_g(u)\\
&+C \eps^{\gamma-1}\|\nabla u\|_{L^1(\Omega)} +
 C \left(\frac\eps\eta\right)^{\gamma}
  H_g(u;A_\eta\setminus A_\delta).
\end{split}\end{equation}
The conclusion in \eqref{e:stima energia approx cost a tratti} then follows at once.
\end{proof}
{
We stress that the specific power-type behaviour of $g_0$ at $0$ (cf. \eqref{e:g0 superlinear}) has been
used only in the last inequality of \eqref{eqg0omegamina2epsb}, and it is exploited in what follows to control the prefactor
$\frac{g_0(\eps)}{\eps} \frac{\eta }{g_0({\eta})}$ {in the third line of \eqref{eqg0omegamina2epsb}}.

We will apply Proposition~\ref{approximation} in Step~2 of the proof of Proposition~\ref{p:lsc surface} below to a sequence $u_\eps$ with
vanishing $\|\nabla u_\eps\|_{L^1}$, with the choice $\eps=M_\eps\eta_\eps$,
$M_\eps$ diverging.
On the one hand, choosing $\eps$ suitably with respect to $\|\nabla u_\eps\|_{L^1}$,
the first summand in \eqref{eqg0omegamina2epsb} vanishes. The third summand can be treated easily, using that $\eps/\delta_\eps\to0$. Moreover, the diverging prefactor $\frac{g_0(\eps)}{g_0({\eta})}$ in the fourth summand
can be compensated by selecting appropriately the slice
$A_{\eta_\eps}\setminus A_{\delta_\eps}$.
On the other hand, the subadditivity of $g_0$ implies only that $\frac{g_0(\eps)}{\eps}
\frac{\eta_\eps}{g_0({\eta_\eps})}=\frac{g_0(M_\eps\eta_\eps)}{M_\eps g_0(\eta_\eps)}$ is bounded, so that the second summand becomes a nonvanishing error term in the estimate.
{The power-type behaviour of $g_0$ at the origin is needed exactly to make the latter term infinitesimal. Of course, more general behaviours are allowed for $g_0$.}

Note that in the case $g_0(t)\sim t|\ln t|$ the previous subadditivity estimate is sharp, and thus with Proposition~\ref{p:lsc surface} below one cannot establish the lower semicontinuity of energies having densities with almost linear growth. For instance, with our method we cannot address the Read and Shockley model for {low-angle} grain boundaries in polycrystals. {If one assumes isotropy, however, then in this situation} lower semicontinuity follows from either a slicing argument or the biconvexity of the integrand
(see \cite[{Theorem~2.1}]{Ambrosio}, \cite[Theorem~3.7]{Ambrosio1990}, see also for another proof \cite[Theorem~4.7]{AFP}).
}

By truncation, we can obtain a similar approximation result for functions in $GSBV$.
 \begin{corollary}\label{approximationGSBV}
  Assume that $g_0$ and $g$ are as in \eqref{e:g consistent}-\eqref{e:g0 superlinear}.
Let $u\in {(GSBV(\Omega))^m}$
{with $|\nabla u| \in L^1(\Omega)$,}
$\eps>0$, $\delta>4\eps$, and $\eta\in(0,\eps)$.
 Then there is $u_\eps\in SBV(\Omega;\R^m)$
{which takes finitely many values away from a $\calL^n$-null set} with
 \begin{equation}\label{eq:convmeas}
\calL^n(\{|u-u_\eps|>\eps\})<\eps
 \end{equation}
{and}
 \begin{equation}\label{e:stima energia approx cost a tratt2i}
 \begin{split}
 H_g(u_\eps)\le &{C\eps+}\left(1+C \left(\frac\eps\delta\right)^\gamma
 + C \left(\frac \eta\eps\right)^{1-\gamma}\right)H_g(u)
 \\ & + C \left(\frac\eps\eta\right)^\gamma H_g(u;
 \{|[u]|\in [\eta,\delta)\}) + C \eps^{\gamma-1} \|\nabla u\|_{L^1(\Omega)}.
  \end{split}\end{equation}
  {If $u\in L^1(\Omega;\R^m)$, then $\|u_\eps-u\|_{L^1(\Omega)}\le C\eps$.}
  \end{corollary}
The fact that the set $u_{{\varepsilon}}(\Omega\setminus N)$ is finite for a $\calL^n$-null set $N$ implies that $\nabla u_\eps=0$ $\calL^n$-a.e.
on $\Omega$ and that $u_\eps\in L^\infty(\Omega;\R^m)$.
\begin{proof}
It suffices to reduce to a function in $L^\infty$ and then to
apply Proposition~\ref{approximation}. We can assume $H_g(u)<\infty$.
Then, by \eqref{eqjuinfty}
\begin{equation}
\lim_{M\to\infty} \int_{\{|u^+|\ge M\}\cap J_{u}} g([u],\nu_{u}) {\dH^{n-1}}
=\lim_{M\to\infty} \int_{\{|u^-|\ge M\}\cap J_{u}} g([u],\nu_{u}) {\dH^{n-1}}
=0.
\end{equation}
Recalling the definition {of the sequence $a_k$ before} \eqref{e:Tk}, there is $k$ such that
\begin{equation}\label{eqcoda}
\calL^n(\{|u|>a_k\})
\leq \eps
\end{equation}
and
\begin{equation}\label{eqgaketaeps}
\int_{A_k}
g([u],\nu_{u}) {\dH^{n-1}}
\leq
\eta^\gamma\eps^{1-\gamma},
\end{equation}
where $A_k:=(\{|u^+|>a_{k}\}\cup\{|u^-|>a_{k}\})\cap J_{u}$.
If {$u\in L^1(\Omega;\R^m)$}, we can additionally require $\int_{\{|u|>a_k\}}|u|\dx<\eps$.

Let $\hat
u:=\mathcal{T}_{k}(u)$. Then $\hat u\in L^\infty\cap SBV(\Omega;\R^m)$ and
$\|\nabla \hat u\|_{L^1(\Omega)}\le\|\nabla u\|_{L^1(\Omega)}$.
We estimate, using that
(up to $\calH^{n-1}$-null sets)
$J_{\hat u}\subseteq J_u$,
$|[\hat u]|\le |[u]|$ on $A_k$,
\eqref{e:g grwoth g0},
the monotonicity of $g_0$,
and finally using \eqref{eqgaketaeps},
\begin{equation}\begin{split}
 H_g(\hat u;A_k)
 {=}& \int_{A_k}
 g([\hat u],\nu_{\hat u}) {\dH^{n-1}}
 \le  C \int_{A_k}
 g_0(|[\hat u]|) {\dH^{n-1}}\\
 \le &  C \int_{A_k}
 g_0(|[u]|) {\dH^{n-1}}
 \le  C\eta^\gamma\eps^{1-\gamma}\,.
\end{split}\end{equation}
Recalling that $\eta\le\eps$ {and $[\hat u]=[u]$ on $J_{\hat u}\setminus A_k$,}
\begin{equation}\begin{split}
 H_g(\hat u)
 \le H_g(u)+H_g(\hat u;A_k)
 \le H_g(u)+C\eps
\end{split}\end{equation}
and
\begin{equation}\begin{split}
H_g(\hat u;
 \{|[\hat u]|\in [\eta,\delta)\})
\le & H_g(u;
 \{|[u]|\in [\eta,\delta)\})
+H_g(\hat u;A_k)\\
 \le &
  H_g(u;
 \{|[u]|\in [\eta,\delta)\})+C\eta^\gamma\eps^{1-\gamma}\,.
\end{split}\end{equation}
	The estimates \eqref{eq:convmeas} and \eqref{e:stima energia approx cost a tratt2i} follow using
Proposition~\ref{approximation} for $\hat u$ and recalling that $\eta\le\eps$. Indeed, by $\|\hat u-u_\eps\|_{L^\infty(\Omega)}\leq \eps$ and \eqref{eqcoda} we get
	\[\calL^n(\{|u-u_\varepsilon|>\eps\})\leq \calL^n(\{|\hat u-u|>0\})\leq \calL^n(\{|u|>a_k\})\leq \eps.
	\]
The $L^1$ estimate follows from
\begin{equation}\begin{split}
\|u-u_\eps\|_{L^1(\Omega)}
\le& \|\hat u-u_\eps\|_{L^1(\Omega)} + \|u-\hat u\|_{L^1(\Omega)}
\\
\le &\calL^n(\Omega) \eps +
\int_{\{|u|>a_k\}} (|u|+|\mathcal{T}_k(u)|)\dx\\
 \le & \calL^n(\Omega) \eps + 2 \int_{\{|u|>a_k\}} |u|\dx\le C\eps.\qedhere
 \end{split}
\end{equation}
\end{proof}

We are now ready to establish the claimed lower semicontinuity property for surface integrands.

\begin{proof}[Proof of Proposition~\ref{p:lsc surface}]
We divide the proof into several steps.

{\bf{Step~1. Blow-up.}} Without loss of generality we can assume that $(H_g(u_j))_j$ is bounded
and that the measures
\[\mu_g^j:=g([u_j],\nu_j)\calH^{n-1}\res J_{u_j}\]	
converge {weak-$*$} in $\calM(\Omega)$ to some positive finite Radon measure $\mu_g$. We shall show that
$H_g(u)\leq \mu_g(\Omega)$, and for this {it suffices} to show that
\[g([u](x_0),{\nu_u}(x_0))\leq \frac{\dd\mu_g}{\dH^{n-1}
{\res J_{u}}
}(x_0),\]
for $\calH^{n-1}$-almost every $x_0\in J_u$.
{By \cite[Theorem~4.34]{AFP} it suffices to}
prove the last inequality for points $x_0\in J_u$ such that
\[
\frac{\dd \mu_g}{\dd\calH^{n-1}\res J_u}(x_0)=\lim_{\rho\to0}\frac{\mu_g(Q^\nu_\rho(x_0))}{\calH^{n-1}(J_u\cap
Q^\nu_\rho(x_0))} \quad \text{exists finite}\,
\]
and
\[
\lim_{\rho\to 0}\frac{\mathcal{H}^{n-1}(J_u\cap Q^\nu_\rho(x_0))}{\rho^{n-1}}=1,
\]
 where $\nu:=\nu_u(x_0)$ and $Q^\nu_\rho(x_0){:=x_0+\rho Q^\nu}$ is the cube centred 
in $x_0$, with side length $\rho$, and one face orthogonal to $\nu$. We remark that such conditions define a set of full
measure in $J_u$.
First note that
\begin{align*}
&\frac{\dd \mu_g}{\dd\calH^{n-1}\res J_u}(x_0)=\lim_{\rho\to0}\frac{\mu_g(Q^\nu_\rho(x_0))}{\rho^{n-1}}
=\lim_{\substack{\rho\in I\\\rho\to 0}}\lim_{j\to\infty}\frac{\mu_g^j(Q^\nu_\rho(x_0))}{\rho^{n-1}}\,,
\end{align*} 
where we have used the fact that $\mu^j_g\weakto\mu_g$ weakly-$*$ in $\calM(\Omega)$ and we have set
$I:=\{\rho\in(0,\frac{2}{\sqrt n}\,\mathrm{dist}(x_0,\partial \Omega)):\,
{\mu_g}(\partial Q^\nu_\rho(x_0))=0\}$.
Therefore, a change of variables yields that
\begin{align}\label{blowupHg}
\frac{\dd {\mu_g}}{\dd\calH^{n-1}\res J_u}(x_0)
= \lim_{{\scriptsize \begin{array}{c}
		\rho\!\!\in\!\! I\\ \!\!\rho\!\!\to\!\!0\end{array}}}\lim_{j\to\infty}
{H_g(u_j^\rho;Q^\nu)}\,,
\end{align} 
where we have denoted by $u^\rho_j(y):=u_j(x_0+\rho y)$ the rescaling of $u_j$. 

Since $|\nabla u_j|$ is equiintegrable,
{and therefore in particular bounded in $L^1$},
by the Dunford-Pettis Theorem~\cite[Theorem~1.38]{AFP} it is weakly
precompact in $L^1(\Omega)$; hence there exists $h\in L^1(\Omega)$ such that a subsequence satisfies $|\nabla
u_j|\rightharpoonup h$ weakly in $L^1(\Omega)$ as $j\to\infty$.
Hence,
\begin{equation}\label{rescaledgrad}
\lim_{\rho\to0}\lim_{j\to\infty}\frac{1}{\rho^{n-1}}\int_{Q^\nu_\rho(x_0)}|\nabla u_j|\dx=
\lim_{\rho\to0}\frac{1}{\rho^{n-1}}\int_{Q^\nu_\rho(x_0)}h\,\dx=0,
\end{equation}
for $\calH^{n-1}$-almost every $x_0\in J_u$ by \cite[Theorem~2.56 {and equation (2.41)}]{AFP}. Note that by scaling one has
\begin{equation}\label{e:scaling nabla uj}
\lim_{\rho\to0}\lim_{j\to\infty}\frac{1}{\rho^{n-1}}\int_{Q^\nu_\rho(x_0)}|\nabla u_j|\dx=
\lim_{\rho\to0}\lim_{j\to\infty}\int_{Q^\nu}|\nabla u_j^\rho|\dy\,.
\end{equation}
Recalling that $u_j\to u$ in {measure} 
as $j\to \infty$, using \eqref{blowupHg},
\eqref{rescaledgrad}, and
\eqref{e:scaling nabla uj}, by diagonalization we can find  two
subsequences $\rho_k\to0$ and $j_k\to\infty$ such that as $k\to \infty$
\begin{eqnarray*}
u_{j_k}^{\rho_k}\to [u](x_0)\chi_{\{y\cdot\nu>0\}}+u^-(x_0) \quad \text{{in measure,}}
\\
\nabla u^{\rho_k}_{j_k}\to 0 \quad \text{in} \quad L^1(Q^\nu;\R^{m\times n}),\\
\frac{\dd \mu_g}{\dd\calH^{n-1}\res J_u}(x_0)= \lim_{k\to\infty}H_g(u_{j_k}^{\rho_k};Q^\nu)\,.
\end{eqnarray*}

\medskip

\noindent {\bf Step~2. Reduction to a piecewise constant sequence.}
By Step~1 it is enough to prove that for
$z\in \R^m$ and $\nu\in S^{n-1}$ we have
\[g(z,\nu)\leq \liminf_{j\to\infty}H_g(u_j;Q^\nu),\]
whenever $u_j\to u_0:= z\chi_{\{x\cdot\nu>0\}}$ in measure
and $\nabla u_j\to 0$ in $L^1(Q^\nu;\R^{m\times
n})$, as $j\to\infty$.
We can also assume that the sequence $H_g(u_j;Q^\nu)$
is bounded.

Consider any sequence $r_j\to0$ such that $ r_j^{\gamma-1} \|\nabla u_j\|_{L^1(Q^\nu)}\to0$, and then 
$M_j\to\infty$, $K_j\to\infty$ such that $r_jM_j^{2K_j+1}\to0$ and $M_j^\gamma K_j^{-1}\to0$.
Select  by averaging  $k_j\in \{1, \dots, K_j\}$ such that
\begin{equation}\label{choiceeps}
H_g(u_j; \{|[u_{{j}}]|\in [r_jM_j^{2k_j-1}, r_j M_j^{2k_j+1})\})
\le \frac{2}{K_j} H_g(u_j;Q^\nu).
\end{equation}
By {Corollary~\ref{approximationGSBV}}, applied for all $j\in\N$ with
{$\eps_j:=r_j M_j^{2k_j}$,}
$\eta_j:=\eps_j/M_j\to0$,
$\delta_j:=\eps_jM_j\to0$, we find
$v_j\in SBV(Q^\nu;\R^m)$ such that $\nabla v_j=0$ $\calL^n$-a.e. on $Q^\nu$, and $v_j$ takes finitely many values away from a $\calL^n$-null set,
\begin{equation}\label{eqconvmeasuve}
{{\calL^n}(\{|u_j-v_j|>\eps_j\})\le\eps_j,}
\end{equation}
\begin{equation}\label{eqHvvjujerw}
\begin{split}
H_g(v_j; {Q^\nu})\le &
{C \eps_j+}
\left(1+C \left(\frac{\eps_j}{\delta_j}\right)^\gamma
+ C \left(\frac {\eta_j}{\eps_j}\right)^{1-\gamma}\right)H_g(u_j; {Q^\nu})
\\ & + C \left(\frac{\eps_j}{\eta_j}\right)^\gamma H_g(u_j;
{Q^\nu\cap}\{|[u_j]|\in [\eta_j,\delta_j)\}) + C  \eps_j^{\gamma-1}\|\nabla u_j\|_{L^1(Q^\nu)}.
\end{split}\end{equation}
By definition $\frac{\eps_j}{\delta_j}=\frac{\eta_j}{\eps_j}=M_j^{-1}\to 0$ as $j\to\infty$,
and moreover, recalling the definition of $r_j$ and $k_j$ (cf. \eqref{choiceeps}),
\begin{align*}
{\liminf_{j\to\infty}}&\left(\left(\frac{\eps_j}{\eta_j}\right)^\gamma H_g(u_j;
{Q^\nu\cap}\{|[u_j]|\in [\eta_j,\delta_j)\}) +
{\frac{1}{\eps_j^{1-\gamma}}}\|\nabla u_j\|_{L^1{(Q^\nu)}}\right)\\
&\leq2\,{\limsup_{j\to\infty}} \frac{M_j^\gamma}{K_j}H_g(u_j;Q^\nu)
+{\limsup_{j\to\infty}\,} 
{\frac{1}{r_j^{1-\gamma}}}
\|\nabla u_j\|_{L^1{(Q^\nu)}}=0\,.
\end{align*}
In conclusion, {from \eqref{eqHvvjujerw}} we infer
\begin{equation}
\liminf_{j\to\infty} H_g(v_j;Q^\nu)\le 
\liminf_{j\to\infty} H_g(u_j;Q^\nu)\,.
\end{equation}
{From $u_j\to u_0$ in measure and \eqref{eqconvmeasuve} we deduce $v_j\to u_0$ in measure.}
{If $u_j\to u_0$ in $L^1$, then  $v_j\to u_0$ in
$L^1$.}

\medskip

{\bf Step~3. Conclusion.}
The rest of the proof is {similar to} \cite[Theorem~5.14]{AFP}.
{By Steps~1 and 2, it is sufficient to prove \eqref{lscsurf} for a sequence of piecewise constant functions, {each taking finitely many values}. Thus,} in order to apply the definition of $BV$-ellipticity of $g$, we need to
modify the boundary datum of $v_j$.

{Define $\phi:\R^{m}\to[0,\infty)$ by $\phi({z}):=g_0(|{z}|)$.
For any function
$w\in SBV(Q^\nu;\R^m)$ with $w(Q^\nu)$ a finite set, we have that  $\phi\circ w\in SBV\cap
L^\infty{(Q^\nu)}$, with $J_{\phi\circ w}\subseteq J_w$ (up to null sets) and
$\nabla (\phi\circ w)=0$ {$\calL^n$-a.e. on $Q^\nu$}.
Further, using subadditivity and monotonicity of $g_0$
\begin{equation}
 [\phi\circ w] = \phi(w^+)-\phi(w^-)
\le g_0(|w^-|+|[w]|)-g_0(|w^-|)
\le g_0(|[w]|)
\end{equation}
so that, repeating the computation with the two traces swapped,
\begin{equation}
|[\phi\circ w]| \le g_0(|[w]|).
\end{equation}
Therefore for every $j$ we have
\begin{equation}
 |D(\phi\circ v_j)|{(Q^\nu)}\le \int_{J_{v_j}} g_0(|[v_j]|)\dd\calH^{n-1}\le C H_g(v_j;Q^\nu)
\end{equation}
and we obtain that $\phi\circ v_j$ is precompact in ${L^1(Q^\nu)}$ (possibly after subtracting the average).
Since $g_0$ is continuous,
from the convergence in measure of $v_j$ to $u_0$, we obtain, possibly passing to a subsequence and dropping a null set,
pointwise convergence of $v_j$ to $u_0$ and of $\phi\circ v_j$ to $\phi\circ u_0$. Therefore
$\phi\circ v_j\to \phi\circ u_0$ in $L^1(Q^\nu)$.}
By
Fubini's theorem, up to subsequences,
for $\calL^1$-a.e. $t\in (0,1)$ there holds
\begin{equation}\label{boundary}
\lim_{j\to\infty}\int_{\partial Q_t^\nu}{|\phi\circ v_j-\phi\circ  u_0|}\dH^{n-1}=0.
\end{equation}
In addition, up to subsequences,
{by convergence in measure and another application of Fubini's theorem}
we may assume that
$v_j\to u_0$ $\calH^{n-1}\res \partial Q_t^\nu$-a.e. on $\partial Q_t^\nu$
and that $\calH^{n-1}(\partial Q_t^\nu\cap J_{v_j})=0$ for all $j$.

{Fix $M$ with $\|u_0\|_{L^\infty(Q^\nu)}<M$. By pointwise convergence and dominated convergence,
\begin{equation}
\lim_{j\to\infty}\int_{\partial Q_t^\nu}
 g_0(|v_j-u_0|)\chi_{\{|v_j|<M\}} \dH^{n-1}=0
\end{equation}
and
\begin{equation}
\lim_{j\to\infty}\int_{\partial Q_t^\nu}
 \chi_{\{|v_j|\ge M\}} \dH^{n-1}=0.
\end{equation}
At the same time,
using $g_0(|v_j-u_0|)\le
g_0(|u_0|)+ g_0(|v_j|)=
2g_0(|u_0|)+g_0(|v_j|)-g_0(|u_0|)$,
\begin{equation}\begin{split}
&\int_{\partial Q_t^\nu}
 g_0(|v_j-u_0|)\chi_{\{|v_j|\ge M\}} \dH^{n-1}\le\\
 &
\int_{\partial Q_t^\nu}
2 g_0(|u_0|)\chi_{\{|v_j|\ge M\}} \dH^{n-1}+
\int_{\partial Q_t^\nu}
|\phi\circ v_j-\phi\circ u_0| \dH^{n-1}.
\end{split}\end{equation}
Combining these estimates leads to
\begin{equation}\label{eqg0vju0pq}
\begin{split}
\lim_{j\to\infty} \int_{\partial Q_t^\nu}
 g_0(|v_j-u_0|) \dH^{n-1}=0
\end{split}\end{equation}
for $\calL^1$-almost every $t\in(0,1)$.
}

Moreover, by \cite[Lemma~5.15]{AFP}, for $\calL^1$-a.e. $t\in (0,1)$ the function
$w_j:=v_j\chi_{Q_t^\nu}+u_0\chi_{Q^\nu\setminus Q^\nu_t}$ satisfies $(w_j^+,w_j^-,\nu_w)=(v_j,u_0,\nu_t)$ for
$\calH^{n-1}$-a.e. $x\in\partial Q^\nu_t$, where $\nu_t$ denotes the inner normal vector to $Q^\nu_t$. Hence,
\begin{multline*}
\liminf_{j\to\infty} H_g(v_j;Q^\nu)\geq \liminf_{j\to\infty} H_g(v_j;Q^\nu_t)\\
\geq 
\liminf_{j\to\infty} \Big(H_g(w_j;Q^\nu)-\int_{\partial Q^\nu_t}g(v_j-u_0,\nu_t)\dH^{n-1}\Big)-H_g(u_0;Q^\nu\setminus
\overline{Q^\nu_t})\\
\geq g(z,\nu)-g(z,\nu)(1-t^{n-1}),
\end{multline*}
where in the last inequality we have used the $BV$-ellipticity of $g$ and
{\eqref{eqg0vju0pq}.}
Hence, as $t\to1^-$, we conclude from Step~2
\[
\liminf_{j\to\infty} H_g(u_j;Q^\nu)\geq\liminf_{j\to\infty} H_g(v_j;Q^\nu)\geq
g(z,\nu)=H_g(u_0;Q^\nu).\qedhere
\]
\end{proof}

\subsection{Lower semicontinuity for bulk energies}
The aim of this section is to prove Proposition~\ref{p:lsc bulk}. We first show that
nonnegative superlinear quasiconvex functions can be approximated by quasiconvex ones of linear growth.
In fact, for later purposes (cf. Section~\ref{s:liminf}) we {do not} assume quasiconvexity of $\Psi$,
and only suppose that
$\Psi:\R^{m\times n}\to[0,\infty)$ is continuous, and for some $q>1$ {satisfies} \eqref{e:Psi gc}, namely
\begin{equation*}
\Big(\frac1c |\xi|^q-c\Big)\vee0\le \Psi(\xi)\le c(|\xi|^q+1)
\hskip1cm\text{ for all }\xi\in\R^{m\times n}.
\end{equation*}
{We are then led to introduce the quasiconvex envelope
$h^\qc$ of a continuous function $h:\R^{m\times n}\to[0,\infty)$, which is}
defined as
\begin{equation}\label{e:qc envelope}
 h^\qc(\xi):=\sup\{\Phi(\xi):\, \Phi\leq  h,\, \Phi\textrm{ quasiconvex}\}\,,
\end{equation}
and characterized by (see \cite[Theorem~6.9]{Dacorogna})
\begin{equation}\label{eq:hqc}
 h^\qc(\xi){=}\inf \left\{\int_{(0,1)^n} h(\xi+\nabla \varphi) \dx: \varphi\in W^{1,\infty}_0((0,1)^n;\R^m) \right\}.
\end{equation}

\begin{proposition}\label{prophdeltaqc}
{Let $\Psi:\,\R^{m\times n}\to[0,\infty)$ be continuous and satisfy \eqref{e:Psi gc}}. For $\delta\in(0,1)$, $\ell>0$, and $p>1$, let  \begin{equation}\label{eqdefhdelta}
h_\delta(\xi):=\Psi(\xi)\wedge \frac{{\ell}}{(1-{\delta^{q'}})^{p-1}}\Psi^{\sfrac1q}(\xi),
\end{equation}
{where  as usual $q':=q/(q-1)$.}
Then
\begin{equation*}
\sup_{\delta\in(0,1)}h_\delta^\qc(\xi)=\lim_{\delta\uparrow 1}h_\delta^\qc(\xi)=\Psi^{\qc}(\xi)\,.     
    \end{equation*}
    \end{proposition}
The proof is based on the following variant of Zhang's truncation result
\cite{Zhang1992} (see also \cite[Sect.~6.6.2]{EvansGariepy}),
the version on a bounded domain stated below appears, for instance, in \cite[Proposition~A.1]{FrieseckeJamesMuller02}.
The treatment of boundary data is discussed for example in
\cite[Lemma~4.1]{DolzmannHungerbuehlerMueller2000}.
\begin{lemma}\label{lemmatrunc}
     Let $\Omega\subset\R^n$ be a bounded Lipschitz set, $s\in[1,\infty)$. There is $c_*=c_*(s,\Omega,m)>0$ such that
     for any $u\in W^{1,s}(\Omega;\R^m)$ and $\lambda>0$ there is $w_\lambda\in W^{1,\infty}(\Omega;\R^m)$ such that
     $\|\nabla w_\lambda\|_{L^\infty(\Omega)}\le c_*\lambda$ and
\[
\lambda^s\calL^n(\{w_\lambda\ne u\}) \le c_* \int_{\{|\nabla u|>\lambda\}} |\nabla u|^s \dx\,.
\]
If, additionally, $u\in W^{1,s}_0(\Omega;\R^m)$, then $w_\lambda\in W^{1,\infty}_0(\Omega;\R^m)$.
    \end{lemma}
\begin{proof}[Proof of Proposition~\ref{prophdeltaqc}]
The inequality $\sup_{\delta\in(0,1)}h_\delta^\qc(\xi)\leq\Psi^{\qc}(\xi)$ follows immediately from 
$h_\delta\leq\Psi$.

To prove the reverse inequality, we discuss the case $\ell=1$ for the sake of notational simplicity. Define $A_\delta:=(1-\delta^{q'})^{\frac {(1-p)q}{q-1}}$, so that
$A_\delta\to\infty$ for $\delta\uparrow 1$ and
\begin{equation}\label{eqdefhfrompsianda}
 {h_\delta}(\eta)=\Psi(\eta)\wedge A_\delta^{1-\sfrac 1q}\Psi^{\sfrac 1q}(\eta)
 =\Psi^{\sfrac1q}(\eta) \left[\Psi(\eta)\wedge A_\delta\right]^{1-\sfrac1q}.
\end{equation}
We fix
$\xi\in\R^{m\times n}$ and, for any $\delta\in(0,1)$, choose
$\varphi_\delta\in W^{1,\infty}_0(Q_1;\R^m)$ such that
\begin{align}\label{e:stima hdeltaqc}
\int_{Q_1}h_\delta(\xi+\nabla \varphi_\delta)\dx\le
h_\delta^{\qc}(\xi)+1-\delta.
\end{align}
Let $N>1$, chosen below, and set
\begin{equation}\label{e:KMdelta}
 K_{N,\delta}:=\max\{k\in\N: N^{k+2}\le {A_\delta^{1/q}} \}.
\end{equation}
Obviously $\lim_{\delta\uparrow 1} K_{N,\delta}=\infty$.
We choose $k_\delta\in\N\cap[1,K_{N,\delta}]$ such that
\begin{equation}\label{eqchoicekdel}
 \int_{\{ N^{k_\delta}\le |\nabla\varphi_\delta| < N^{k_\delta+1}\} }
 h_\delta(\xi+\nabla\varphi_\delta) \dx
 \le \frac{h_\delta^\qc(\xi)+1}{K_{N,\delta}}
 \le \frac{\Psi^\qc(\xi)+1}{K_{N,\delta}}.
\end{equation}
By Lemma~\ref{lemmatrunc} applied with $s=1$
{and $\lambda:=N^{k_\delta}$} there is
$\theta_\delta\in W^{1,\infty}_0(Q_1;\R^m)$ such that
$|\nabla \theta_\delta|\le c_* {\lambda}$ and
\begin{equation*}
 \calL^n(\{ \theta_\delta\ne\varphi_\delta\}) 
 \le \frac{c_*}{{\lambda}} \int_{\{|\nabla\varphi_\delta|> {\lambda}\}} |\nabla\varphi_\delta| \dx.
\end{equation*}
The constant $c_*$ depends only on $n$ and $m$.

{From the definition of $h_\delta$ we deduce that
there are $\delta_0\in (0,1)$, $N_0\in\N$, $c'>0$, all depending only on 
$\xi$, $n$, $m$, and the parameters $c$ and $q$ from
\eqref{e:Psi gc}, such that for any $\delta\in(\delta_0,1)$ and $N\ge N_0$ the following holds:
\begin{equation}\label{eqetapsihlower}
 \text{if $|\eta|\le N\lambda$ then 
 $h_\delta(\xi+\eta)=\Psi(\xi+\eta)$},
\end{equation}
\begin{equation}\label{eqetapsihupper}
 \text{if $\lambda \le |\eta|\le N\lambda$ then 
 $h_\delta(\xi+\eta)\ge c' |\eta| \lambda^{q-1}$,}
\end{equation}
and
\begin{equation}\label{eqetapsihupper2}
 \text{if $|\eta|\ge N\lambda$ then 
 $h_\delta(\xi+\eta)\ge c' |\eta| (N\lambda)^{q-1}$.}
\end{equation}
To prove the first one, by \eqref{eqdefhfrompsianda} it suffices to show that in the relevant regime one has
$\Psi(\xi+\eta)\le A_\delta$. We compute from \eqref{e:Psi gc}
\begin{equation*}
\begin{split} \Psi(\xi+\eta)
 \le &c (|\xi+\eta|^q+1)
 \le c  ({2^{q-1}} |\xi|^q+ {2^{q-1}}|\eta|^q+1)\\
 \le &c {2^{q-1}} (|\xi|^q+1)+  c {2^{q-1}} N^{q(K_{N,\delta}+1)}.
 \end{split}
  \end{equation*}
For a suitable choice of $\delta_0$, the first term is bounded by $A_\delta/2$. By \eqref{e:KMdelta}, the second one is bounded by
$c {2^{q-1}} A_\delta/N^q$, so that for a suitable choice of $N_0$ this is also bounded by $A_\delta/2$.
This proves \eqref{eqetapsihlower}.}

{We next observe that if $|\eta|\ge\lambda$ then 
for a suitable choice of $N_0$ and $\tilde c>0$, depending on $|\xi|$, $q$, and $c$, we have
$|\xi+\eta|\ge|\eta|/2$ and
 by \eqref{e:Psi gc}
\begin{equation}\label{eqpsixietaxi}
 \Psi(\xi+\eta)\ge \frac1c |\xi+\eta|^q-c
 \ge \frac1{\tilde c}|\eta|^q .
\end{equation}
Together with \eqref{eqetapsihlower},
this proves \eqref{eqetapsihupper}.}

{Finally, we turn to \eqref{eqetapsihupper2}.
If $\Psi(\xi+\eta)\le A_\delta$, then it follows from \eqref{eqpsixietaxi}.
In order to obtain the other case, using again \eqref{e:KMdelta} and then
\eqref{eqpsixietaxi},
\begin{equation*}
A_\delta^{1-\sfrac 1q} \Psi^{\sfrac 1q}(\xi+\eta) \ge 
N^{(k_\delta+2)(q-1)}
\frac1{\tilde c^{1/q}} |\eta| 
\ge (N\lambda)^{q-1}
\frac1{\tilde c^{1/q}} |\eta|.
\end{equation*}
This concludes the proof of \eqref{eqetapsihupper2}.}

Recalling the choice of $\theta_\delta$,
{and using \eqref{eqetapsihlower} and $N_0\ge c_*$ to obtain
$h_\delta(\xi+\nabla \theta_\delta)
=\Psi(\xi+\nabla \theta_\delta)$,}
\begin{equation}\label{eqpsiqc}
\begin{split}
 \Psi^\qc(\xi)\le& \int_{Q_1} \Psi(\xi+\nabla \theta_\delta) \dx =
 \int_{\{\theta_\delta=\varphi_\delta\}} \Psi(\xi+\nabla \theta_\delta) \dx +
 \int_{\{\theta_\delta\ne\varphi_\delta\}} \Psi(\xi+\nabla \theta_\delta) \dx
 \\
 \le&
 \int_{Q_1} h_\delta(\xi+\nabla \varphi_\delta) \dx
 +c {({c_*^q}\lambda^q+1)}\calL^n(\{\theta_\delta\ne\varphi_\delta\})\\
 \le & h^\qc_\delta(\xi)+(1-\delta) + 
 {C\lambda^{q-1}} \int_{\{|\nabla\varphi_\delta|> \lambda\}} |\nabla\varphi_\delta| \dx.
\end{split}
\end{equation}
{We split the last term into two parts. The integral on $\{ \lambda<
  |\nabla\varphi_\delta|< N\lambda\}$
  is estimated using \eqref{eqchoicekdel} and \eqref{eqetapsihupper}, which lead to
\begin{equation}
\lambda^{q-1} \int_{\{ \lambda< |\nabla\varphi_\delta| < N\lambda\} }
 |\nabla\varphi_\delta| \dx\le C\frac{\Psi^\qc(\xi)+1}{K_{N,\delta}}.
\end{equation}
  For the contribution on the set 
 $\{
  |\nabla\varphi_\delta|\ge N\lambda\}$
  we instead  use \eqref{eqetapsihupper2} and 
  \eqref{e:stima hdeltaqc}, 
  to obtain
  \begin{equation}
\lambda^{q-1} \int_{\{ N\lambda\le |\nabla\varphi_\delta| \} }
 |\nabla\varphi_\delta| \dx\le C N^{1-q} (\Psi^\qc(\xi)+1).
\end{equation}  }

Inserting the last two inequalities into \eqref{eqpsiqc} leads to
\begin{equation*}
 \Psi^\qc(\xi)\le h^\qc_\delta(\xi)+(1-\delta) + 
 {C} \left(\frac{1}{N^{q-1}}+\frac{1}{K_{{N},\delta}}\right)
(\Psi^\qc(\xi)+1).
\end{equation*}
Taking first $\delta\to1$ and then $N\to\infty$ gives the result
(recall that $\delta\mapsto h_\delta(\xi)$ is nondecreasing).
    \end{proof}

We are now ready to prove the claimed lower semicontinuity property for bulk integrands.
\begin{proof}[Proof of Proposition~\ref{p:lsc bulk}]
We divide the proof into several steps.

{\bf Step~1. Reduction to coercive integrands.}
Assume that \eqref{e:lsc volume} holds for coercive integrands $\Psi$, namely satisfying for every
$\xi\in\R^{m\times n}$
\begin{equation}\label{e:Psi coercive}
\frac1c|\xi|^q-c\le \Psi(\xi)\leq c(|\xi|^q+1)\,.
\end{equation}
We show next how to deduce lower semicontinuity for integrands satisfying more generally \eqref{e:Psi gc gen}.
Recalling ($\Psi$-a''), assumption ($\Psi$-{c}) can be equivalently stated as
follows: for every $\eps>0$ there is $k_\eps{\ge1}$ such that
\begin{equation}\label{e:equiintegrability}
\sup_j\int_{\{\Psi(\nabla u_j)<-{k_\eps}\}}|\Psi(\nabla u_j)|\dx<\eps\,.
\end{equation}
Clearly, we may suppose that $k_\eps\to\infty$ as $\eps\to0$.

Let $q\geq 1$, and define $\Psi_\eps(\xi):=\max\{\Psi(\xi),-k_\eps\}+\eps|\xi|^q$. Then,
\begin{align*}
&\left|\int_\Omega \Psi_\eps(\nabla u_j)\dx-\int_\Omega \Psi(\nabla u_j)\dx\right|\\&\leq
\int_{{\{\Psi(\nabla u_j)<-k_\eps\}}} |\Psi(\nabla u_j)|\dx
+k_\eps\calL^n(\{\Psi(\nabla u_j)<-k_\eps\})+\eps\int_\Omega |\nabla u_j|^q\dx\\
&\leq C\eps
\end{align*}
for some $C> 0$, thanks to assumption {($\Psi$-a'')} and \eqref{e:equiintegrability}.
Clearly, $\Psi_\eps$ is coercive {and quasiconvex}, thus the corresponding integrand is lower
semicontinuous so that
\begin{align*}
\liminf_{j\to\infty}\int_\Omega \Psi(\nabla u_j)\dx\geq
\liminf_{j\to\infty}\int_\Omega \Psi_\eps(\nabla u_j)\dx-C\eps\geq \int_\Omega \Psi_\eps(\nabla u)\dx-C\eps\,,
\end{align*}
and the conclusion follows
{since $\Psi_\eps\ge \Psi$}
by letting $\eps\to 0$.
\medskip

We are thus left with proving inequality \eqref{e:lsc volume} for integrands satisfying 
\eqref{e:Psi coercive}.
{Since \eqref{e:Psi coercive} implies that $\Psi$ is bounded from below,
adding a constant, one easily reduces to the case of nonnegative integrands.}

{\bf Step~2. 
Reduction to a truncated sequence converging in $L^1$.}
Let $\mathcal T_k$ be the truncation defined in \eqref{e:Tk}. Having fixed $M\in\N$, by taking into account assumption
($\Psi$-a'') and averaging, we choose, for every $j$,
an integer
$k_j\in\{M+1,\dots, 2M\}$ such that
\begin{equation}\label{e:average}
\int_{\{a_{k_j}<|u_j|<a_{k_j+1}\}} |\nabla u_j|^q \dx
\leq \frac CM,
\end{equation}
for some $C>0$, which implies that $\hat u_j:=\mathcal{T}_{k_j}(u_j){\in SBV(\Omega;\R^m)}$ obeys
\begin{equation}\label{eqthatkwcm}
\begin{split}
&\int_{\Omega} \Psi(\nabla \hat u_j) \dx \le 
\int_{\Omega} \Psi(\nabla u_j) \dx+\frac CM
+C\calL^n(\{|u_j|\ge a_{M}\}).
\end{split}
\end{equation}
Indeed, we have
\begin{align*}
\int_{\Omega}& \Psi(\nabla \hat u_j) \dx \leq\int_{\{|u_j|\leq a_{k_j}\}} \Psi(\nabla u_j) \dx\\
&+\int_{\{a_{k_j}<|u_j|< a_{k_j+1}\}} \Psi(\nabla \hat u_j)  \dx +\Psi(0)\calL^n(\{|u_j|\ge a_{k_j+1}\})\\
&\leq\int_{\Omega} \Psi(\nabla u_j) \dx+C\int_{\{a_{k_j}<|u_j|< a_{k_j+1}\}} |\nabla u_j|^{q}\dx +C\calL^n(\{|u_j|\ge
a_{k_j}\})\,,
\end{align*}
where in the last inequality we have used that {$\mathrm{Lip}(\mathcal{T}_{k_j})\leq1$}
together with \eqref{e:Psi coercive}.
The inequality in \eqref{eqthatkwcm} then follows from \eqref{e:average}. 

Passing to a further subsequence we can assume $k_j=k$ independent of $j$ (but still depending on $M$). Notice that
$\hat u_j$ and $\hat u:=\mathcal{T}_{k}(u)$ satisfy ($\Psi$-a''), ($\Psi$-b''), and ($\Psi$-c) by definition of
$\mathcal T_k$. {Moreover, up to subsequences, $\hat u_j$ converges pointwise to $\hat u$ and hence in $L^1$ by dominated convergence.}

Let us assume for the moment that \eqref{e:lsc volume} holds for $\hat u_j$ and $\hat u:=\mathcal{T}_{k}(u)$. Then,
\eqref{eqthatkwcm} yields
\begin{multline*}
\int_\Omega\Psi(\nabla \hat u)\dx\le\liminf_{j\to\infty} \int_\Omega\Psi(\nabla \hat u_j)\dx\\
\leq \liminf_{j\to\infty}\int_{\Omega} \Psi(\nabla u_j) \dx+\frac CM
+C\calL^n(\{|u|\ge a_{M}\}).
\end{multline*}
By this and
\[
\Big|\int_\Omega\Psi({\nabla \hat u})\dx-\int_\Omega\Psi(\nabla u)\dx\Big|\leq C\int_{\{|u|>a_M\}}(|\nabla
u|^q+1)\dx,
\]
letting $M\to\infty$, we conclude that \eqref{e:lsc volume} holds.
In the rest of the proof we will assume ${u},\,u_j\in L^\infty\cap SBV(\Omega;\R^m)$, and $u_j\to u$ in
$L^1(\Omega;\R^m)$.

{\bf Step~3. Blow-up.} For every $j\in\N$ consider the measures defined by
\begin{equation*}
\mu_j:=\Psi(\nabla u_j)\calL^n\LL \Omega,\qquad\nu_j:=g_0(|[u_j]|)\calH^{n-1}\LL J_{u_j}\,.
\end{equation*}
Assumptions \eqref{e:Psi coercive}, {($\Psi$-a'')} and ($\Psi$-b'') {imply} that $(\mu_j)_j$ and $(\nu_j)_j$ are
equibounded
in mass. Passing to a subsequence we may assume that $\mu_j\weakto\mu_\Psi$ and 
$\nu_j\weakto{\nu_\Psi}$ weakly-$*$ in the sense of measures on $\Omega$
as $j\to\infty$, for some $\mu_\Psi,\,\nu_\Psi\in\calM^+_b(\Omega)$.
In addition, we may assume that the right-hand side  of \eqref{e:lsc volume} is a limit, 
which is finite in view of \eqref{e:Psi coercive} and assumption {($\Psi$-a'')}.
Notice that if $H_\Psi(u)\leq \mu_\Psi(\Omega)$ then
necessarily the conclusion in \eqref{e:lsc volume} holds. 
Clearly, to establish the former inequality it suffices to show that
\begin{equation}\label{eqlbvolulmepart}
\Psi(\nabla u(x_0))\leq \frac{\dd\mu_\Psi}{\dd\calL^n}(x_0)\,
\end{equation}
for $\calL^n$-a.e. $x_0\in \Omega$. 
By the Besicovitch derivation theorem \cite[Theorem~2.22]{AFP} we have 
\begin{equation}\label{e: Besicovitch}
\frac{\dd\mu_\Psi}{\dd\calL^n}(x_0)+\frac{\dd{\nu_\Psi}}{\dd\calL^n}(x_0)<\infty\,
\end{equation}
for $\calL^n$-a.e. $x_0\in \Omega$. Next we observe that for $\calL^n$-a.e. $x_0\in \Omega$ one has
\begin{equation*}
\frac{\dd \mu_\Psi}{\dd\calL^{n}}(x_0)
=\lim_{\rho\to0} \frac{\mu_\Psi(Q_\rho(x_0))}{\rho^n}
=\lim_{\rho\to0\atop \rho\in I} \lim_{j\to\infty}\frac{\mu_j(Q_\rho(x_0))}{\rho^n}
\end{equation*}
and
\begin{equation*}
\frac{\dd {\nu_\Psi}}{\dd\calL^{n}}(x_0)
=\lim_{\rho\to0} \frac{{\nu_\Psi}(Q_\rho(x_0))}{\rho^n}
=\lim_{\rho\to0\atop \rho\in I} \lim_{j\to\infty}\frac{\nu_j(Q_\rho(x_0))}{\rho^n}
\end{equation*}
where $Q_\rho(x_0):=x_0+(-\frac12\rho,\frac12\rho)^n$ and
$I:=\{\rho\in (0, {\frac{2}{\sqrt n}}\dist(x_0,\partial \Omega)):\,
\mu_\Psi(\partial Q_\rho(x_0))=\nu_\Psi(\partial Q_\rho(x_0))=0\}$.
We define {$u^\rho:Q_1\to\R^m$ by}
\begin{equation*}
u^\rho(y):=\frac{  u(x_0+\rho y)- u(x_0)} {\rho}.
\end{equation*}
By the Calder\'on-Zygmund theorem \cite[Theorem~3.83]{AFP}, for $\calL^n$-a.e. $x_0\in \Omega$,
after possibly extracting a further subsequence,
$u^\rho(y)\to \nabla u(x_0)y$ in $L^1(Q_1;\R^m)$ as $\rho\to0$.
We shall establish \eqref{eqlbvolulmepart} for all points $x_0$ satisfying
\eqref{e: Besicovitch},
and for which $u^\rho\to \nabla u(x_0)y$ in $L^1(Q_1;\R^m)$.

We further define
\begin{equation*}
u^\rho_j(y):=\frac{ u_j(x_0+\rho y)- u(x_0)} {\rho}
\end{equation*}
so that $u^\rho_j\to u^\rho$ in $L^1(Q_1;\R^m)$ as $j\to\infty$, for every fixed $\rho>0$.
We take a diagonal subsequence so that $w_i(y):=u^{\rho_i}_{{j_i}}(y)\to \nabla u(x_0)y$ in $L^1(Q_1;\R^m)$,
\begin{equation}\label{eqdmudlx0abdsdf}
\frac{\dd \mu_\Psi}{\dd\calL^{n}}(x_0)=
\lim_{i\to\infty}\int_{Q_1} \Psi(\nabla w_i) \dx\,,
\end{equation}
and
\begin{equation}\label{e:g0 wi bdd}
\sup_i\frac{1}{\rho_i}\int_{J_{w_i}\cap Q_1
} g_0(\rho_i|[w_i]|)\dd\calH^{n-1}<\infty\,.
\end{equation}
\medskip 

{\bf Step~{4}. Reduction to blow-ups bounded in $L^\infty$ and with vanishing singular total variations.}
By a truncation argument analogous to that used in Step~2, having fixed $M\in\N$, we find
$k_i\in\{M+1,\dots, 2M\}$ such that
$\hat w_i:=\mathcal{T}_{k_i}(w_i){\in SBV(Q_1;\R^m)}$, with $\mathcal{T}_{k_i}$ defined {in \eqref{e:Tk}}, satisfies
\begin{equation}\label{eqthatkwcm2}
\begin{split}
&\int_{Q_1} \Psi(\nabla \hat w_i) \dx \le 
\int_{Q_1} \Psi(\nabla w_i) \dx+\frac CM
+C\calL^n(\{|w_i|\ge a_{M}\}).
\end{split}
\end{equation}
Moreover, note that if ${a_M}>\|\nabla u(x_0)y\|_{L^\infty(Q_1)}$ then
$w_i\to \nabla u(x_0)y$ implies
$\hat w_i\to \nabla u(x_0)y$ in $L^1(Q_1;\R^m)$.
Since
 $\mathcal{T}_{k_i}\in C^1$, we deduce
$J_{\hat w_i}\subset J_{w_i}$; from the fact that $\mathcal{T}_{k_i}$ {is} $1$-Lipschitz and $g_0$
is monotone, we obtain 
\begin{equation}\label{e:g0 hatwi bdd}
\sup_i\frac{1}{\rho_i}\int_{J_{\hat w_i}\cap Q_1
} g_0(\rho_i|[\hat w_i]|)\dd\calH^{n-1}\leq \sup_i\frac{1}{\rho_i}\int_{J_{w_i}\cap Q_1
} g_0(\rho_i|[w_i]|)\dd\calH^{n-1}<\infty\,.
\end{equation}
{We observe that
subadditivity and monotonicity of $g_0$ imply}
\begin{equation}\label{eqg0rho}
\frac{g_0(\rho)}{\rho}\leq 2\frac{g_0(t\rho)}{t\rho}
\end{equation}
for every $\rho>0$ and $t\in(0,1]$.
Define the set $J_{\hat w_i}^1:=\{x\in J_{\hat w_i}: |[\hat w_i]|\geq1\}$. Then
\begin{align}\label{e:bd Ds hatwi}
|D^s\hat w_i|&(Q_1)=\int_{J_{\hat w_i}\cap Q_1} |[\hat w_i]|\dd\calH^{n-1}\notag\\
&\leq 2\int_{(J_{\hat w_i}\setminus J_{\hat w_i}^1)\cap Q_1
} \frac{g_0(\rho_i|[\hat w_i]|)}{g_0(\rho_i)}\dd\calH^{n-1}
+2a_{2M}\calH^{n-1}(J_{\hat w_i}^1\cap Q_1)\notag\\
&\leq 2\int_{{(J_{\hat w_i}\setminus J_{\hat w_i}^1)}\cap Q_1
} \frac{g_0(\rho_i|[\hat w_i]|)}{g_0(\rho_i)}\dd\calH^{n-1}
+2a_{2M}\int_{J_{\hat w_i}^1\cap Q_1} 
\frac{g_0(\rho_i|[\hat w_i]|)}{g_0(\rho_i)}\dd\calH^{n-1}\notag\\
&\leq C(1+a_{2M})\frac{\rho_i}{{g_0(\rho_i)}}\,,
\end{align}
where to deduce the first inequality we have used
{\eqref{eqg0rho}}
and the $L^\infty$  bound on $\hat w_i$; to deduce the second inequality we have used again the monotonicity of $g_0$,
and to deduce the last we have used \eqref{e:g0 hatwi bdd}.
Therefore, recalling that $g_0$ is superlinear at the origin, we infer that
\begin{align}\label{e:Ds hatwi vanishing}
|D^s\hat w_i|(Q_1)\to0\,.
\end{align}
\medskip

{\bf Step~{5}. Lower semicontinuity for coercive integrands.}
We are now ready to conclude the proof.
Indeed, thanks to \eqref{e:Psi coercive}, if $q>1$ we can find an increasing sequence of quasiconvex functions with
linear growth $(\Psi_k)_k$ such that $\sup_{k\in\N}\Psi_k(\xi)=\Psi(\xi)$ by applying Proposition~\ref{prophdeltaqc} with
$\delta=\delta_k\to 1^-$ and setting $\Psi_k:=h_{\delta_k}^\qc$.
We claim that for every $k\in\N$
\begin{equation}\label{e:lsc volume Fk}
\liminf_{i\to\infty}\int_{Q_1} \Psi_k(\nabla \hat w_i) \dx\geq \Psi_k(\nabla u(x_0))\,.
\end{equation}
Assuming the claim for the moment, by \eqref{eqdmudlx0abdsdf}, \eqref{eqthatkwcm2}, and \eqref{e:lsc volume Fk} we deduce that
\begin{align*}
\frac{\dd \mu_\Psi}{\dd\calL^{n}}(x_0)&=\lim_{i\to\infty}\int_{Q_1} \Psi(\nabla w_i) \dx
\geq\liminf_{i\to\infty}\int_{Q_1} \Psi(\nabla \hat w_i)\dx-\frac CM
\\&
\geq\liminf_{i\to\infty}\int_{Q_1} \Psi_k(\nabla \hat w_i) \dx-\frac CM
\geq \Psi_k(\nabla u(x_0))-\frac CM\,.
\end{align*}
The inequality in \eqref{eqlbvolulmepart} then follows at once by passing to the supremum over $k\in\N$ and letting
$M\to\infty$.

To conclude, we are left with establishing \eqref{e:lsc volume Fk}.
{From the definition in \eqref{eqdefhdelta} we}
check that $0\leq\Psi_k(\xi)\leq C_k{(|\xi|+1)}$, for some $C_k>0$ and all $\xi\in\R^{m\times n}$.
{Therefore we} may argue as in \cite[Lemma~1.6]{Kristensen99} to infer that for every $r\in(0,1)$
\begin{align}\label{e:BV qcvxty}
\int_{Q_1} \Psi_k(\nabla \hat w_i) \dx&\geq {r^n}\Psi_k(\nabla u(x_0))\notag\\
&-{C_k |D^s\hat w_i|(Q_1)}
-\frac{{C_k}}{1-r}\int_{Q_1\setminus {Q_r}}|\hat w_i-\nabla u(x_0) y|\dy\,.
\end{align}
Hence, \eqref{e:Ds hatwi vanishing} and the convergence $\hat w_i\to \nabla u(x_0)y$ in $L^1(Q_1;\R^m)$ imply
\eqref{e:lsc volume Fk} by letting $i\to\infty$ and then $r\to 1^-$.

If $q=1$ we may argue as above in order to deduce inequality \eqref{e:BV qcvxty} directly for $\Psi$. 
The conclusion then follows at once.
\end{proof}

\subsection{Lower semicontinuity for free-discontinuity energies}

As a consequence of Propositions~\ref{p:lsc surface} and~\ref{p:lsc bulk}, we obtain 
the following lower semicontinuity statement for a large class of superlinear free-discontinuity energies generalizing
\cite[Theorem~4.5, Remark~4.6]{Ambrosio1994}, and the result in \cite{Kristensen99} (see the comments after
\cite[Theorem~1.2]{Kristensen99} and Remark~1 after \cite[Theorem~6.2]{Kristensen99}).
\begin{theorem}\label{t:lsc}
Let $g:\R^m\times S^{n-1}\to[0,\infty)$ satisfy  \eqref{gBVell}-\eqref{e:g0 superlinear}, and let
$\Psi:\R^{m\times n}\to[0,\infty)$ be continuous, quasiconvex and satisfy \eqref{e:Psi gc gen} with
$q\in(1,\infty)$.
Then
\[
H_g(u)+H_\Psi(u)\leq\liminf_{j\to\infty}(H_g(u_j)+H_\Psi(u_j))\,,
\]
if $u$, $u_j\in {(GSBV(\Omega))^m}$ satisfy ($\Psi$-a''), ($\Psi$-b''), and ($\Psi$-c).
\end{theorem}
\begin{proof}
Proposition~\ref{p:lsc bulk} directly implies the lower semicontinuity inequality for the bulk part.
Furthermore, $(\nabla u_j)_j$ is equiintegrable in $L^1(\Omega;\R^{m\times n})$ by assumption {($\Psi$-a'')}
since $q>1$, so that ($g$-a') is satisfied. The conclusion then follows at once from 
Proposition~\ref{p:lsc surface} and the superadditivity of the $\liminf$ operator.
\end{proof}

\section{Relaxation}\label{s:relaxation}

In this section we dispense with the quasiconvexity assumption on $\Psi$ and the $BV$-ellipticity assumption on $g$, and
identify the lower semicontinuous envelope of the corresponding energy
${F:=}H_\Psi+H_g$ with respect to the strong $L^1$ topology.
We recall that, given a metric space $(X,d)$ endowed with the topology induced by $d$, and given a functional $F:X\to
[0,\infty]$, the relaxation $\overline F:X\to[0,\infty]$ is defined by
\[\overline F(u):=\inf\{\liminf_{k\to\infty}F(u_k):\ u_k\stackrel{d}{\to} u\}.\]	
{The relaxation of $H_\Psi$, which is the restriction of $F$ to Sobolev functions, is
an integral functional with energy density given by the quasiconvex
envelope of $\Psi$, which was defined in \eqref{e:qc envelope}, see \cite{Dacorogna}.}
{If one assumes that $g\ge c>0$ then, both in the case of partition problems (where $u$ takes values in a finite
 or discrete set) and for functionals defined on $SBV^p$, which in particular implies finiteness of the measure of the jump set, the relaxation
 is an integral functional with $g$ replaced by its $BV$-elliptic envelope, see
\cite{BFLM}.} The $BV$-elliptic envelope of
$g:\R^m\times S^{n-1}\to[0,\infty)$ {is defined by}
\begin{equation}\label{e:BV elliptic envelope}
g_{BV}(z,\nu):=\sup\{h(z,\nu):\, h\leq g,\, h\textrm{ $BV$-elliptic}\}\,.
\end{equation}
 We recall (see \eqref{e:BV ellipticity}) that $g:\R^m\times S^{n-1}\to[0,\infty)$ is $BV$-elliptic if
 for every $z\in\R^m$, every $\nu \in S^{n-1}$,
 every 
 {piecewise constant} 
 $u\in SBV(Q^\nu;\R^m)$ such that 
 $u-z\chi_{\{x\cdot\nu>0\}}$ has compact support in the cube $Q^\nu$ one has
 \begin{equation}\label{eqdefGbvellip}
  g(z,\nu)\le\int_{Q^\nu\cap J_u} g([u],{\nu_u}) \dH^{n-1}.
 \end{equation}
{Since we are dealing with test functions $u$ which take values in an infinite set, there are several variants
of this definition,
with different sets of test functions, which are not obviously equivalent.
We adopt the one from \cite{AFP} which uses piecewise constant functions, in the sense of functions which are constant
on a Caccioppoli partition. Such functions have a jump set of finite $\calH^{n-1}$-measure, which is a stronger property than finiteness of the integral in 
\eqref{eqdefGbvellip}. In
\cite[Sect.~3.1]{BFLM} the even smaller class of piecewise constant functions which take finitely many values is used.
Alternatively, one could consider all $SBV$ functions with $\nabla u=0$ almost everywhere, or require the level sets
to be polyhedral.
Based on the density result in
\cite{CFI23}, we show in Lemma~\ref{lemmaBVellipiinf} that these definitions are {all} equivalent.
To do this, it is useful to introduce some notation.}
 
We denote by $PA(\R^n;\R^m)$ the space of functions for which there
exists a locally finite decomposition of $\R^n$ into simplexes, such that $u$ is affine in the interior of each of them.
By restriction to $\Omega$, one obtains {that any $u\in PA(\R^n;\R^m)$ satisfies
$u\in SBV\cap L^\infty(\Omega;\R^m)$.}

 For $h:\R^m\times S^{n-1}\to[0,\infty)$ we define $T(h):\R^m\times S^{n-1}\to[0,\infty)$ by
\begin{equation}\label{eqdefThBV}
\begin{split}
  T(h)(z,\nu):=&\inf \left\{ \int_{Q^\nu\cap J_u} h([u],{\nu_u}) \dH^{n-1}:
  u\in PA(\R^n;\R^m); \nabla u=0\ \ {\calL^n\text{-a.e.}}; \right.\\
  & \hskip4cm \left.\phantom{\int_{Q^\nu}}
  \supp(u-z\chi_{\{x\cdot\nu>0\}})\subset\subset Q^\nu
  \right\}
  \end{split}
\end{equation}
{and, provided $h$ is Borel measurable, $T_*(h):\R^m\times S^{n-1}\to[0,\infty)$ by
\begin{equation}\label{eqdefTstarhBV}
\begin{split}
  T_*(h)(z,\nu):=&\inf \left\{ \int_{Q^\nu\cap J_u} h([u],{\nu_u}) \dH^{n-1}:
  u\in SBV(Q^\nu;\R^m); \nabla u=0\ \ {\calL^n\text{-a.e.}}; \right.\\
  & \hskip4cm \left.\phantom{\int_{Q^\nu}}
  \supp(u-z\chi_{\{x\cdot\nu>0\}})\subset\subset Q^\nu
  \right\}.
  \end{split}
\end{equation}
The set of test functions in \eqref{eqdefThBV} is smaller than in
\eqref{eqdefGbvellip}, since the test functions take finitely many values, and the level sets are polygonals. In \eqref{eqdefTstarhBV} the set
of test functions is larger than in
\eqref{eqdefGbvellip}, as it includes functions with $\calH^{n-1}(J_u\cap Q^\nu)=\infty$.}
\begin{lemma}\label{lemmaBVellipiinf}
Let $g:\R^m\times S^{n-1}\to[0,\infty)$ satisfy  \eqref{e:g consistent}-\eqref{eqgg0} for some
nondecreasing, subadditive function $g_0\in C^0([0,\infty);[0,\infty))$ with $g_0^{-1}(0)=\{0\}$.
Then $g_{BV}\in C^0(\R^m\times S^{n-1})$ satisfies {\eqref{gBVell}}-\eqref{eqgg0}, {with the same $g_0$}, and
\begin{equation}\label{eqgBVinf}
 g_{BV}=T(g)={T_*(g)}.
\end{equation}
\end{lemma}
\begin{proof}
The fact that $g_{BV}$ is $BV$-elliptic follows from the fact that for any function $h$ as in \eqref{e:BV elliptic
envelope} and any test function $u$ as in \eqref{eqdefGbvellip}
one has
\begin{equation}
  h(z,\nu)\le\int_{Q^\nu\cap J_u} h([u],{\nu_u}) \dH^{n-1}
  \le\int_{Q^\nu\cap J_u} g_{BV}([u],{\nu_u}) \dH^{n-1}.
\end{equation}

{We first check that the definition of $T(g)$ is independent of rotations of $Q^\nu$ around the axis $\nu$. Let $z\in\R^m$, $\nu\in S^{n-1}$, and let $Q^\nu$ and $Q^{\nu}_*$ be unit cubes, with one face orthogonal to $\nu$. Fix $\eps>0$. For $\lambda\in(0,1)$ sufficiently small depending on $\eps$, select $x_1,\dots,x_{N_\lambda}\in Q^\nu\cap \{x\cdot\nu=0\}$ such that
the cubes $q_i:=x_i+\lambda Q^{\nu}_*$ are pairwise disjoint, and
\[\calH^{n-1}\Big((Q^\nu\cap \{x\cdot\nu=0\})\setminus \bigcup_{i=1}^{N_\lambda}q_i\Big)<\eps.\]
By definition of $T(g)(z,\nu;Q^\nu_*)$, where we have made explicit the dependence on the reference cube, there exists $u$ as in \eqref{eqdefThBV} such that $u_i(y):=u(\frac{y-x_i}{\lambda})$, $y\in q_i$, satisfies
\begin{equation}\label{contTh}\eps+T(g)(z,\nu;Q^\nu_*)\geq \frac{1}{\lambda^{n-1}}\int_{q_i\cap J_{u_i}}g([u_i],\nu_{u_i})\dH^{n-1}.
\end{equation}
Define
\[
\tilde u:=\sum_{i=1}^{N_\lambda} u_i\chi_{q_i}+z\chi_{\{x\cdot\nu>0\}\setminus \cup_i q_i},
\]
which is a competitor as in \eqref{eqdefThBV} for $T(g)(z,\nu; Q^\nu)$. Hence, using
\eqref{contTh}, we get
\begin{multline*}T(g)(z,\nu;Q^\nu)\leq \int_{Q^\nu\cap J_{\tilde u}}g([\tilde u],\nu_{\tilde u})\dH^{n-1}\\
\leq\sum_{i=1}^{N_\lambda}\int_{q_i\cap J_{u_i}}g([u_i],\nu_{u_i})\dH^{n-1}+
g(z,\nu)\eps\\
\leq \eps+T(g)(z,\nu;Q^\nu_*)+g(z,\nu)\eps.
\end{multline*}
Taking the limit as $\eps\to 0$, we get $T(g)(z,\nu;Q^\nu)\leq T(g)(z,\nu;Q^\nu_*)$ and, reversing the roles of $Q^\nu$ and $Q^\nu_*$, we get the equality.}

{One important ingredient of the proof is that for any
{locally bounded}
$h:\R^m\times S^{n-1}\to[0,\infty)$ one has
\begin{equation}\label{eqTThTH}
 T(T(h))=T(h).
\end{equation}
The inequality $T(T(h))\le T(h)$ is obvious; the reverse inequality follows from
a standard covering argument.} Indeed, fix $z\in\R^m$ and $\nu\in S^{n-1}$. For $\eps>0$, by definition of $T(T(h))$ there is $u$ as in \eqref{eqdefThBV} such that
\[
\int_{{Q^\nu\cap}J_u}T(h)([u],\nu_u)\dd\calH^{n-1}\leq T(T(h))(z,\nu) +\eps.
\]
Since $u$ is piecewise constant on a {triangulation} of $Q^\nu$, we can cover $J_u$ {(up to a null set)} by
countably many {pairwise disjoint} cubes {$Q_j=x_j+Q_{r_j}^{\nu_j}$
with $x_j\in J_u$, $\nu_j$ the normal to $J_u$ in $x_j$,
and $u(x)=u^\pm(x_j)$ almost everywhere on $Q_j\cap\{\pm(x-x_j)\cdot\nu_j>0\}$. Further,} {since $u\in L^\infty(Q^\nu)$} for some $M$ large
\[
\sum_{j> M}\int_{J_u\cap Q_j}h([u],\nu_u)\dd\calH^{n-1}\leq \eps.
\]
For $j=1,\dots,M$, by definition of $T(h)$ there is $w_j\in {PA(\R^n;\R^m)}$  such that $w_j=u$ {in a neighbourhood
of}  $\partial Q_j$ and
\[
\int_{{Q_j\cap} J_{w_j}}h([w_j],\nu_{w_j})\dd\calH^{n-1}\leq r_j^{n-1}T(h)([u],\nu_u)
{+\frac\eps M}
.
\]
Defining $w:=w_j$ in $Q_j$, for $j=1, \dots, M$, and $w:=u$ otherwise in $Q^\nu$, we get
\begin{multline*}
T(h)(z,\nu)\leq \int_{J_w} h([w],\nu_w)\dd\calH^{n-1}\leq \sum_{j=1}^M \int_{J_{w_j}\cap Q_j}
h([w_j],\nu_{w_j})\dd\calH^{n-1}+\eps\\\leq \sum_{j=1}^M \calH^{n-1}(J_u\cap Q_j)T(h)([u],\nu_u)
+ 2\eps\leq
T(T(h))(z,\nu)+ 3\eps.
\end{multline*}
As $\eps\to 0$, we get $T(h)\leq T(T(h))$. {This concludes the proof of \eqref{eqTThTH}.}

We next check that if $g$ obeys
\eqref{e:g consistent}--\eqref{eqgg0}, then so does $T(g)$. Indeed, property
\eqref{e:g consistent} is immediate.
The upper bound in \eqref{e:g grwoth g0} follows from $T(g)\le g$; the lower bound follows from the fact that since $g_0$ is subadditive the function $(z,\nu)\mapsto g_0(|z|)$ is $BV$-elliptic, {see \cite[Example 2.8]{AmbrosioBraides1990}}. Further, \eqref{eqTThTH} implies (again by rescaling and gluing competitors) that $T(g)$ is subadditive in the first argument. Hence
$T(g)(z,\nu)\le T(g)(z',\nu)+g(z-z',\nu)\le T(g)(z',\nu)+cg_0(|z-z'|)$, which proves
property \eqref{eqgg0}.

{Let us now show that 
$T(g)$ is continuous. To this end, first note that it is sufficient to show the continuity of $T(g)(z,\cdot)$ for every $z\in\R^{m}$
in view of the estimate in \eqref{eqgg0} for $T(g)$, which implies the continuity of $T(g)(\cdot,\nu)$ uniformly in $\nu\in S^{n-1}$.
Then fix $z\in \R^m$, $\eps>0$, $\nu,\nu'\in S^{n-1}$, two unit cubes $Q^\nu,\,Q^{\nu'}$, and 
consider $u$ as in \eqref{eqdefThBV} such that
\begin{equation}\label{contTh2}
\int_{Q^{\nu'}\cap J_{u}}g([u],\nu_{u})\dH^{n-1}\leq T(g)(z,\nu')+\eps\,.
\end{equation}
We assume that $Q^{\nu'}$  converges to $Q^{\nu}$ in the Hausdorff distance as $\nu'\to\nu$, and 
denote by $\lambda_{\nu'}\in(0,1]$ the maximum dilation factor such that $\lambda_{\nu'}Q^{\nu'}\subseteq Q^{\nu}$. 
Note then that $\lambda_{\nu'}\to1^-$ as $\nu'\to \nu$.
Next define $w(x):=u(\lambda_{\nu'}^{-1}x)$ if $x\in\lambda_{\nu'}Q^{\nu'}$ and $w:=z\chi_{\{x\cdot\nu\geq0\}}$ if 
$x\in \R^n\setminus\lambda_{\nu'}Q^{\nu'}$. By construction $w\in PA(\R^n;\R^m)$ with $\nabla w=0$ $\calL^n$-a.e. and 
$\mathrm{supp}(w-z\chi_{\{x\cdot\nu\geq0\}})\subset\subset Q^{\nu}$. Moreover, on setting $\Sigma^{\nu'}:=\{x\in \lambda_{\nu'}\partial Q^{\nu'}:\,(x\cdot\nu)(x\cdot\nu')<0\}\cup\{x\in Q^{\nu}\setminus\lambda_{\nu'}Q^{\nu'}:\,x\cdot\nu=0\}$, 
the growth condition in \eqref{e:g grwoth g0} and a change of variable yield
\begin{align*}
T(g)(z,\nu)&\leq\int_{Q^{\nu}\cap J_{w}}g([w],\nu_{w})\dH^{n-1}\\
&\leq\int_{\lambda_{\nu'}Q^{\nu'}\cap J_{w}}g([w],\nu_{w})\dH^{n-1}
+cg_0(|z|)\calH^{n-1}(\Sigma^{\nu'})\\
&=\lambda_{\nu'}^{n-1}\int_{Q^{\nu'}\cap J_{u}}g([u],\nu_{u})\dH^{n-1}
+cg_0(|z|)\calH^{n-1}(\Sigma^{\nu'})\\
&\leq \lambda_{\nu'}^{n-1}(T(g)(z,\nu')+\eps)+cg_0(|z|)\calH^{n-1}(\Sigma^{\nu'})\,.
\end{align*}
Since $\calH^{n-1}(\Sigma^{\nu'})\to 0$ as $\nu'\to \nu$ we infer that
\[
T(g)(z,\nu)\leq\liminf_{\nu'\to\nu}T(g)(z,\nu')\,.
\] 
Exchanging the roles of $\nu$ and $\nu'$ in the construction above, we conclude that 
\[
\limsup_{\nu'\to\nu}T(g)(z,\nu')\leq T(g)(z,\nu)\,,
\] 
and the claimed continuity follows at once. 
}

{In order to prove \eqref{eqgBVinf}, it suffices to show that
\begin{equation}\label{eqTstarTgTg}
T_*(T(g))=T(g).
\end{equation}
Indeed, combining \eqref{eqTstarTgTg} with the immediate inequalities $T_*(T(g))\le T_*(g)\le T(g)$ we deduce $T_*(g)=T(g)$.
Similarly, since the set of test functions in \eqref{eqdefTstarhBV} contains the one in
\eqref{eqdefGbvellip}, \eqref{eqTstarTgTg} implies in particular that $T(g)$ is $BV$-elliptic, hence it is one of the
functions entering \eqref{e:BV
elliptic envelope}, and
$T(g)\le g_{BV}$.
Conversely, from the definition {of $T(g)$}
and the fact that $g_{BV}$ is $BV$-elliptic
one obtains $g_{BV}\le  T(g)$.
Therefore $T(g)=g_{BV}$, which concludes the proof.}

{
It remains to prove \eqref{eqTstarTgTg}.
The inequality
$T_*(T(g))\le T(g)$ is obvious.}
Fix $z\in\R^m$, $\nu\in S^{n-1}$,
and a test function $u$ as in \eqref{eqdefTstarhBV}, that is,
a function
$u\in SBV(Q^\nu;\R^m)$ such that
$\nabla u=0$ {$\calL^n$-a.e.} {on $Q^\nu$} and
$ \supp(u-z\chi_{\{x\cdot\nu>0\}})\subset\subset Q^\nu$.
We need to show that
\begin{equation}\label{eqTgBVell}
T(g)(z,\nu)\le \int_{Q^\nu\cap J_u} T(g)([u],{\nu_u}) \dH^{n-1}.
\end{equation}
The key point is to apply the density result in
\cite{CFI23}. It provides a sequence $u_j\in PA(\R^n;\R^m)$, converging in $L^1$ and for which the corresponding energies on the right-hand side of \eqref{eqTgBVell} converge to that of $u$. Here $PA(\R^n;\R^m)$ represents
piecewise affine functions with polyhedral jump sets.
A key point is to ensure that the condition
$ \supp(u-z\chi_{\{x\cdot\nu>0\}})\subset\subset Q^\nu$ is preserved.
To do this, we observe {that} if one applies the construction in
\cite[Theorem~1.1]{CFI23} to a function $\tilde u$ with $\supp \tilde u\subset\subset \Omega$, then (for sufficiently
large $j$) one also has $\supp \tilde u_j\subset\subset \Omega$.
To see  this, it suffices to follow the proof of that theorem
in \cite[Section~4.3]{CFI23}.
One first extends $\tilde u$ by zero to a function defined on $\R^n$, then
remarks that,
{in the notation of that proof, the cubes $Q_z^\gamma$ with $z\in  A_\delta$ obey
$\dist(Q_z^\gamma,\partial\Omega)\ge \delta(\sqrt n/2-1/4)\ge\delta/3$,
so that the large-jump part of the reconstructed sequence vanishes around the boundary. Secondly, for the small-jump
and diffuse parts, by
\cite[{Theorem~1.1 Step~3 and }Proposition~4.3(iv)]{CFI23}
{the} construction vanishes outside
$B_{\eps\sqrt n}(\supp\tilde u)$.}

Therefore, an application of
\cite[Theorem~1.1]{CFI23}
to $\hat u:=u-z\chi_{\{x\cdot\nu>0\}}$ (with $p=1$ and {$g_0$})
leads to a piecewise constant sequence $\hat u_j\in PA(\R^n;\R^m)$ which is compactly supported in $Q^\nu$.
For large $j$ the function
$u_j:=\hat u_j+z\chi_{\{x\cdot\nu>0\}}$ is an admissible test function in
 \eqref{eqdefThBV} {with $h=T(g)$}. Further, {by} \cite[Corollary~2.1]{CFI23}, one has
\begin{equation}
 T(T(g))(z,\nu)\le \lim_{j\to\infty}
\int_{Q^\nu\cap J_{u_j}} T(g)([u_j],{\nu_{u_j}}) \dH^{n-1}=
\int_{Q^\nu\cap J_u} T(g)([u],{\nu_u}) \dH^{n-1}.
 \end{equation}
 Recalling that $T(T(g))=T(g)$, this proves \eqref{eqTgBVell}
and concludes the proof {of \eqref{eqTstarTgTg}}.
\end{proof}

{By Lemma~\ref{lemmaBVellipiinf} one easily obtains that for any $z\in\R^m$, $\nu\in S^{n-1}$, and}
$\eps>0$ there is a function {$u\in {PA(\R^n;\R^m)}$ with $\nabla u=0$ {$\calL^n$-a.e.}} such that
\begin{equation}\label{eq:scalcov}
\|u-z\chi_{\{x\cdot\nu>0\}}\|_{L^1(Q^\nu)}\le\eps \ \ \text{and} \ \ \int_{Q^\nu\cap J_u} g([u],{\nu_u}) \dH^{n-1}\le
  g_{BV}(z,\nu)+\eps.
\end{equation}
This is obtained by covering the mid-plane of $Q^\nu$ by small cubes ${x_i+}Q^\nu_\rho$ and by gluing together the
{corresponding translations of the function} $u_\rho$ defined on $Q^\nu_\rho$, obtained by suitably rescaling
a good competitor $u$ for \eqref{eqgBVinf}. The details are left to the reader.

Using this remark,
{we will in fact prove  the following stronger
version of Theorem~\ref{relaxation} {that will be used in Proposition~\ref{p:ubp}}.
\begin{theorem}\label{relaxation2}
Under the same hypotheses as in Theorem~\ref{relaxation}, let
${H_0}:L^1(\Omega;\R^{m})\to[0,\infty]$ be
\begin{equation}\label{Hg}
{H_0}(u):=\begin{cases}
\displaystyle \int_\Omega \Psi(\nabla u)\dx+\int_{J_u{\cap\Omega}}g([u],\nu_u)\dH^{n-1}, & \text{if }u\in
PA({\R^n};\R^m),\\ 
& {\calH^{n-1}(J_u\cap \partial\Omega)=0,}\\
\infty, & \text{otherwise.}
\end{cases}
\end{equation}
Then, the relaxation $\overline{H}_0$ with respect to the {strong topology of $L^1(\Omega;\R^m)$} is the functional
\begin{equation}\label{Hbarg}
\overline{H}_0(u)=\int_\Omega \Psi^\qc(\nabla u)\dx + \int_{J_u}g_{BV}([u],\nu_u)\dH^{n-1},
\end{equation}
if $u\in
{(GSBV(\Omega))^m}$
with $\nabla u\in L^q(\Omega;\R^{m\times n})$, and
$\overline{H}_0(u)=\infty$ otherwise.
\end{theorem}}
\begin{proof}
Let $\calF$ be the functional in the right-hand side of \eqref{Hbarg}. We first show that
$\calF\leq \overline{H}_0$. In
view of the definition of $H_0$, we immediately have $\calF\leq H_0$. Since the functional
$\calF$ is lower semicontinuous with respect to the {strong topology of $L^1(\Omega;\R^m)$} thanks to Lemma \ref{lemmaBVellipiinf},
Propositions~\ref{p:lsc surface}, and \ref{p:lsc bulk}, we conclude that $\calF \leq \overline{H}_0$.

Conversely, let us prove that $\calF\geq \overline{H}_0$.
By a standard diagonal argument,
it suffices to prove the following:
Given $u\in {(GSBV\cap L^1(\Omega))^m}$ such that $\calF(u)<\infty$ and $\delta>0$,
there is $U_\delta\in {PA(\R^n;\R^m)}$ {with $\calH^{n-1}(J_{{U_\delta}}\cap \partial\Omega)=0$} such that $\|u-U_\delta\|_{L^1({\Omega})}\le C\delta$
and $H_0(U_\delta{;\Omega})\le \calF(u{;\Omega})+C\delta$.

A careful inspection of the proof of
\cite[{Corollary~2.4}]{CFI23} {applied with} $\Psi^{\qc}$ {(in place of $\Psi$),}
$g_0$, and {$g_{BV}$ in place of $g$}, {gives that}
there exist an open set $\Omega'$, with $\overline{\Omega}\subset\Omega'$, and a function $u_1\in PA(
{\R^n};\R^m)$ such that $\|u_1-u\|_{{L^1(\Omega')}}\le \delta$, {$\calH^{n-1}(J_{u_1}\cap\partial\Omega)=0$} and
\begin{equation}\label{eqcalfu1u}
\calF(u_1;{\Omega'})\le \calF(u;{\Omega})+\delta.
\end{equation}
{Indeed, \cite[Corollary~2.4]{CFI23} is based on \cite[Theorem~1.1]{CFI23}. An extension of $u$ to $\R^n$, not relabelled, satisfying good energy estimates and such that $\calH^{n-1}(J_u\cap \partial\Omega)=0$, is  provided first in \cite[Theorem~1.1, Step~1]{CFI23}.
	Then, in {\cite[Theorem~1.1, end of Step~3]{CFI23}} 
	one {can check} 
	that near $\partial\Omega$ our function $u_1$, 
	{that is $w_\zeta$ in the notation there, is defined as the piecewise affine function $\Pi_{\eps,\zeta}u$, with $\zeta\in B_\eps$.}
	The choice of $\zeta$ is performed in \cite[Theorem~1.1, Step~7]{CFI23} and can be {further} refined excluding a $\calL^n$-negligible set {without affecting the conclusions of \cite[Theorem~1.1]{CFI23}. More precisely,} using that $\calH^{n-1}\res\partial \Omega$ is finite, 
{we choose $\zeta\in B_\eps$ satisfying $\calH^{n-1}(\partial\Omega\cap J_{\Pi_{\eps,\zeta}u})=0$,
in addition to the conditions imposed in \cite[Theorem~1.1, Step~7]{CFI23}. Indeed, the latter condition 
holds for every $\zeta\in B_\eps$ up to a $\calL^n$-negligible subset.}
The properties of $\Pi_{\eps,\zeta}$ in \cite[Proposition~4.3]{CFI23} finally give \eqref{eqcalfu1u} when 
$\Psi$ and $g$ are replaced by $|\cdot|^q$ and $g_0$. The analogous estimate with the given $\Psi$ and 
$g$ follows by the proof of \cite[Corollary~2.4]{CFI23}.}

In particular, we can assume that there are finitely many simplexes $E_k$
such that $\Omega \cap E_k$ has positive measure for each $k$, $u_1$ is affine on each of them, and the sets $\Omega\cap E_k$ cover $\Omega$ up to a null set. 

Let
\begin{equation}\label{eqdefMdelta}
 M_\delta:=1+\| \Psi(\nabla u_1)\|_{L^\infty(\Omega)}
 +\| \nabla u_1\|_{L^\infty(\Omega)}
 +\|[u_1]\|_{L^\infty(\Omega\cap J_{u_1};\calH^{n-1})}
 + \calH^{n-1}(J_{u_1}).
\end{equation}
This constant is finite, since $\nabla u_1$ takes finitely many values, and $[u_1]$ can be estimated  in terms of finitely many
affine functions. It depends on $u$ and $\delta$ via $u_1$; we keep only the dependence on $\delta$ explicit.

For each $k$ we select a Lipschitz set $E_k'\subset \Omega\cap E_k$ such that
$\dist(E_k',\partial E_k)>0$ and
\begin{equation}\label{eqchoiceEkp}
\calL^n(\Omega\cap E_k\setminus E_k')\le
\frac\delta{M_\delta} \calL^n(\Omega\cap E_k).
\end{equation}
{We refer to Figure~\ref{fig-relax-upper} for  a sketch of the decomposition of the domain.}
Then, by the definition of quasiconvexity and a standard covering argument (see for example \cite[Theorem~9.1]{Dacorogna} and {\cite[Proposition~2.8 in Chapter~X]{ekeland1999convex}), there exists $\varphi_k\in W^{1,\infty}_0(E_k';\R^m)$ piecewise affine} such that
$\|\varphi_k\|_{L^\infty{(E_k')}}\le\delta$
and
\begin{equation}\label{eqqcdelta}
\int_{E_k'}\Psi(\nabla u_1+\nabla \varphi_k)\dx\le
\delta\calL^n(E_k') + \int_{E_k'}\Psi^{\qc}(\nabla u_1)\dx.
\end{equation}
{Extending $\varphi_k$ to $0$ to the rest of $\Omega$},
we then define $u_2:\Omega\to\R^m$ by
\begin{equation}\label{eq:u2}
 u_2:=
 u_1+\sum_k 
 \varphi_{k}.
\end{equation}
We observe that 
\begin{equation}\label{equ2u1}\|u_2-u_1\|_{L^\infty(\Omega)} \le\delta,
\end{equation}
with
$u_2=u_1$ in a neighbourhood of $\cup_k\partial E_k$. Further,
from \eqref{eqqcdelta} we obtain
\begin{equation}
\int_{\Omega\cap E_k}\Psi(\nabla u_2)\dx\le
\Psi(\nabla u_1|_{E_k}) \calL^n(\Omega\cap E_k\setminus E_k')+
\delta\calL^n(E_k')+ \int_{\Omega\cap E_k}\Psi^{\qc}(\nabla u_1)\dx.
\end{equation}
Summing over $k$, recalling \eqref{eqchoiceEkp},
the fact that the $E_k$ are disjoint and cover $\Omega$,
and the definition of $M_\delta$, leads to
\begin{equation}\label{eqestpsinablau2}
 \int_\Omega \Psi(\nabla u_2)\dx
 \le 2\delta\calL^n(\Omega) +
 \int_\Omega \Psi^\qc(\nabla u_1)\dx .
\end{equation}

\begin{figure}
\begin{center}
 \includegraphics[width=8cm]{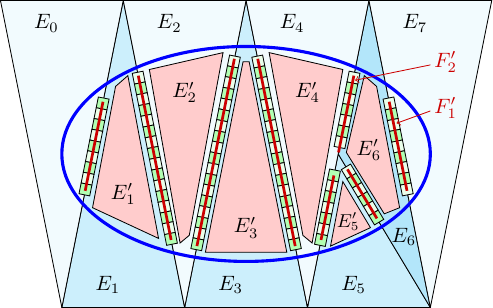}
 \end{center}
 \caption{Sketch of the construction in the proof of {Theorem~\ref{relaxation2}}.
 The set $\Omega$ (blue boundary) is covered by finitely many simplexes $E_k$ (light blue); a large part of the
interior of each of them is contained in the sets $E_k'$.
The codimension-one sets $F_j'$ (red) cover most of the part of the faces inside $\Omega$, and are in turn covered by the
union of
small cubes of size $\rho$ (green).}
 \label{fig-relax-upper}
\end{figure}

We next turn to the interfaces between the different sets $E_k$, which contain the jump set of $u_1$.
Consider the sets
$\Omega\cap \partial E_k\cap \partial E_{k'}$, $k\ne k'$.
We denote by $F_j$, $j=1,\dots, J$, those having positive $\calH^{n-1}$ measure,
and on which $[u_1]$ is not identically zero.
Then
\begin{equation}
 \calF(u_1)=
 \sum_k \int_{E_k\cap\Omega} \Psi^\qc(\nabla u_1) \dx+
 \sum_j \int_{F_j} g_{BV}(u_1^+-u_1^-,\nu_j) \dH^{n-1},
\end{equation}
where $\nu_j$ is the normal to $F_j$.
We remark that for each $j$ the function {$u_1$} is affine on each side of $F_j$.
As above, we fix {relatively} open sets $F_j'\subset F_j$ such that $\dist(F_j',\partial F_j)>0$
(the boundary here is taken in the $(n-1)$-dimensional sense, each of them being contained in an $(n-1)$-dimensional
plane)
and
\begin{equation}\label{eqchoiceFjp}
 \sum_j \calH^{n-1}(F_j\setminus F_j')\le \frac{\delta}{g_0(M_\delta)}.
\end{equation}

{Fix $\rho>0$, chosen below depending on $u_1$, $E_k'$, and $\delta$,} and cover a large part of each $F_j$
by cubes on a scale $\rho$. To do this,
let {$A_j:\R^n\to\R^n$} be an affine {isometry} that maps $\R^{n-1}$
{(which we identify with $\R^{n-1}\times \{0\}\subset\R^n$)}  to the $(n-1)$-dimensional affine space containing $F_j$.
For  $\pcube\in \rho\Z^{n-1}$ we let
$Q_{j,\pcube}:=A_j(\pcube+\rho(-\frac12,\frac12)^n)$.
Let $I_j$ denote the set of $\pcube\in\rho\Z^{n-1}$ such that $Q_{j,\pcube}\cap F_j'\ne\emptyset$.
Since
$\dist(E_k',\partial E_k)>0$ for all $k$ and
$\dist(F_j',F_l)>0$ for $j\ne l$,
for $\rho$ sufficiently small,
for all $j$ and all $\pcube\in I_j$ one has
$u_1=u_2$ on $Q_{j,\pcube}$
and that, for all $l\ne j$,
$Q_{j,\pcube}\cap F_l=\emptyset$,
and $Q_{j,\pcube}\cap Q_{l,\pcube'}=\emptyset$ for $\pcube'\in I_l$.

For each $j$ and $\pcube\in I_j$,
let $v_{j,\pcube}^\pm:=u_1^\pm(A_j(\pcube))$ denote the two traces of $u_1$ at the center of $Q_{j,\pcube}$ (the traces are pointwise defined,
since $u_1$ is piecewise affine).
From these two values we construct
a piecewise constant function by
\begin{equation}
V_{j,\pcube}(x):=
 \begin{cases}
 v_{j,\pcube}^+, & \text{ if } x\in
 A_j(\pcube+\R^{n-1}\times[0,\infty)),\\
 v_{j,\pcube}^-, & \text{ if } x\in
 A_j(\pcube+\R^{n-1}\times(-\infty,0)).
 \end{cases}
\end{equation}
We remark that $\|V_{j,\pcube}-u_1\|_{L^\infty(Q_{j,\pcube})}\le \rho \sqrt n M_\delta$.

By \eqref{eq:scalcov} there is $u_{j,\pcube}\in {PA(\R^n;\R^m)}$ with $\nabla u_{j,\pcube}=0$
$\calL^n$-a.e. on $Q_{j,\pcube}$ such that
\begin{equation}\label{eqgbvujp}
 \int_{Q_{j,\pcube}\cap J_{u_{j,\pcube}}} g([u_{j,\pcube}],\nu_{u_{j,\pcube}}) \dH^{n-1}\le
 \rho^{n-1} (\frac{\delta}{M_\delta}+g_{BV}(
 {[V_{j,\pcube}]},\nu_j) ),
\end{equation}
with
\begin{equation}\label{equVjp}
 {\supp (u_{j,\pcube}-V_{j,\pcube})\subset\subset} Q_{j,\pcube},\quad\text{and}\quad \|u_{j,\pcube}-V_{j,\pcube}\|_{L^1(Q_{j,\pcube})}\le \delta \rho^n.
\end{equation}
We set
\begin{equation}\label{e:urho}
 u^{(\rho)}:=u_2
 +\sum_j \sum_{\pcube\in I_j}
 \chi_{Q_{j,\pcube}} (u_{j,\pcube}-V_{j,\pcube}).
\end{equation}
By \eqref{equVjp}, no additional jump is inserted on $\partial Q_{j,\pcube}$.
In particular, $u^{(\rho)}\in SBV(\Omega;\R^m)$ with
\begin{equation}\label{eqnablaurhonablau2}
 \nabla u^{(\rho)}=\nabla u_2,
\end{equation}
so that \eqref{eqestpsinablau2} holds {with $u^{(\rho)}$ in place of $u_2$}, for any $\rho>0$. Further,
\begin{equation}\label{eqjumpurho}
\begin{split}
&{\calH^{n-1}(J_{u^{(\rho)}}\cap \partial\Omega)=0,}\\
& [u^{(\rho)}]=[u_{j,\pcube}]+[u_2-V_{j,\pcube}] \text{ in } Q_{j,\pcube},\\
& [u^{(\rho)}]=[u_2] \text{ outside }\bigcup_j\bigcup_{\pcube\in I_j} Q_{j,\pcube}.
\end{split}
\end{equation}
We recall that $u_1=u_2$ in each $Q_{j,\pcube}$, so that
\begin{equation}\label{e:stima u2 Vjp}
\|u_2-V_{j,\pcube}\|_{L^\infty(Q_{j,\pcube})} =
\|u_1-V_{j,\pcube}\|_{L^\infty(Q_{j,\pcube})}
\le \rho \sqrt n M_\delta.
\end{equation}
Choosing $\rho$ sufficiently small, we can have
\begin{equation}\label{eqg0rhodelta}
 g_0(2 \rho \sqrt n M_\delta)\le \frac{\delta}{M_\delta}.
\end{equation}
Since both $u_2$ and $V_{j,\pcube}$ jump only on the midplane of $Q_{j,\pcube}$, the second line of \eqref{eqjumpurho}, \eqref{eqgg0}, \eqref{eqgbvujp}, {\eqref{e:stima u2 Vjp}, and \eqref{eqg0rhodelta}} lead to
\begin{align}
 \int_{Q_{j,\pcube}\cap J_{u^{(\rho)}}}&
 g([u^{(\rho)}],{\nu_{u^{(\rho)}}}) \dH^{n-1}\notag\\
 \le&
 \int_{Q_{j,\pcube}\cap J_{u_{j,\pcube}}}
 g([u_{j,\pcube}],{\nu_{u_{j,\pcube}}}) \dH^{n-1}
 +{c} g_0(2 \rho \sqrt n M_\delta)\rho^{n-1}\notag
 \\
 \le &
 \int_{Q_{j,\pcube}\cap F_j} \left[\frac{\delta}{M_\delta}+
 g_{BV}([V_{j,\pcube}],{\nu_j})\right] \dH^{n-1}
 +{c}\frac{\delta}{M_\delta}\rho^{n-1}.
\end{align}
In turn, since 
$g_{BV}$ obeys \eqref{eqgg0}, we obtain {by \eqref{e:stima u2 Vjp}}
\begin{equation}
 g_{BV}([V_{j,\pcube}],{\nu_j})
 \le
 g_{BV}([u_1],\nu_j)+
 {c}g_0(2\rho \sqrt n M_\delta)
\end{equation}
 $\calH^{n-1}$-a.e. on $Q_{j,\pcube}\cap F_j$,
so that with \eqref{eqg0rhodelta} the above estimate reduces to
\begin{equation}\begin{split}\label{eqsplit1}
 \int_{Q_{j,\pcube}\cap J_{u^{(\rho)}}}
 g([u^{(\rho)}],\nu_{u^{(\rho)}}) \dH^{n-1}
 \le &
 3 {c}\frac{\delta}{M_\delta}\rho^{n-1}+
 \int_{Q_{j,\pcube}\cap F_j}
 g_{BV}([u_1],\nu_j) \dH^{n-1}
 \end{split}
\end{equation}
and summing over $\pcube\in I_j$,
\begin{equation}\begin{split}
\sum_{\pcube\in I_j}
\int_{Q_{j,\pcube}\cap J_{u^{(\rho)}}}
 g([u^{(\rho)}],\nu_{u^{(\rho)}}) \dH^{n-1}
 \le &
 3 {c}\frac{\delta}{M_\delta} \calH^{n-1}(F_j)+
 \int_{\Omega\cap F_j}
 g_{BV}([u_1],\nu_j) \dH^{n-1}
 .
 \end{split}
\end{equation}
At the same time, the last identity in \eqref{eqjumpurho} implies
\begin{equation}\begin{split}\label{eqsplit2}
\int_{F_j\setminus \cup_\pcube Q_{j,\pcube}}
 g([u^{(\rho)}],\nu_{u^{(\rho)}}) \dH^{n-1}
 =&
 \int_{F_j\setminus \cup_\pcube Q_{j,\pcube}}
 g([u_1],\nu_j) \dH^{n-1}\\
 \le&{c}  g_0(M_\delta) \calH^{n-1}(F_j\setminus F_j'),
 \end{split}
\end{equation}
where we used \eqref{eqdefMdelta} and monotonicity of $g_0$.
Since these sets cover the jump set of $u^{(\rho)}$, we conclude {by \eqref{eqsplit1} and \eqref{eqsplit2}}
\begin{equation}\begin{split}
 \int_{\Omega\cap J_{u^{(\rho)}}}
 g([u^{(\rho)}],\nu_{u^{(\rho)}}) \dH^{n-1}\le&
{c}   g_0(M_\delta)\sum_j \calH^{n-1}(F_j\setminus F_j')+
 3{c}\frac{\delta}{M_\delta}\sum_j\calH^{n-1}(F_j)\\
 & + \int_{\Omega\cap J_{u_1}} g_{BV}([u_1],{\nu_{u_1}}) \dH^{n-1}.
\end{split}\end{equation}
For $\rho$ sufficiently small we conclude from
\eqref{eqchoiceFjp} and \eqref{eqdefMdelta} that
\begin{equation}\label{e:stima urho u1}
 \int_{\Omega\cap J_{u^{(\rho)}}}
 g([u^{(\rho)}],\nu_{u^{(\rho)}}) \dH^{n-1}\le
{C}\delta+  \int_{\Omega\cap J_{u_1}}
 g_{BV}([u_1],{\nu_{u_1}}) \dH^{n-1}
\end{equation}
and, from {\eqref{e:urho},} \eqref{equ2u1} and \eqref{equVjp}, that
\begin{align}
\|u^{(\rho)}-u_1\|_{L^1(\Omega)}\le &
\calL^n(\Omega)
\|u_2-u_1\|_{L^\infty(\Omega)}\notag\\
&+\sum_{j,\pcube}\|u_{j,\pcube}-V_{j,\pcube}\|_{L^1(Q_{j,\pcube})}\leq {C}\delta\calL^n(\Omega).
\end{align}
{Combining \eqref{eqestpsinablau2},
\eqref{eqnablaurhonablau2},
and \eqref{e:stima urho u1}, we obtain
${H_0(u^{(\rho)})}\le \calF(u_1)+C\delta$, and with \eqref{eqcalfu1u}
the proof is concluded.}
\end{proof}

\section{Phase-field approximation}\label{s:phase field}

In this section we establish a variational approximation of lower semicontinuous energies as in \eqref{eq1} by phase-field models.

\subsection{Model}\label{ss:data}

We assume that $\Psi:\R^{m\times n}\to[0,\infty)$ is continuous, and for some
$q>1$ satisfies \eqref{e:Psi gc}, namely
\begin{equation*}
\Big(\frac1c |\xi|^q-c\Big)\vee0\le \Psi(\xi)\le c(|\xi|^q+1)
\hskip1cm\text{ for all }\xi\in\R^{m\times n}.
\end{equation*}
Moreover, we assume that the limit 
\begin{equation}\label{e:Psiinfty}
 \Psiinfty(\xi):=\lim_{t\to\infty} \frac{\Psi{(t\xi)}}{t^q}
\end{equation}
exists, and that the convergence is uniform on the set of $\xi$ with $|\xi|=1$. This means that for every $\delta>0$ there is $t_\delta$ such
that $|\Psi(t\xi)/t^q-\Psiinfty(\xi)|\le \delta$ for all $t\ge t_\delta$
and all $\xi$ with $|\xi|=1$, which is the same as
\begin{equation}\label{eqpsipsiinf}
|\Psi(\xi)-\Psiinfty(\xi)|\le \delta|\xi|^q\hskip5mm \text{ for all $|\xi|\ge t_\delta$}.
\end{equation}
By scaling, $\Psiinfty(t\xi)=t^q\Psiinfty(\xi)$ and in particular $\Psiinfty(0)=0$.
Uniform convergence also implies $\Psiinfty\in C^0(\R^{m\times n})$ and \eqref{e:Psi gc} {yields}
\begin{equation}\label{e:Psiinftygc}
\frac1c |\xi|^q\le \Psi_\infty(\xi)\le c|\xi|^q
\hskip1cm\text{ for all }\xi\in\R^{m\times n}.
\end{equation}

Following the analysis in the scalar case in \cite[Section~7.2]{ContiFocardiIurlano2016}
we provide an approximation of a model with power-law growth at small openings
(cf. \cite{ContiFocardiIurlano2022} for the analogous model with linear growth).
For all $\eps>0$, $q>1$, and  $p>0$ we consider the functional
$\Functepspq:L^1(\Omega;\R^{m+1})\times\calB(\Omega)\to[0,\infty]$ given by
\begin{equation}\label{functeps p}
{ \Functepspq(u,v;A):= \int_A \left( f_{\eps,p,q}^q(v) \Psi(\nabla u) +
\frac{(1-v)^{q'}}{q'q^{\sfrac{q'}q}\eps} + \eps^{q-1}|\nabla v|^q\right) \dx}
\end{equation}
if $(u,v)\in W^{1,q}(\Omega;\R^m\times [0,1])$ and $\infty$ otherwise, where
{$q'$ denotes the conjugate exponent of $q$, and for every
$t\in[0,1)$}
\begin{equation*}
 f_p(t):=\frac{\ell t}{(1-t)^p},\qquad
 f_{\eps,p,q}(t):= 1\wedge \eps^{{1-\sfrac1q}} f_p(t), \qquad
{f_{\eps,p,q}(1):=1}\,.
\end{equation*}
{Here and below, $\ell>0$ is a fixed parameter.}
Let us remark that $(0,\infty)\ni p\mapsto f_{\eps,p,q}(t)$ is increasing for all
$t\in[0,1]$. Therefore, we deduce that $\Functepspq\geq\calF_{\eps,1,q}$ on $L^1(\Omega;\R^{m+1})$ if $p\geq1$.
In particular, for $q=2$, the asymptotic analysis of the family $(\calF_{\eps,1,2})_{\eps>0}$ has been performed in
\cite{ContiFocardiIurlano2022}, where its convergence to a functional of cohesive type with surface energy densities
with linear growth for small jump amplitudes was established.
For every  $q>1$ and $p=1$ one can prove a completely analogous result following the same arguments.
Instead, here we will focus on the supercritical case $p>1$ leading to surface energy densities with superlinear growth
for small jump amplitudes, as investigated in \cite[Theorem~7.4]{ContiFocardiIurlano2016} in the scalar isotropic case.
In the subcritical setting, i.e. $p\in(0,1)$, a comparison argument yields that the
$\Gamma$-limit is then trivial, i.e. identically zero {(cf. \cite{ACF25-I})}.

To state our result we need to introduce some notation.
For all {Borel} subsets $A\in \calB(\Omega)$, and $(u,v)\in W^{1,q}(\Omega;\R^m\times [0,1])$, it is convenient
to define for $\cutoffconst\in(0,\infty]$
\begin{equation}\label{Feps* p}
\Functepspq^\cutoffconst(u,v;A):=\int_A \Big((\cutoffconst^{q-1}\wedge \eps^{q-1}f_{p}^q(v)) \Psiinfty(\nabla u)
+\frac{(1-v)^{q'}}{q'q^{\sfrac{q'}q}\eps} + \eps^{q-1}|\nabla v|^q\Big)\dx.
\end{equation}
In particular, we have
\begin{equation}\label{Feps infty}
\Functepspq^\infty(u,v;A)=\int_A \Big(\eps^{q-1}f_{p}^q(v) \Psiinfty(\nabla u)
+\frac{(1-v)^{q'}}{q'q^{\sfrac{q'}q}\eps} + \eps^{q-1}|\nabla v|^q\Big)\dx.
\end{equation}

Lower and upper bounds for $\Gamma(L^1)$-limits of $(\Functepspq)_\eps$  will be expressed
in terms of two different surface energies, defined respectively as
\begin{equation}\label{eqdefGpsnu inf}\begin{split}
\gpinf(z,\nu):= \inf \{&\liminf_{j\to\infty}
\calF_{\eps_j,p,q}^{\cutoffconst_j}
(u_j,v_j; Q^\nu): \\
&\|u_j- z\chi_{\{x\cdot\nu>0\}}\|_{L^1(Q^\nu)}\to0,\, \eps_j\to 0, \,\cutoffconst_j\to\infty\}
\end{split}
\end{equation}
and
\begin{equation}\label{eqdefGpsnu sup}
\begin{split}
\gpsup(z,\nu):=
\inf \{&\liminf_{j\to\infty}
\calF_{\eps_j,p,q}^\infty
(u_j,v_j; Q^\nu): \\
&\|u_j- z\chi_{\{x\cdot\nu>0\}}\|_{L^1(Q^\nu)}\to0, \, \eps_j\to 0\}\,,
\end{split}
\end{equation}
with $(z,\nu)\in \R^{m}\times S^{n-1}$.
{By Lemma~\ref{lemmaQnuQnup}, these expressions depend on $\nu$ and not on the choice of $Q^\nu$.}
Obviously, one can restrict to sequences $v_j\to1$ in $L^1(Q^\nu)$.
It is immediate to see that
\begin{equation}\label{e:trivial inequality gpinf gpsup}
\gpinf(z,\nu)\le\gpsup(z,\nu)\quad \text{ for all $(z,\nu)\in \R^m\times S^{n-1}$}.
\end{equation}
In general, we are not able to prove the equality of the above functions.
As a consequence, we carry out the full $\Gamma$-convergence analysis under the conditional assumption
\begin{equation}\label{e:gpinf=gpsup}
\gpinf(z,\nu)=\gpsup(z,\nu)\quad \text{ for all $(z,\nu)\in \R^m\times S^{n-1}$},
\end{equation}
which we actually show to be verified in the case where the $q$-recession function $\Psi_\infty$ of
$\Psi$ satisfies the projection property, namely
\begin{equation}\label{e:projectionproperty}
\Psi_\infty(\xi)\geq \Psi_\infty(\xi\nu\otimes\nu) \quad \text{for every $(\xi,\nu)\in\R^{m\times n}\times S^{n-1}$}.
\end{equation}
In general, we provide lower and upper bounds for the $\Gamma-\liminf$ and $\Gamma-\limsup$ which differ only as far as
the surface energy density is concerned. Indeed, a major difficulty in the asymptotic analysis is to prove that the
limit does not depend on the chosen infinitesimal sequence, a fact that is related to the lack of a rescaling property of the
functional $\calF_{\eps,p,q}$.
This issue was solved in the linear case $p=1$ (and $q=2$) in any dimension by an elementary truncation argument in the
$v$-variable (see \cite[Proposition~4.1]{ContiFocardiIurlano2022}). The latter has no immediate analogue in the current
superlinear setting due to the presence of the two exponents $p>1$ (in $f_p$) and $q$ (in the dissipation potential), which
yield two different
powers of $\eps$ as truncation thresholds. On the other hand, the one-dimensional superlinear setting was dealt
with in \cite[Theorem~7.4]{ContiFocardiIurlano2016} with an ad hoc argument, exploiting an $\eps$-independent
alternative
characterization of the surface energy density (cf. \eqref{e:gscal}). A different argument is used in what follows to
settle the case in which $\Psi_\infty$ satisfies the projection property.

\begin{theorem}\label{t:finale p}
Let $\Psi$ satisfy \eqref{e:Psi gc} and \eqref{eqpsipsiinf}, and let
$\Functepspq$ be the functional defined in \eqref{functeps p}, with $p,\,q>1$.
Assume that \eqref{e:gpinf=gpsup} holds, and denote by $g(z,\nu)$ the common value.

Then for all $(u,v)\in L^1(\Omega;\R^{m+1})$ it holds
$$\Gamma(L^1)\text{-}\lim_{\eps\to0}\Functepspq(u,v)=\Functlim(u,v),$$
where
\begin{equation}\label{Fp}
\Functlim(u,v):=\int_\Omega \Psi^\qc(\nabla u)\dx + \int_{J_u}g([u],\nu_u)\dH^{n-1},
\end{equation}
if $u\in (GSBV\cap L^1(\Omega))^m$ {with $\nabla u\in L^q(\Omega;\R^{m\times n})$}
and $v=1$ $\calL^n$-a.e. on $\Omega$, and $\Functlim(u,v):=\infty$ otherwise. {In particular, $g$ is $BV$-elliptic.}
\end{theorem}

In case the projection property holds for $\Psi_\infty$ we are able to prove the equality in \eqref{e:gpinf=gpsup} and
then deduce Theorem~\ref{t:onedim intro} in the introduction
(cf. \cite[Remark~3.6]{ContiFocardiIurlano2022}). We recall the statement for the sake of convenience.
\begin{theorem}\label{t:onedim}
Let $\Psi$ satisfy \eqref{e:Psi gc}, \eqref{eqpsipsiinf}, and \eqref{e:projectionproperty}.
Then  $\gpinf(z,\nu)=\gpsup(z,\nu)=g(z,\nu)$  for all $(z,\nu)\in \R^m\times S^{n-1}$, where
\begin{equation}\label{e:Psi sliceable}
g(z,\nu):=\lim_{T\uparrow\infty}\inf_{{\calV}_z^T}
\int_{-\sfrac T2}^{\sfrac T2}\Big(f_p^q(\beta(t))\Psi_\infty\big({\alpha'(t)}\otimes\nu\big)
+\frac{(1-\beta(t))^{q'}}{q'q^{\sfrac{q'}q}}+|\beta'(t)|^q\Big)\dd t
\end{equation}
and
\begin{align}\label{e:utildeznu}
\calV^T_z:=\{(\alpha,\beta)\in & W^{1,q}((-\sfrac T2,\sfrac T2);\R^{m+1})\colon
\notag\\ &0\leq \beta\leq 1,\, \beta(\pm\sfrac T2)=1,\, \alpha(-\sfrac T2)=0,\,\alpha(\sfrac T2)=z\}.
\end{align}
In particular, the $\Gamma(L^1)$-convergence result stated in Theorem~\ref{t:finale p} holds in this case.
\end{theorem}
The proof uses in a crucial way the density result obtained in \cite{CFI23}.
In particular, this completes the proof of the scalar case \cite[Theorem~7.4]{ContiFocardiIurlano2016} for which a density theorem was missing.

Having fixed $p,q>1$, throughout the rest of the paper, to simplify the notation,
we omit the subscripts $p$ and $q$ in the approximating functionals
$\Functepspq$, {in the limiting functional $\Functlim$},
and also in $\Functepspq^M$ for every $M\in(0,\infty]$.
Thus, we will simply write $\Functeps$, {$\calF$} and $\Functeps^M$, respectively.
{Similarly, we will write $f_{\eps}$ for $f_{\eps,p,q}$.}

\subsection{Characterizations of the surface energy densities}\label{ss:g}

In this section we prove several structural properties of  the surface energy densities $\gpinf$ and $\gpsup$.

{We first show that they are well-defined, in the sense that they do not depend on the choice of $Q^\nu$.  This is a standard argument, well known to experts. 
We briefly sketch it here to show that the interaction of the rescaling with the sequence $M_j$ does not cause any problem.
\begin{lemma}\label{lemmaQnuQnup}
 The definitions of $\gpsup$ and $\gpinf$ do not depend on the choice of the cube $Q^\nu$.
\end{lemma}}
\begin{proof}
We deal explicitly with $\gpinf$, the case  of $\gpsup$ is similar but simpler, as one does not have to scale $M_j$.
Fix $\nu\in S^{n-1}$, $z\in \R^m$, and two possible choices of the cube, call them $Q^\nu$ and $Q_*^\nu$ for definiteness.
They differ by a rotation around the axis $\nu$.
For clarity, in this proof we denote the two corresponding expressions for $\gpinf$ by 
$\gpinf(z,\nu; Q^\nu)$ and 
$\gpinf(z,\nu;{Q_*^\nu})$, respectively.
We intend to show that $\gpinf(z,\nu; Q_*^\nu)\le \gpinf(z,\nu;{ Q^\nu})$.

By definition of $\gpinf$ there are sequences $u_j$, $v_j$, $\eps_j\to0$, $M_j\to\infty$ such that
\begin{equation}\label{eqQnuQnuchoice}
\lim_{j\to\infty} \calF_{\eps_j}^{M_j}(u_j,v_j;Q^\nu)
= \gpinf(z,\nu; Q^\nu),
 \hskip5mm
\lim_{j\to\infty} \|u_j-z\chi_{\{x\cdot\nu>0\}}\|_{L^1(Q^\nu)}=0.
\end{equation}
Let $\lambda\in(0,1)$. Let $\{x_k\}_{k=1, \dots, K_\lambda}$ be
points with $x_k\cdot \nu=0$ and such that the cubes
$q_k:=x_k+\lambda Q_*^\nu$ are pairwise disjoint and contained in $Q^\nu$. If one chooses the points on a regular grid aligned with $Q_*^\nu$, one can ensure that
\begin{equation}\label{eqQnuQnuKlam}
 \lim_{\lambda\to0}\lambda^{n-1}K_\lambda=1.
\end{equation}
From \eqref{eqQnuQnuchoice} we obtain that for any fixed $\lambda$
\begin{equation}
{\limsup_{j\to\infty}} \sum_{k=1}^{K_\lambda} \left[ \calF_{\eps_j}^{M_j}(u_j,v_j;q_k)-\frac{\gpinf(z,\nu;{Q^\nu})}{K_\lambda} \right] {\leq} 0,
\end{equation}
therefore, we can choose $k_*\in\{1,\dots, K_\lambda\}$ such that
\begin{equation}
 \limsup_{j\to\infty} \left[\calF_{\eps_j}^{M_j}(u_j,v_j;q_{k_*})-\frac{\gpinf(z,\nu;Q^\nu)}{K_\lambda} \right] \le0.
\end{equation}
{We also obtain that
$\displaystyle \lim_{j\to\infty}\|u_j-z\chi_{\{x\cdot\nu>0\}}\|_{L^1(q_{k_*})}=0$
by the second of \eqref{eqQnuQnuchoice} and $q_{k_*}\subset Q^\nu$.}
We define $u^{(\lambda)}_j(x):=u_j(x_{k_*}+\lambda x)$, 
$v^{(\lambda)}_j(x):=v_j(x_{k_*}+\lambda x)$, 
$\eps^{(\lambda)}_j:=\eps_j/\lambda $,
$M^{(\lambda)}_j:=M_j/\lambda $.
Then, as $j\to\infty$, one has
$\eps^{(\lambda)}_j\to0$,
$M^{(\lambda)}_j\to\infty$,
$u^{(\lambda)}_j\to z\chi_{\{x\cdot\nu>0\}}$ in $L^1(Q_*^\nu;{\R^m})$,
and a simple scaling, using that $\Psi_\infty$ is positively $q$-homogeneous, gives
\begin{equation*}\begin{split}
 \limsup_{j\to\infty} \calF_{\eps_j^{(\lambda)}}^{M^{(\lambda)}_j}(u^{(\lambda)}_j,v^{(\lambda)}_j;Q_*^\nu)= &
 \limsup_{j\to\infty} \lambda^{1-n}\calF_{\eps_j}^{M_j}(u_j,v_j;q_{k_*})\\\le &
 \frac{1}{\lambda^{n-1}K_\lambda}\gpinf(z,\nu;{Q^\nu}),
\end{split}\end{equation*}
implying
\begin{equation*}
 \gpinf(z,\nu;{Q_*^\nu})\le
 \frac{1}{\lambda^{n-1}K_\lambda}
  \gpinf(z,\nu;{Q^\nu}).
\end{equation*}
Since $\lambda$ was arbitrary, with \eqref{eqQnuQnuKlam} the proof is concluded.
\end{proof}

We start by showing that test sequences in the definitions of $\gpinf$ in \eqref{eqdefGpsnu inf} and of $\gpsup$ in
\eqref{eqdefGpsnu sup} may be taken converging in $L^q$ and satisfying periodic boundary conditions in
$(n-1)$ directions orthogonal to $\nu$ and mutually orthogonal to each other, analogously to what was established in
\cite[Section~3.1]{ContiFocardiIurlano2022} for $p=1$. This is the content of the next two propositions.
We fix a mollifier $\varphi_1\in C^\infty_c(B_1{;[0,\infty)})$, with $\int_{B_1}\varphi_1\dx=1$ {and $\varphi_1(-x)=\varphi_1(x)$,} and set
$\varphi_\eps(x):=\eps^{-n}\varphi_1(x/\eps)$.

\begin{proposition}\label{proplbboundarysingle}
{There are
$\eps_0>0$, $\gamma_1>0$, $\gamma_2>0$, depending only on $p$ and $q$, such that the following holds.}
Let $(u,v)\in L^q(Q^\nu;\R^{m+1})$,
$\eps\in(0,{\eps_0})$, $M\in(0,\infty]$, and let {$(z,\nu)\in \R^m\times S^{n-1}$}.
Then there exists $(u^*,v^*)$ such that
\begin{equation*}
\calF_{\eps}^M(u^*,v^*; Q^\nu)\le
\calF_{\eps}^M(u,v; Q^\nu)
+  {C d_q^{\gamma_1}+
C \eps^{\gamma_2}},
\end{equation*}
where
$d_q:=
\|u-{z}\chi_{\{x\cdot\nu>0\}}\|_{L^q(Q^\nu)}$,
and
\begin{equation}\label{eqpropboundrybv}
u^*=({z}\chi_{\{x\cdot\nu>0\}})\ast\varphi_{\eps}\,,\hskip1cm
v^*=\chi_{\{|x\cdot\nu|\ge2\eps\}}\ast\varphi_{\eps} \hskip1cm
\text{ on } \partial Q^\nu\,.
\end{equation}
The constant $C$ may depend on $q$, $p$, $\ell$, $z$, $\nu$,  and $\Psi$, but not on $u$, $v$, $M$, and $\eps$.
\end{proposition}
\begin{proof}
{\bf Step~1. Construction of $u^*$ and $v^*$.}
We write
\begin{equation}\label{eqdefUjVj}
U:=({z} \chi_{\{x\cdot\nu>0\}})\ast\varphi_{\eps}, \hskip1cm
V:=\chi_{\{|x\cdot\nu|\ge2\eps\}}\ast\varphi_{\eps}.
\end{equation}
Obviously, $\|U-{z}\chi_{\{x\cdot\nu>0\}}\|_{L^q(Q^\nu)}^q\le C\eps$,
with $C>0$ depending on ${z}$, so that $ \|u-U\|_{L^q(Q^\nu)}^q\le C(\eps + d_q^q)$.
Moreover, by construction {$\nabla U=0$ if $|x\cdot\nu|\ge\eps$,}
$V=0$ if $|x\cdot\nu|\le\eps$, and $V=1$ if $|x\cdot\nu|\ge3\eps$.
Therefore, by $\Psiinfty(0)=0$ and $f_p(0)=0$, we have
\[
\calF_{\eps}^M(U,V;Q^\nu){=}\calF_{\eps}^M(0,V;Q^\nu)
\leq C+\eps^{q-1}\int_{\{x\in Q^\nu:\,\eps< |x\cdot\nu|<3\eps\}}|\nabla V|^q\dx\leq C\,,
\]
as $\|\nabla V\|_{L^{\infty}(\R^n)}\leq \sfrac C\eps$.
This implies that there is $c_0>0$ such that if
$d_q\vee \calF_{\eps}^M(u,v; Q^\nu)\ge c_0$,
then the pair $(U,V)$ will give the desired conclusion.
Therefore, in the rest of the proof we can assume that
\begin{equation*}
d_q\le c_0\text{ and } \calF_{\eps}^M(u,v; Q^\nu)\le c_0 .
\end{equation*}
Next, we fix $\eta\in [\eps^{1/q},{\sfrac1{10}}]$
and set
$K:=\lfloor 4\eta^q/\eps\rfloor\ge4$,
{so that $(2K+1)\eps<1$ and $K\eps\ge \eta^q$.}
We let $R_k:=Q^\nu_{1-k\eps}\setminus Q^\nu_{1-(k+1)\eps}$, where, as usual, $Q^\nu_r=rQ^\nu$ is the
scaled cube.
We select $k\in \{K+1,\dots, 2K\}$ such that
\begin{equation}\label{eqchoice1}
\|u-U\|_{L^q(R_k)}^q\le \frac C{K}\|u-U\|_{L^q(Q^\nu)}^q
\le C \frac{\eps+d_q^q}{K}\,,
\end{equation}
and
\begin{equation}\label{eqchoice2}
\calF_{\eps}^M(u,v;R_k)
+\calF_{\eps}^M(U,V;R_k)
\le \frac C{K}  .
\end{equation}
We fix $\theta\in C^1_c(Q^\nu_{1-k\eps};{[0,1]})$ with
$\theta=1$ on $Q^\nu_{1-(k+1)\eps}$
and $|\nabla \theta|\le 3/\eps$,
and define
\begin{equation*}
u^*:= \theta u  + (1-\theta) U.
\end{equation*}
The construction of $v^*$ is more complex. In the interior region, it should match $v$, and in the exterior region, $V$. In the
interpolation region $R_k$, which has volume of order $\eps$, it should not be larger than $v$ and $V$, but also
not larger than $1-\eta$.
The {parameter $\eta\in(0,{\sfrac{1}{10}})$} will be chosen so that one can obtain a good estimate for the integral of
$\eps^{q-1}f_p^q(v)\Psi_\infty(\nabla u^*)$ over $R_k$.
Therefore, we first define
\begin{equation}\label{eqdefhatv}
v_\mathrm{mid}(x):= \min\{1, 1-\eta+\frac1{\eps} \dist(x, R_k)\},
\end{equation}
which coincides with $1-\eta$ in the interpolation region $R_k$, and with $1$ at distance larger than $\eta\eps$ from it, in particular inside $Q_{1-(k+2)\eps}^\nu$ {and outside $Q_{1-(k-1)\eps}^\nu$}. Then, {we} set
\begin{equation}\label{eqdefhatvv}
v_\mathrm{out}(x):= \min\{1, V(x)+\frac{2}{\eps} \dist(x, Q^\nu\setminus Q^\nu_{1-(k+1)\eps})\}
\end{equation}
which coincides with $V$ outside $Q^\nu_{1-(k+1)\eps}$, and with 1 inside $Q^\nu_{1-(k+2)\eps}$ 
as well as for $|x\cdot\nu|\geq3\eps$, and finally
\begin{equation}
v_\mathrm{in}{(x)}:=\min\{1,v{(x)}+\frac2{k\eps} \dist(x, Q^\nu_{1-k\eps})\}\,,
\end{equation}
which is equal to $v$ on $Q^\nu_{1-k\eps}$ and to $1$ at distance larger than $\frac {k\eps}2$ from it.
We then combine these three ingredients to obtain
\begin{equation*}
v^*:=\min\{v_\mathrm{in}, v_\mathrm{mid}, v_\mathrm{out}\}.
\end{equation*}
On $\partial Q^\nu$,  both $v_\mathrm{in}$ and $v_\mathrm{mid}$ equal $1$,
hence $v^*=v_\mathrm{out}=V$ on {this} set.
\smallskip

\noindent{\bf Step~2. Estimate of the elastic energy.} By the definition of $u^*$,
\begin{equation*}
|\nabla u^*|\le |\nabla u|+|\nabla U| + \frac{3}{\eps} |u-U|
\end{equation*}
therefore, by \eqref{e:Psiinftygc}, in $R_k$ we have
\begin{equation*}
{\Psiinfty}(\nabla u^*)\le C{\Psiinfty}(\nabla u)+C{\Psiinfty}(\nabla U)
+ \frac{C}{\eps^q} |u-U|^q\,.
\end{equation*}
Set $F(t):=M^{q-1}\wedge\eps^{q-1}f_p^q(t)$ for brevity. One easily checks that $F$ is increasing, with $F(0)=0$. {We observe that $v^*= \min\{v, 1-\eta, V\}$ in $R_k$, and
since, by definition, $V=0$ on ${R_k}\cap\{|x\cdot\nu|\leq\eps\}$, we get
$v^*=V=0$ on ${R_k}\cap\{|x\cdot\nu|\leq\eps\}$. Moreover, as} $\nabla U=0$
on ${R_k}\cap\{|x\cdot\nu|\geq\varepsilon\}$, the term $F(v^*)\Psiinfty(\nabla U)$ vanishes in $R_k$. Therefore, using $F(v^*)\le F(v_\mathrm{mid})\le \eps^{q-1}\ell^q/\eta^{pq}$ in $R_k$ we have
\begin{equation*}
F(v^*)\Psiinfty(\nabla u^*)
\le C F(v) \Psiinfty(\nabla u)
+C \frac{|u-U|^q}{\eps \eta^{pq}}.
\end{equation*}
Integrating over $R_k$ and using \eqref{eqchoice2} in the first term,
and \eqref{eqchoice1} in the second one,
\begin{equation*}
\int_{R_k}F(v^*) \Psiinfty(\nabla u^*)\dx
\le \frac C{K}+ C\frac {\eps+d_q^q}{K\eps \eta^{pq} }.
\end{equation*}
Recalling that the definition of $K$ implies $ K\eps\ge\eta^q$,
\begin{equation*}
\int_{R_k}F(v^*) \Psiinfty(\nabla u^*)\dx
\le \frac CK + C \frac
{\eps+d_q^q}{\eta^{q(p+1)}}.
\end{equation*}
Using that {$u^*=U$ and $v^*\leq v_{\mathrm{out}}= V$ in $Q^\nu\setminus Q^\nu_{1-k\eps}$,  that }$V=0$ on $Q^\nu\cap\{{|}x\cdot\nu{|}\leq\eps\}$, and {$\nabla U=0$} on $Q^\nu\cap\{|x\cdot\nu|\geq\eps\}$, we have
\begin{equation}\label{eq:nablaU}
\int_{Q^\nu\setminus Q^\nu_{1-k\eps}}F(v^*) {\Psiinfty}(\nabla u^*)\dx
{=0.}
\end{equation}
In $Q^\nu_{1-(k+1)\eps}$, instead, we have $u^*=u$ and $v^*\le {v_{\mathrm{in}}{=}} v$, so that
\begin{equation*}
\int_{Q^\nu_{1-(k+1)\eps}} F(v^*) {\Psiinfty}(\nabla u^*)\dx
\le\int_{Q^\nu_{1-(k+1)\eps}} F(v) {\Psiinfty}(\nabla u)\dx.
\end{equation*}
Adding terms,
\begin{equation}\label{eqfinelastcu}
\int_{Q^\nu}F(v^*) \Psiinfty(\nabla u^*)\dx
\le  \int_{Q^\nu}F(v) \Psiinfty(\nabla u)\dx+\frac CK + C \frac
{\eps+d_q^q}{\eta^{q(p+1)}} .
\end{equation}
\smallskip

\noindent{\bf Step~3. Estimate of the energy of the phase field.} By the definition of $v^*$,
\begin{equation}\label{eqcalFst0}
\calF_{\eps}^M(0,v^*; Q^\nu)
\le
\calF_{\eps}^M(0,v_\mathrm{in}; Q^\nu)+
\calF_{\eps}^M(0,v_\mathrm{mid}; Q^\nu)+
\calF_{\eps}^M(0,v_\mathrm{out}; Q^\nu).
\end{equation}
From \eqref{eqdefhatv} we have $|1-v_\mathrm{mid}|\le \eta$
with $|\{v_\mathrm{mid}\ne 1\}|\le C\eps$ and
$|\nabla  v_\mathrm{mid}|\le 1/\eps$
with $|\{\nabla v_\mathrm{mid}\ne 0\}|\le C\eps\eta$, so that
\begin{equation*}
\calF_{\eps}^M(0,v_\mathrm{mid}; Q^\nu)=
\int_{Q^\nu}\Big(\frac{(1-v_\mathrm{mid})^{q'}}{q'q^{\sfrac{q'}q}\eps}
+\eps^{q-1}|\nabla v_\mathrm{mid}|^q\Big)
\dx \le C \eta .
\end{equation*}
From the definition of $V$ and $v_\mathrm{out}$,
we see that $|\{v_\mathrm{out}\ne 1\}|\le C \eps\cdot K\eps$
and $\eps|\nabla v_\mathrm{out}|\le C$, so that
\begin{equation*}
\calF_{\eps}^M(0,v_\mathrm{out}; Q^\nu)\le C K\eps.
\end{equation*}
Let $c_q>0$ be such that for every $\delta\in(0,1]$
and every $a,\,b\in\R^n$ one has $|a+b|^q\leq(1+\delta)|a|^q
+c_q \delta^{1-q} |b|^q$. As
$v_\mathrm{in}=v$ in ${Q^\nu_{1-k\eps}}$, $v_\mathrm{in}\ge v$,
and $|\nabla v_\mathrm{in}|\le |\nabla v|+2/(k\eps)$ in $Q^\nu \setminus Q^\nu_{1-k\eps}$
\begin{align*}
\calF_{\eps}^M(0,v_\mathrm{in}; Q^\nu)&
\le(1+\delta)\calF_{\eps}^M(0,v; Q^\nu)
+{\frac{c_q\eps^{q-1}}{\delta^{q-1}(k\eps)^q}}\calL^n(Q^\nu\setminus Q^\nu_{1-k\eps})\\
&\le{(1+\delta)}\calF_{\eps}^M(0,v; Q^\nu)+\frac{c_q}{\delta^{q-1}k^{q-1}}.
\end{align*}
Recalling that $k\ge {K+1}$ and $K\eps\le 4\eta^q\le 4\eta$,
\eqref{eqcalFst0} leads to
\begin{equation*}
\calF_{\eps}^M(0,v^*; Q^\nu)
\le{(1+\delta)}
\calF_{\eps}^M(0,v; Q^\nu)+C\eta+\frac{c_q}{\delta^{q-1}K^{q-1}}
\end{equation*}
so that choosing $\delta=K^{-1/q'}$
\begin{equation}\label{eqfinpfcu}
\calF_{\eps}^M(0,v^*; Q^\nu)
\le
\calF_{\eps}^M(0,v; Q^\nu)+C\eta+\frac{C}{K^{1/q'}}.
\end{equation}
\noindent{\bf Step~4. Conclusion of the proof.}
{Combining
\eqref{eqfinpfcu}}
with \eqref{eqfinelastcu},
using
$1/K\le 1/K^{1/q'}$
 and $K\ge \eta^q/\eps$ to eliminate irrelevant terms
leads to
\begin{equation*}
\calF_{\eps}^M(u^*,v^*; Q^\nu)\le
\calF_{\eps}^M(u,v; Q^\nu)
+ C \frac{\eps+d_q^q}{\eta^{q(p+1)}}+C\eta
+{C
\left(\frac{\eps}{\eta^{q}}\right)^{1/q'}}
\end{equation*}
{for any $\eta\in [\eps^{1/q},1/10]$. It remains to choose $\eta$.
If $d_q\le \eps^{1/q}$, {using $\eta^q\le \eta$,} the error term is controlled by
\begin{equation*}
 \frac{\eps}{\eta^{q(p+1)}}
 +\eta
+
\left(\frac{\eps}{\eta^{q}}\right)^{1/q'}
 \le
 \eta+2
\left(\frac{\eps}{\eta^{q(p+1)}}\right)^{1/q'};
\end{equation*}
we pick $\eta:=\eps^{1/(q'+q(p+1))}\ge\eps^{1/q}$
and conclude the proof.
Similarly, if instead
$d_q>\eps^{1/q}$,
\begin{equation*}
 \left(\frac{d_q}{\eta^{p+1}}\right)^q
 +\eta
+
\left(\frac{d_q}{\eta}\right)^{q/q'}
 \le
 \eta+2
\left(\frac{d_q}{\eta^{p+1}}\right)^{q/q'},
\end{equation*}
we conclude with $\eta:=d_q^{q/(q'+q(p+
1))}\wedge \sfrac1{10}$.
}
\end{proof}
 
We are now ready to show equivalent characterizations of $\gpinf$ and $\gpsup$.
We follow the arguments developed in \cite[Sections~3.1 and 3.2]{ContiFocardiIurlano2022}.

\begin{proposition}\label{p:gp boundary value}
For every $(z,\nu)\in\R^m\times S^{n-1}$
\begin{equation}\label{e:gpinf periodicity}
\begin{split}
\gpinf(z,\nu) =\inf \{\liminf_{j\to\infty}&
\calF^{{\cutoffconst_j}}_{\eps_j}(u_j,v_j; Q^\nu):\, \|u_j- z\chi_{\{x\cdot\nu>0\}}\|_{L^q(Q^\nu)}\to0, \\
&\eps_j\to 0,\,\cutoffconst_j\to\infty,\,\textrm{$(u_j,v_j)$ satisfy \eqref{eqpropboundrybv}}\}
\end{split}\end{equation}
and
\begin{equation}\label{e:gpsup periodicity}
\begin{split}
\gpsup(z,\nu) =\inf \{\liminf_{j\to\infty}&
\calF^\infty_{\eps_j}(u_j,v_j; Q^\nu):\, \|u_j- z\chi_{\{x\cdot\nu>0\}}\|_{L^q(Q^\nu)}\to0, \\
&\eps_j\to 0,\,\textrm{$(u_j,v_j)$ satisfy \eqref{eqpropboundrybv}}\}.
\end{split}\end{equation}

\end{proposition}
\begin{proof}
We give the proof only for $\gpinf$, the one for $\gpsup$ being completely analogous.

{{\bf Step~1. Reduction to an optimal sequence in \eqref{eqdefGpsnu inf} converging in $L^q(Q^\nu;\R^{m+1})$.}
Let $\eps_j\to 0$, {$\cutoffconst_j\to\infty$}, and
$(u_j,v_j)\to (z\chi_{\{x\cdot\nu>0\}},1)$ in $L^1(Q^\nu;\R^{m+1})$ be such that
\begin{equation*}
{\gpinf(z,\nu)}
 = \lim_{j\to\infty}\calF^{{\cutoffconst_j}}_{\eps_j}(u_j,v_j; Q^\nu).
\end{equation*}
Recall that $v_j\in[0,1]$ $\calL^n$-a.e. in $Q^\nu$; therefore $v_j\to 1$ in $L^q(Q^\nu)$.
By the definition of $\mathcal{T}_k$ in \eqref{e:Tk}, we claim that for all $j,\,N\in\N$
there is $k_{N,j}\in\{N+1,\ldots,2N\}$ such that
\begin{equation}\label{e:truncation g}
\calF^{{\cutoffconst_j}}_{\eps_j}(\mathcal{T}_{k_{N,j}}(u_j), v_j; Q^\nu)\leq\Big(1+\frac
CN\Big)\calF^{{\cutoffconst_j}}_{\eps_j}(u_j, v_j; Q^\nu)\,,
\end{equation}
where $C>0$ is a constant independent of $N$ and $j$.
If $a_N>{|z|=\|z\chi_{\{x\cdot\nu>0\}}\|_{L^\infty(Q^\nu)}}$ then
${\mathcal{T}_{k_{N,j}}}(u_j)\to z\chi_{\{x\cdot\nu>0\}}$ in $L^q(Q^\nu;\R^m)$, and \eqref{e:truncation g} yields
\[
\limsup_{j\to\infty}\calF^{{\cutoffconst_j}}_{\eps_j}(\mathcal{T}_{k_{N,j}}(u_j), v_j; Q^\nu)\leq\Big(1+\frac
CN\Big){\gpinf(z,\nu)} \,,
\]
which, by the arbitrariness of $N\in\N$, implies
\begin{equation*}\begin{split}
{\gpinf(z,\nu)} =\inf &\{\liminf_{j\to\infty}
\calF^{{\cutoffconst_j}}_{\eps_j}(u_j,v_j; Q^\nu): \\
&\|u_j- z\chi_{\{x\cdot\nu>0\}}\|_{L^q(Q^\nu;\R^m)}\to0, \eps_j\to 0,
{\cutoffconst_j\to\infty}
\}.
\end{split}\end{equation*}
We are left with establishing \eqref{e:truncation g}. To this aim, consider $\mathcal{T}_k(u_j)$ and note that
\begin{align}\label{e:Feps* Tk}
\calF^{{\cutoffconst_j}}_{\eps_j}&(\mathcal{T}_k(u_j),v_j;Q^\nu)=\calF^{{\cutoffconst_j}}_{\eps_j}(u_j,v_j;\{|u_j|\leq
a_k\})\notag\\
&+\calF^{{\cutoffconst_j}}_{\eps_j}(\mathcal{T}_k(u_j),v_j;\{a_k<|u_j|< a_{k+1}\})
+\calF^{{\cutoffconst_j}}_{\eps_j}(0,v_j;\{|u_j|\geq a_{k+1}\})\,.
\end{align}
We estimate the second term in \eqref{e:Feps* Tk}. The growth conditions on {$\Psi_\infty$ (cf. \eqref{e:Psiinftygc})}
and
{$\Lip(\mathcal{T}_k)\leq 1$} yield for a constant $C>1$
\begin{align}\label{e:Feps* Tk 2}
\calF^{{\cutoffconst_j}}_{\eps_j}&(\mathcal{T}_k(u_j),v_j;\{a_k<|u_j|<a_{k+1}\})
\notag\\&\leq C\int_{\{a_k<|u_j|< a_{k+1}\}}
{(\cutoffconst_j^q\wedge \eps_j^{q-1} f^q_{p}(v_j))}\Psiinfty(\nabla u_j)\dx\notag\\
&+\mathcal F^{{\cutoffconst_j}}_{\eps_j}(0,v_j;\{a_k<|u_j|< a_{k+1}\})\,.
\end{align}
Collecting \eqref{e:Feps* Tk} and \eqref{e:Feps* Tk 2}
and using that for $A$ and $B$ disjoint
$\calF_{\eps_j}^{{\cutoffconst_j}}(u_j,v_j;A)+\calF_{\eps_j}^{{\cutoffconst_j}}(0,v_j;B)\le
\calF_{\eps_j}^{{\cutoffconst_j}}(u_j,v_j;A\cup B)$,
we conclude that
\begin{equation*}
\calF^{{\cutoffconst_j}}_{\eps_j}(\mathcal{T}_k(u_j),v_j;Q^\nu)\leq
\calF^{{\cutoffconst_j}}_{\eps_j}(u_j,v_j;Q^\nu)
+ C\calF^{\cutoffconst_j}_{\eps_j}(u_j,v_j;\{a_k<|u_j|<a_{k+1}\})\,.
\end{equation*}
Let now $N\in\N$. By averaging, there exists $k_{N,j}\in\{N+1,\ldots,2N\}$ such that
\begin{align}\label{e:truncation}
\calF^{{\cutoffconst_j}}_{\eps_j}(\mathcal{T}_{k_{N,j}}(u_j),v_j;Q^\nu)&\leq\frac
1N\sum_{k=N+1}^{2N}\calF^{{\cutoffconst_j}}_{\eps_j}(\mathcal{T}_k(u_j),v_j;Q^\nu)
\notag\\
&\leq\Big(1+\frac CN\Big)\calF^{{\cutoffconst_j}}_{\eps_j}(u_j,v_j;Q^\nu)\,,
\end{align}
i.e. \eqref{e:truncation g}.
 \smallskip

 \noindent{\bf Step~2. {Construction of an optimal sequence converging in $L^q(Q^\nu;\R^{m+1})$ and satisfying \eqref{eqpropboundrybv}}.}
 In view of Step~1 there is an optimal sequence for $\gpinf(z,\nu)$ in
 \eqref{eqdefGpsnu inf} converging in $L^q(Q^\nu;\R^{m+1})$.}
 Therefore, we may apply Proposition~\ref{proplbboundarysingle} to modify the sequence to satisfy 
  \eqref{eqpropboundrybv} without increasing the limiting energy, and conclude.
\end{proof}
We show next that the infimum in the definition of $\gpsup$ is achieved for any infinitesimal sequence $\eps_j$ 
(cf. \cite[Proposition~3.2]{ContiFocardiIurlano2022}).
	\begin{proposition}\label{periodicity}
		For any $(z,\nu)\in \R^m\times S^{n-1}$ and
		any $\eps_j^*\downarrow0$ there is $(u_j^*,v_j^*)\to (z\chi_{\{x\cdot\nu>0\}},1)$ in $L^q(Q^\nu;\R^{m+1})$,
		with $v_j^*\in[0,1]$ $\calL^n$-a.e. in $Q^\nu$, and satisfying \eqref{eqpropboundrybv}, such that
		\begin{equation}\label{e:g periodic sequence}
		\lim_{j\to\infty}
		\calF^{\infty}_{\eps_j^*}({{u_j^*,v_j^*}}; Q^\nu)=  \gpsup(z,\nu)\,.
		\end{equation}
	\end{proposition}
	\begin{proof}
We first apply  Proposition~\ref{p:gp boundary value} to infer that for some $\eps_j\to 0$ and
$(u_j,v_j)\to (z\chi_{\{x\cdot\nu>0\}},1)$ in $L^q(Q^\nu;\R^{m+1})$ satisfying \eqref{eqpropboundrybv} it holds
\begin{equation}\label{e:gsup=lim Fepsj}
\gpsup(z,\nu)= \lim_{j\to\infty}\calF^{\infty}_{\eps_j}(u_j,v_j; Q^\nu)
\end{equation}
(cf. \eqref{e:gpsup periodicity}).
Since $\displaystyle{\lim_{k\to\infty}\lim_{j\to\infty}} \eps_j^*/\eps_k=0$, we can
select a nondecreasing sequence $k(j)\to\infty$ such that {$\lambda_j:=\eps_j^*/\eps_{k(j)}\to0$.}
We let $\tilde Q^\nu:=(\Id_n-\nu\otimes \nu)Q^\nu\subset \nu^\perp\subset\R^n$ {and}
select $x_1, \dots, x_{I_j}\in \tilde Q^\nu$, with $I_j:=\lfloor1/\lambda_j\rfloor^{n-1}$,
such that $x_i+\tilde Q^\nu_{\lambda_j}$ are pairwise disjoint subsets of $\tilde Q^\nu$. We set
\begin{equation*}
u_j^*(x):=\begin{cases}
u_{k(j)}(\frac{x-x_i}{{\lambda_j}}), & \text{ if } x-x_i\in Q^\nu_{{\lambda_j}} \text{ for some $i$},\\
{U_j^*}(x), & \text{ otherwise in $Q^\nu$}
\end{cases}
\end{equation*}
and
\begin{equation*}
v_j^*(x):=\begin{cases}
v_{k(j)}(\frac{x-x_i}{{\lambda_j}}), & \text{ if } x-x_i\in Q^\nu_{{\lambda_j}} \text{ for some $i$},\\
{V_j^*}(x),  & \text{ otherwise in $Q^\nu$},
\end{cases}
\end{equation*}
where {$U_j^*$ and $V_j^*$} are defined as in \eqref{eqdefUjVj} using $\eps_j^*$. One easily verifies that
$U_j^*(x)=U_{k(j)}(\frac{x-y}{\lambda_j})$ for all $y\in \nu^\perp$, and the same for $V_j^*$.
By the boundary conditions \eqref{eqpropboundrybv}, these functions are continuous and therefore in
$W^{1,q}(Q^\nu;\R^{m+1})$. We further estimate
\begin{equation*}
\calF^{\infty}_{\eps_j^*}(u_j^*, v_j^*; Q^\nu)\le
I_j \lambda_j^{n-1} \calF^{\infty}_{\eps_{k(j)}}(u_{k(j)}, v_{k(j)}; Q^\nu)
+ C \calH^{n-1}({\tilde  Q}^\nu\setminus \cup_i (x_i+{\tilde  Q} ^\nu_{\lambda_j}){)}.
\end{equation*}
Taking $j\to\infty$ and recalling \eqref{e:gsup=lim Fepsj}
concludes the proof.
\end{proof}

{In what follows we provide a further equivalent characterization for the surface energy density
$\gpsup$ in the spirit of \cite[Proposition~4.3]{ContiFocardiIurlano2016} (see also \cite[Proposition
3.3]{ContiFocardiIurlano2022}).
\begin{proposition}\label{p:chargT}
For any $(z,\nu)\in \R^m\times S^{n-1}$ one has
\begin{equation}\label{e:characg}
\gpsup(z,\nu) =\lim_{T\to\infty}\inf_{(u,v)\in \calU_{z,\nu}^T}
\frac{1}{T^{n-1}}\calF^\infty_1(u,v;Q^\nu_T)\,,
\end{equation}
where
\begin{multline}\label{e:uznu}
\calU^T_{z,\nu}:=\Big\{(u,v)\in W^{1,q}(Q^\nu_T;\R^{m+1})\colon 0\leq v\leq 1,\
v=\chi_{\{|x\cdot\nu|\ge 2\}}\ast\varphi_{1}\, \text{ and }\\
u=(z\chi_{\{x\cdot\nu>0\}})\ast\varphi_{1}
\text{ on } \partial Q^\nu_{{T}} \Big\}.
\end{multline}
\end{proposition}
\begin{proof}
For every $(z,\nu)\in \R^m\times S^{n-1}$ and $T>0$ set
\begin{equation}\label{eqdefgTznu}
g_T(z,\nu):=\inf_{(u,v)\in \calU_{z,\nu}^T}\frac{1}{T^{n-1}}\calF^\infty_1(u,v;Q^\nu_T)\,.
\end{equation}
We first prove that
\begin{equation}\label{e:limsup leq g}
\limsup_{T\to\infty}g_T(z,\nu)\leq \gpsup(z,\nu) \,.
\end{equation}
If $T_j\uparrow\infty$ is a sequence achieving the $\limsup$ on the left{-}hand side above,
thanks to Proposition~\ref{periodicity} we may
consider $(u_j,v_j)\in W^{1,q}(Q^\nu;\R^{m+1})$ with $0\leq v_j\leq 1$, $(u_j,v_j)\to (z\chi_{\{x\cdot\nu>0\}},1)$ in
$L^q(Q^\nu;\R^{m+1})$,
\begin{equation}\label{e:bound}
u_j=(z\chi_{\{x\cdot\nu>0\}})\ast\varphi_{\frac1{T_j}}\,,\qquad
v_j=\chi_{\{|x\cdot\nu|\ge \frac2{T_j}\}}\ast\varphi_{\frac1{T_j}}
\text{ on } \partial Q^\nu,
\end{equation}
and
\begin{equation}\label{e:estf}
\lim_{j\to\infty}\calF^{\infty}_{\sfrac1{T_j}}(u_j,v_j; Q^\nu)=\gpsup(z,\nu) .
\end{equation}
Then, define $(\tilde{u}_j(y),\tilde{v}_j(y)):=\big(u_j(\frac y{T_j}),v_j(\frac y{T_j})\big)$ for $y\in Q^\nu_{T_j}$,
and note that, by a change of variables,
\begin{equation*}
\frac{1}{T_j^{n-1}}{\calF^{\infty}_{1}}(\tilde{u}_j,\tilde{v}_j;Q^\nu_{T_j})=
\calF^{\infty}_{\sfrac1{T_j}}(u_j,v_j;Q^\nu)\,,
\end{equation*}
and that $(\tilde{u}_j,\tilde{v}_j)\in\calU^{T_j}_{z,\nu}$
in view of \eqref{e:bound}. Then, by \eqref{e:estf}, the choice of
$T_j$, and the definition of $g_T(z,\nu)$, we conclude \eqref{e:limsup leq g}.

In order to prove the converse inequality
\begin{equation}\label{e:liminf geq g}
\liminf_{T\to\infty}g_T(z,\nu)\geq \gpsup(z,\nu) \,,
\end{equation}
we assume for the sake of notational simplicity $\nu=e_n$.
We then fix $\rho>0$ and take $T>6$, depending on $\rho$,
and $(u_T,v_T)\in \calU^T_{z,e_n}$ such that
\begin{equation}\label{e:estf2}
\frac{1}{T^{n-1}}\calF^\infty_1(u_T,v_T;Q^{e_n}_T)\leq
\liminf_{T\to\infty}g_T(z,{e_n})+\rho\,.
\end{equation}
Let $\eps_j\to0$ and set
\begin{equation}\label{eqdefujscaleT}
u_j(y):=\left\{ \begin{array}{ll}
{\displaystyle u_T\left(\frac{y}{\eps_{j}}-d\right)},
& \textrm{if } y\in 
\eps_j (Q^{e_n}_T+d)\subset\subset Q^{e_n},\\
{{(z\chi_{\{x\cdot {e_n} >0\}}\ast \varphi_{1})(\frac{y}{\eps_j})}}, & \text{otherwise in }Q^{e_n},
\end{array} \right.
\end{equation}
\begin{equation}\label{eqdefvjscaleT}
v_j(y):=\left\{ \begin{array}{ll}
{\displaystyle v_T\left(\frac{y}{\eps_{j}}-d\right)} ,
& \textrm{if } y\in\eps_j (Q^{e_n}_T+d)\subset\subset Q^{e_n},\\
{{(\chi_{\{|x\cdot {e_n}| >2\}}\ast \varphi_{1})(\frac{y}{\eps_j})}}, & \text{otherwise in }Q^{e_n},
\end{array} \right.
\end{equation}
with $d\in\Z^{n-1}\times\{0\}$.
Then, $(u_j,v_j)\to(z\chi_{\{x\cdot{e_n} >0\}},1)$ in {$L^1(Q^{e_n};\R^{m+1})$}, and {letting $I_{\eps_j}:=\{d\in
\Z^{n-1}\times\{0\}:\,\eps_j (Q^{{e_n}}_T+d)\subset\subset Q^{e_n}\}$}, a change of variables yields
(cf. also the discussion after \eqref{eqdefUjVj})
\begin{align*}
\gpsup(z,{e_n})\leq&\limsup_{j\to\infty}\calF^{\infty}_{\eps_j}(u_j,v_j;Q^{{e_n}})\\
\leq&\limsup_{j\to\infty}\Big(\sum_{d\in I_{\eps_j}}
\calF^{\infty}_{\eps_j}(u_j,v_j;\eps_j(Q^{{e_n}}_T+d))\\
&+{\frac{c}{\eps_j}}\calL^{n}\Big({Q^{{e_n}}\cap}
{\{|x_n|\leq 3\eps_j\}}\setminus \!\!\bigcup_{d\in
I_{{\eps_j}}}\eps_j (Q^{{e_n}}_T+d)\Big)\Big)\\
=& \limsup_{j\to\infty}
\eps_j^{n-1}\#I_{\eps_j}\,\calF^\infty_1(u_T,v_T;Q^{{e_n}}_T) \\
\leq&\frac1{T^{n-1}}\calF^\infty_1(u_T,v_T;Q^{{e_n}}_T)\leq
\liminf_{T\to\infty}g_T(z,{e_n})+\rho\,,
\end{align*}
by the choice of $(u_T,v_T)$ and $T$ {(cf. \eqref{e:estf2})}.
As $\rho\to 0$, we get {\eqref{e:liminf geq g}}.

Estimates \eqref{e:limsup leq g} and \eqref{e:liminf geq g} yield the existence of the limit of $g_T(z,\nu)$ as
$T\uparrow\infty$ and equality \eqref{e:characg}, as well.
\end{proof}

With this representation of $\gpsup$ at hand we can obtain a version of
Proposition~\ref{periodicity} which also accounts for a regularization term
of the form $\eta_\eps\int_\Omega |\nabla u|^q\dx$, {which is useful to prove the
convergence of minimizers stated in the introduction.}
\begin{corollary}\label{propuvjump}
For any $\eps_j\downarrow0$ and $\eta_j\downarrow0$
with $\sfrac{\eta_j}{\eps_j^{q-1}}\to0$, and any $(z,\nu)\in \R^m\times S^{n-1}$
there is $(u_j,v_j)\to (z\chi_{\{x\cdot\nu>0\}},1)$ in
$L^q(Q^\nu;\R^{m+1})$, with $v_j\in[0,1]$ $\calL^n$-a.e.
in $Q^\nu$, such that
\begin{equation*}
\lim_{j\to\infty}
\calF^{\infty}_{\eps_j}({{u_j,v_j}}; Q^\nu) = \gpsup(z,\nu) \,,
\end{equation*}
\[
\lim_{j\to\infty} \eta_j \int_{Q^\nu} |\nabla u_j|^q\dx=0\,,
\]
and
\begin{equation*}
u_j=(z\chi_{\{x\cdot\nu>0\}})\ast\varphi_{\eps_j}\,,\hskip1cm
v_j=\chi_{\{|x\cdot\nu|\ge2\eps_j\}}\ast\varphi_{\eps_j} \hskip1cm
\text{ on } \partial Q^\nu.
\end{equation*}
\end{corollary}
\begin{proof}
We use the same construction as {in \eqref{eqdefujscaleT}-\eqref{eqdefvjscaleT}} (without loss of generality, explicitly written only for
$\nu=e_n$), and compute similarly
\begin{align*}
\|\nabla u_j\|_{L^q(Q^{{e_n}})}^q
&\leq\sum_{d\in I_{\eps_j}}
\|\nabla u_j\|_{L^q(\eps_j(Q^{{e_n}}_T+d))}^q
+\frac{C}{\eps_j^q}\calL^{n}\left(Q^{{e_n}}\cap\{|{x_n}|\leq \eps_j\}
\right)\\
&={\eps_j^{n-q}}\# I_{\eps_j}\,\|\nabla u_T\|_{L^q(Q^{{e_n}}_T)}^q
+\frac{C}{\eps_j^{q-1}}\le \frac{C_T}{\eps_j^{{q-1}}}.
\end{align*}
To conclude the proof, it suffices to choose $T_j\to\infty$ so slowly that
$\eta_j C_{T_j}/\eps_j^{q-1}\to0$.
\end{proof}
We finally prove that $\gpinf$ and $\gpsup$ coincide in the case where $\Psi_\infty$ satisfies the projection property {in \eqref{e:projectionproperty}}. {Before starting the proof, we observe that continuity of $\Psi_\infty$ implies that for any $\delta>0$ there is $C(\delta)>0$ such that
	\begin{equation}\label{eqPsiinftyxieta}
	\Psi_\infty(\xi+\eta)\le (1+\delta)\Psi_\infty(\xi)+C(\delta)|\eta|^q \hskip5mm
	\text{ for all }\xi,\eta\in\R^{m\times n},
	\end{equation}
	which by $q$-homogeneity and coercivity immediately implies, {for $L>0$,}
	\begin{equation}\label{eqPsiinftyxieta2}
	\Psi_\infty(\xi+\eta)\le (1+\delta+C(\delta)L^q)\Psi_\infty(\xi) \hskip5mm
	\text{ for all }\xi,\eta\in\R^{m\times n} \text{ with } |\eta|\le L|\xi|.
	\end{equation}
	By the $q$-homogeneity of $\Psi_\infty$, it suffices to prove \eqref{eqPsiinftyxieta} under the normalization $|\xi+\eta|=1$. If this were false, there would be $\delta>0$ and sequences with
	\begin{equation*}
	\Psi_\infty(\xi_k+\eta_k)> (1+\delta)\Psi_\infty(\xi_k)+k|\eta_k|^q.
	\end{equation*}
	As $\xi_k+\eta_k$ is bounded, we obtain $\eta_k\to0$, and therefore, passing to a subsequence, $\xi_k\to\xi_*$ with $|\xi_*|=1$.
	Continuity of $\Psi_\infty$ at $\xi_*$ then implies
	\begin{equation*}
	\Psi_\infty(\xi_*)\ge (1+\delta)\Psi_\infty(\xi_*),
	\end{equation*}
	a contradiction. This proves \eqref{eqPsiinftyxieta}.}}

\begin{proposition}\label{euclidean}
Let $\Psi_\infty$ satisfy the projection property in \eqref{e:projectionproperty}. Then
\[\gpinf(z,\nu)=\gpsup(z,\nu)=g(z,\nu)\quad \text{ for every }(z,\nu)\in\R^m\times S^{n-1},\]
where $g$ is defined in \eqref{e:Psi sliceable}.

In particular, if $\Psi_\infty(\xi)=|\xi|^q$, then $g(z,\nu)=g_\scal(|z|)$
for all $(z,\nu)\in \R^m\times S^{n-1}$, where for every $s\geq0$
\begin{equation}\label{e:gscal}
g_\scal(s):=\inf_{\mathcal U_s}\int_0^1|1-\beta|
(f_p^q(\beta)|\alpha'|^q+|\beta'|^q)^{\sfrac1q}\,\dd \tau
\end{equation}
and
\[
\mathcal U_s:=\{{(\alpha,\,\beta)}\in W^{1,q}((0,1);\R^2):\,
0\leq\beta\leq 1,\,\beta(0)=\beta(1)=1,\,\alpha(0)=0,\,\alpha(1)=s\}\,.
\]
\end{proposition}
\begin{proof}
Let $(z,\nu)\in \R^m\times S^{n-1}$, $z\ne0$, and $T>0$. It is convenient first to recall some notation
used in the slicing technique and some results, for which we refer to \cite{AFP}:
if $(u,v)\in \calU^T_z$, then for $\calH^{n-1}$-a.e. $y\in \tilde Q^\nu_T:=(\Id_n-\nu\otimes \nu)Q^\nu_T\subset
\nu^\perp$
the slices
\[
u_y^\nu(t):=u(y+t\nu),\quad v_y^\nu(t):=v(y+t\nu)
\]
belong to ${\calV}^T_z$ defined in \eqref{e:utildeznu}.
Moreover, for every $(\alpha,\beta)\in {W^{1,q}_\loc(\R;\R^{m+1})}$,
$\eps>0$,
$A\in\calB(\R)$,
and $M\in(0,\infty]$,
set
\begin{align*}
\widetilde{\calF}^{\cutoffconst}_{\eps}(\alpha,\beta;{A}):=\int_A
\Big(&(\cutoffconst^{q-1}\wedge \eps^{q-1}f_p^q(\beta(t)))\Psi_\infty(\alpha'(t)\otimes\nu)\\ &
+\frac{(1-\beta(t))^{q'}}{q'q^{\sfrac{q'}q}\eps}+\eps^{q-1}|\beta'(t)|^q\Big)\dd t\,.
\end{align*}

\noindent{\bf Step~1. We show that $\gpsup(z,\nu)=g(z,\nu)$.}

To do this, we set
\[
\widetilde{g}_T(z,\nu):=\inf_{{\calV}_z^T}\widetilde{\calF}^\infty_1(\alpha,\beta;(-\sfrac T2,\sfrac
T2))\,,
\]
{and note that \eqref{e:Psi sliceable} can be rewritten as $g(z,\nu)=\displaystyle{\lim_{T\to\infty}\widetilde{g}_T(z,\nu)}$.
On the one hand,
to prove that
\begin{equation}\label{e:g leq gpsup}
\liminf_{T\to\infty}\widetilde{g}_T(z,\nu)\geq \gpsup(z,\nu) \,,
\end{equation}
we assume for the sake of notational simplicity $\nu=e_n$.
We then fix $\rho>0$ and take $T>0$, depending on $\rho$,
and $(\alpha_T,\beta_T)\in \calV^T_{z}$ such that
\begin{equation}\label{e:estf2 bis}
\widetilde{\calF}^{\infty}_1(\alpha_T,\beta_T;(-\sfrac T2,\sfrac T2))\leq
\liminf_{T\to\infty}\widetilde{g}_T(z,e_n)+\rho\,.
\end{equation}
Let $\eps_j\to0$ and set
\[
u_j(y):=\left\{ \begin{array}{ll}
{\displaystyle \alpha_T\left(\frac{y_n}{\eps_{j}}\right)},
& \textrm{if } |y_n|{\leq} \frac{T\eps_j}2,\\
z\chi_{\{x\cdot {e_n} >0\}}, & \text{otherwise in }Q^{e_n},
\end{array} \right.
\]
\[
v_j(y):=\left\{ \begin{array}{ll}
{\displaystyle \beta_T\left(\frac{y_n}{\eps_{j}}\right)} ,
& \textrm{if } |y_n|{\leq}\frac{T\eps_j}2,\\
1, & \text{otherwise in }Q^{e_n}.
\end{array} \right.
\]
Then, $(u_j,v_j)\to(z\chi_{\{x\cdot{e_n} >0\}},1)$ in {$L^1(Q^{e_n};\R^{m+1})$}, and a change of variable yields
\begin{align*}
\limsup_{j\to\infty}& \calF^{\infty}_{\eps_j}(u_j,v_j;Q^{{e_n}})
=\limsup_{j\to\infty}\calF^{\infty}_{\eps_j}(u_j,v_j;\{|y_n|\leq T\eps_j/2\})\\
&= 
\widetilde{\calF}^\infty_1(\alpha_T,\beta_T;(-\sfrac T2,\sfrac T2))
\leq\liminf_{T\to\infty}\widetilde{g}_T(z,{e_n})+\rho\,,
\end{align*}
by the choice of $(\alpha_T,\beta_T)$ and $T$ {(cf. \eqref{e:estf2 bis})}.
As $\rho\to 0$,  we conclude \eqref{e:g leq gpsup}, recalling the definition of $\gpsup$ in \eqref{eqdefGpsnu sup}.
}

On the other hand, for every $(u,v)\in  \calU^T_z$, $T>0$, we use the projection property in \eqref{e:projectionproperty}
and Fubini's theorem to deduce that
\begin{align*}
\calF^\infty_1(u,v;Q^\nu_T)\geq
\int_{\tilde Q^\nu_T}\, \widetilde{\calF}^\infty_1(u_y^\nu,v_y^\nu;(-\sfrac T2,\sfrac T2))\mathrm{d}\calH^{n-1}(y)\,.
\end{align*}
In turn, as $(u_y^\nu,v_y^\nu)\in {\calV}^T_z$ for $\calH^{n-1}$-a.e. $y\in \tilde Q^\nu_T$,
the latter inequality implies $g_T(z,\nu)\geq\widetilde{g}_T(z,\nu)$. {Therefore, Proposition~\ref{p:chargT} yields that
\[
\limsup_{T\to\infty}\widetilde{g}_T(z,\nu)\leq \limsup_{T\to\infty}g_T(z,\nu)=\gpsup(z,\nu) \,,
\]
and the claim in Step~1 follows by also taking into account \eqref{e:g leq gpsup}.}
\smallskip

 \noindent{\bf Step~2. We show that $\gpinf(z,\nu)=\gpsup(z,\nu)$.}

By the trivial inequality in \eqref{e:trivial inequality gpinf gpsup} and Step~1, it is sufficient to show that
$\gpinf(z,\nu)\geq g(z,\nu)$. To this aim, by taking into account {Proposition~\ref{p:gp boundary value}}, we may consider a
sequence $(u_j,v_j)\to (z\chi_{\{x\cdot \nu>0\}},1)$ in $L^q(Q^\nu;\R^{m+1})$ in the definition of $\gpinf$ that 
additionally satisfies \eqref{eqpropboundrybv}
(cf. \eqref{e:gpinf periodicity}). Hence, we may argue as above to deduce that
\begin{align*}
\calF^{\cutoffconst_j}_{\eps_j}(u_j,v_j;Q^\nu)&\geq \int_{\tilde Q^\nu}\,
\widetilde{\calF}^{\cutoffconst_j}_{\eps_j}((u_j)^\nu_y,(v_j)^\nu_y;(-\sfrac12,\sfrac12))\mathrm{d}\calH^{n-1}(y)\\
&\geq \inf_{{\calV}^1_z}\widetilde{\calF}^{\cutoffconst_j}_{\eps_j}(\alpha,\beta;(-\sfrac12,\sfrac12))\,.
\end{align*}
Therefore,
\begin{equation}\label{e:stima gpinf dal basso}
\gpinf(z,\nu)\geq\liminf_{j\to\infty}\inf_{{\calV}^1_z}
\widetilde{\calF}^{\cutoffconst_j}_{\eps_j}(\alpha,\beta;(-\sfrac12,\sfrac12))\,.
\end{equation}
Up to extracting a subsequence not relabeled, we may assume the latter $\liminf$ to be a limit, and consider
$(\alpha_j,\beta_j)\in {\calV}^1_z$ to be such that
\[
\widetilde{\calF}^{\cutoffconst_j}_{\eps_j}(\alpha_j,\beta_j;(-\sfrac12,\sfrac12))
\le
\inf_{{\calV}^1_z}\widetilde{\calF}^{\cutoffconst_j}_{\eps_j}(\alpha,\beta;(-\sfrac12,\sfrac12))
+\frac1j
\,.
\]
Denote by $\lambda_j\in(0,1)$ the unique root in $(0,1)$ of the equation $\frac{\ell
t}{(1-t)^p}=(\sfrac{\cutoffconst_j}{\eps_j})^{\sfrac1{q'}}$. Clearly, $\lambda_j\to 1$ as $j\to\infty$, and more
precisely $(1-\lambda_j)^p(\sfrac{\cutoffconst_j}{\eps_j})^{\sfrac1{q'}}\to\ell$  as $j\to\infty$. Then, as
\[
\sup_j\widetilde{\calF}^{\cutoffconst_j}_{\eps_j}(\alpha_j,\beta_j;(-\sfrac12,\sfrac12))\leq C\,,
\]
using the $q$-homogeneity of $\Psi_\infty$ we infer
\begin{equation}\label{e:variaz alphaj sopralivello betaj bound}
\cutoffconst_j^{q-1}\int_{\{\beta_j>\lambda_j\}}|\alpha_j'|^q\dd t\leq C\,
\end{equation}
so that, by Jensen's inequality, we deduce that
\begin{equation}\label{e:variaz alphaj sopralivello betaj}
\int_{\{\beta_j>\lambda_j\}}|\alpha_j'|\dd t\leq
C\Big(\frac1{\cutoffconst_j}\calL^1(\{\beta_j>\lambda_j\})\Big)^{1-\frac 1q}
\leq C\cutoffconst_j^{\frac 1q-1}\,.
\end{equation}
Define
\[
\eta_j(t):=\int_{-\sfrac12}^t\alpha_j'(x)\chi_{\{\beta_j\leq\lambda_j\}}(x)\dd x\,.
\]
Then $\eta_j\in W^{1,q}((-\sfrac12,\sfrac12);\R^m)$ with $\eta_j(-\sfrac12)=0$.
Clearly, $\eta_j'=\alpha_j'$ $\calL^1$-a.e. on $\{\beta_j\leq\lambda_j\}$ and
 $\eta_j'=0$ $\calL^1$-a.e. on $\{\beta_j>\lambda_j\}$. Therefore,
 the $q$-homogeneity of $\Psi_\infty$ implies 
\begin{align}\label{e:stima energ xij}
\widetilde{\calF}^\infty_{\eps_j}&(\eta_j,\beta_j;(-\sfrac12,\sfrac12))=
\widetilde{\calF}^\infty_{\eps_j}(\alpha_j,\beta_j;\{\beta_j\leq\lambda_j\})
+\widetilde{\calF}^\infty_{\eps_j}(0,\beta_j;\{\beta_j>\lambda_j\})
\notag\\
&{=}\widetilde{\calF}^{\cutoffconst_j}_{\eps_j}(\alpha_j,\beta_j;\{\beta_j\leq\lambda_j\})
+\widetilde{\calF}^{\cutoffconst_j}_{\eps_j}(0,\beta_j;\{\beta_j>\lambda_j\})
\notag\\&\leq
\widetilde{\calF}^{\cutoffconst_j}_{\eps_j}(\alpha_j,\beta_j;(-\sfrac12,\sfrac12))
\leq\inf_{{\calV}^1_z}\widetilde{\calF}^{\cutoffconst_j}_{\eps_j}(\alpha,\beta;(-\sfrac12,\sfrac12))+\frac1j\,.
\end{align}

Finally, 
we observe that $\eta_j(\sfrac12)=z-\int_{\{\beta_j>\lambda_j\}}\alpha_j'\dd t$ and define $e_j:=z-\eta_j(\sfrac12)$.
By \eqref{e:variaz alphaj sopralivello betaj},
we infer that $e_j\to0$
as $j\to\infty$. As
{$|z|\ne 0$, for $j$ sufficiently large we have $|e_j|\le |z|/2$, and we can
choose a matrix $A_j\in \R^{m\times m}$ such that
$A_j(z-e_j)=e_j$ and
$|A_j|\le 2 |e_j|/|z|$. We
set
\begin{equation*}
 \zeta_j:=\eta_j+ A_j \eta_j,
\end{equation*}
check that $\zeta_j\in {\calV}^1_z$,
and estimate with
\eqref{eqPsiinftyxieta2}, for any $\delta>0$,
\begin{equation*}
 \Psi_\infty(\zeta_j'\otimes \nu)
 \le (1+\delta + C(\delta) |A_j|^q) \Psi_\infty(\eta_j'\otimes \nu)
\end{equation*}
and therefore, since $A_j\to0$,
\begin{equation}\label{e:stima enrg zetaj}
\limsup_{j\to\infty}
\widetilde{\calF}^\infty_{\eps_j}(\zeta_j,\beta_j;(-\sfrac12,\sfrac12))\le
(1+\delta)\limsup_{j\to\infty}
\widetilde{\calF}^\infty_{\eps_j}(\eta_j,\beta_j;(-\sfrac12,\sfrac12))
\end{equation}
for any $\delta>0$, and therefore after letting $\delta\to0$.
}
Hence, the existence of the limit defining $g(z,\nu)$ established in Step~1 yields
\begin{align*}
g(z,\nu)&\leq\liminf_{j\to\infty}\widetilde{\calF}^\infty_{\eps_j}(\zeta_j,\beta_j;(-\sfrac12,\sfrac12))
\stackrel{\eqref{e:stima enrg
zetaj}}{\leq}\liminf_{j\to\infty}\widetilde{\calF}^\infty_{\eps_j}(\eta_j,\beta_j;(-\sfrac12,\sfrac12))\\
&\stackrel{ \eqref{e:stima energ xij}}{\leq}
\liminf_{j\to\infty}\inf_{{\calV}^1_z}\widetilde{\calF}^{\cutoffconst_j}_{\eps_j}(\alpha,\beta;(-\sfrac12,\sfrac12))
\stackrel{ \eqref{e:stima gpinf dal basso}}{\leq}\gpinf(z,\nu)\,.
\end{align*}

\smallskip

 \noindent{\bf Step~3. We show that if $\Psi_\infty=|\cdot|^q$, then $g(z,\nu)= g_\scal(|z|)$.}

We first claim that
\[
g(z,\nu)=\widetilde{g}(|z|)\,,
\]
for every $(z,\nu)\in\R^m\times S^{n-1}$, where for every $s\geq 0$
\[
\widetilde{g}(s):=
\lim_{T\uparrow\infty }\inf_{(\gamma,\beta)\in\calU^T_s}\calF^\infty_1(\gamma,\beta;(-\sfrac T2,\sfrac T2))\,,
\]
and the space of test functions can be written as
\begin{multline*}
\calU^T_s:=\{(\gamma,\beta)\in {W^{1,q}}((-\sfrac T2,\sfrac T2);\R^2)\colon
0\leq \beta\leq 1,\ \beta(\pm\sfrac T2)=1,\\ \gamma(-\sfrac T2)=0\,,\ \gamma(\sfrac T2)=s\}\,.
\end{multline*}
Note that the existence of the limit defining $\widetilde{g}$ is guaranteed by Proposition~\ref{p:chargT} in the case $n=1$,
and that the claim is trivial if $z=0$.

To prove the above claim for $|z|>0$, first notice that  if $(\alpha,\beta)\in {\calV}^T_z$, then
$(|\alpha|,\beta)\in \calU^T_{|z|}$,
with $\widetilde{\calF}^\infty_1(\alpha,\beta;(-\sfrac T2,\sfrac T2))\ge \calF^\infty_1(|\alpha|,\beta;(-\sfrac T2,\sfrac
T2))$,
so that $g(z,\nu)\geq\widetilde{g}(|z|)$.

To establish the opposite inequality, {given} $T>0$ and $(\gamma,\beta)\in \calU^{T}_{|z|}$,
we obtain a competitor $(\alpha,\beta)$ for the problem defining $g$ by setting $\alpha(t):=\gamma(t)\frac z{|z|}$.
Moreover,  $\widetilde{\calF}^\infty_1(\alpha,\beta;(-\sfrac T2,\sfrac T2))= \calF^\infty_1(\gamma,\beta;(-\sfrac
T2,\sfrac T2))$, so that the inequality follows.

Finally, the equality  $\widetilde{g}(s)=g_\scal(s)$ is established in \cite[Proposition~7.3]{ContiFocardiIurlano2016}.
\end{proof}

\subsection{Structural properties of the surface energy densities}

In this section, we establish some structural properties of $\gpinf$ and $\gpsup$.
\begin{lemma}\label{lemmapropg}
Let $\gpinf,\,\gpsup:\R^m\times S^{n-1}\to[0,\infty)$ be defined by  \eqref{eqdefGpsnu inf} and \eqref{eqdefGpsnu sup},
respectively. Then, the following properties hold.
\begin{enumerate}
\item\label{e:gp growth}
There is $C>0$ such that, for all $(z,\nu)\in \R^m\times S^{n-1}$,
\begin{equation*}
\frac1C (|z|^{\frac2{p+1}}\wedge 1) \le  \gpinf(z,\nu)\leq
\gpsup(z,\nu) \le C (|z|^{\frac2{p+1}}\wedge 1)\,.
\end{equation*}
\item\label{e:gp subadd}
For any $\nu\in S^{n-1}$ and $z,\,z'\in\R^m$ one has
\[
\gpsup(z+z',\nu)\le \gpsup(z,\nu)+\gpsup(z',\nu).
\]
\item\label{e:gp cont}
$\gpinf,\,\gpsup\in C^0(\R^m\times S^{n-1})$.
\end{enumerate}
\end{lemma}
{Note that property \ref{e:gp growth} immediately implies that there is $C>0$ such that, for all $\nu\in S^{n-1}$,
\begin{equation}\label{e:eqg in zero}
\frac1C\leq \liminf_{z\to 0}\frac{\gpinf(z,\nu)}{|z|^{\frac2{p+1}}}\leq \limsup_{z\to
0}\frac{\gpsup(z,\nu)}{|z|^{\frac2{p+1}}}\leq C\,.
\end{equation}
}
\begin{proof}
The bounds in \ref{e:gp growth} may be derived by estimating $\calF_\eps$ by its one-dimensional counterpart
(see \cite[Prop.~7.3]{ContiFocardiIurlano2016}); the properties in \ref{e:gp subadd} and
\ref{e:gp cont} for $\gpsup$ follow by arguing analogously to \cite[Lemma~3.8]{ContiFocardiIurlano2022},
which deals with the case $p=1$).

To conclude, we establish property \ref{e:gp cont} for $\gpinf$; the same argument works for
$\gpsup$, as well. By \ref{e:gp growth}, for every
$\nu\in S^{n-1}$ the function $\gpinf$ is continuous in $(0,\nu)$.
Therefore, let $(z,\nu),\,(z',\nu')\in\R^m\times S^{n-1}$, with $|z|\cdot|z'|\neq 0$, and consider functions
$(u_j,v_j)\to z'\chi_{\{x\cdot\nu'>0\}}$ in $L^q(Q^{\nu'};\R^{m+1})$ and sequences $\eps_j\to0$,
$\cutoffconst_j\to\infty$.
Let $(R,T)\in SO(m)\times SO(n)$ be such that $R\frac{z'}{|z'|}=\frac z{|z|}$, {$T\nu'=\nu$,
and $|R-\Id_m|\le C |
\frac{z'}{|z'|}
-\frac{z}{|z|}|$,
$|T-\Id_n|\le C |\nu-\nu'|$.
By Lemma~\ref{lemmaQnuQnup}, we can assume that
$TQ^{\nu'}=Q^\nu$. We}  define {on $Q^\nu$} the maps
$(\widetilde{u}_j,\widetilde{v}_j)$ by
\[
\widetilde{u}_j(x):=\frac{|z|}{|z'|}Ru_j(T^{-1}x),\quad\widetilde{v}_j(x):=v_j(T^{-1}x)\,.
\]
Then, a straightforward change of variables leads to
\[
\int_{Q^\nu}|\widetilde{u}_j-z\chi_{\{x\cdot\nu>0\}}|^q\dx
=\left(\frac{|z|}{|z'|}\right)^q\int_{Q^{\nu'}}|u_j-z'\chi_{\{x\cdot\nu'>0\}}|^q\dx\,,
\]
to $\|\widetilde{v}_j-1\|_{L^q(Q^\nu)}=\|v_j-1\|_{L^q(Q^{\nu'})}$, and moreover to
\begin{equation}\label{e:stima enrg widetilde}
\calF^{\cutoffconst_j}_{\eps_j}(\widetilde u_j,\widetilde v_j;Q^\nu)=\left(\frac{|z|}{|z'|}\right)^q
\int_{Q^{\nu'}}(\cutoffconst_j^{q-1}\wedge 
\eps_j^{q-1}
f_p^q(v_j))\Psi_\infty(
{R\nabla u_jT^{-1}})\dx
+\calF^{\cutoffconst_j}_{\eps_j}(0,v_j;Q^{\nu'})\,.
\end{equation}
Therefore, for any $\delta>0$, using \eqref{eqPsiinftyxieta2}, we have $\calL^n$-a.e. on $Q^{\nu'}$  that
$R\nabla u_jT^{-1}=\nabla u_j+(R-\Id_m)\nabla u_jT^{-1}+\nabla u_j(T^{-1}-\Id_n)$ and therefore 
\begin{equation}
 \Psi_\infty(R\nabla u_jT^{-1})
 \le (1+\delta + C_\delta|T^{{-1}}-\Id_{n}|^q+C_\delta|{R}
 -\Id_{m}|^q)\Psi_\infty(\nabla u_j).
\end{equation}
In particular, the latter estimate together with \eqref{e:stima enrg widetilde} implies
\begin{align*}&\calF^{\cutoffconst_j}_{\eps_j}(\widetilde u_j,\widetilde v_j;Q^\nu)\\&\leq
\left\{1\vee \left(\frac{|z|}{|z'|}\right)^q(1+
{\delta+C_\delta}
(|T^{-1}-\Id_n|{^q}+|{R} 
-\Id_m|{^q}))\right\}
\calF^{\cutoffconst_j}_{\eps_j}(u_j,v_j;Q_{\nu'})
\end{align*}
which in turn implies
\begin{equation}\label{e: stima gpinfznu}
\gpinf(z,\nu)\leq\left\{1\vee\left(\frac{|z|}{|z'|}\right)^q
(1+{\delta+C_\delta}
(|T^{-1}-\Id_n|{^q}+|{R} 
-\Id_m|{^q}))\right\}\gpinf(z',\nu')\,.
\end{equation}
Clearly, $T^{-1}\to \Id_n$ as $\nu'\to\nu$, and ${R}
\to \Id_m$ as $z'\to z$, so that
\[
\gpinf(z,\nu)\leq {(1+\delta)}\liminf_{(z',\nu')\to(z,\nu)}\gpinf(z',\nu')
\]
for any $\delta>0$, and hence after letting $\delta\to0$.
Exchanging the roles of $(z,\nu)$ and $(z',\nu')$ yields the reverse inequality and therefore continuity of $\gpinf$ in $(z,\nu)$.
\end{proof}
 
\subsection{$\Gamma-\liminf$}\label{s:liminf}

We start by determining the domain of the eventual $\Gamma$-limit and proving a lower bound inequality which will turn
out to be optimal in the case where the projection property holds for $\Psi_\infty$. To this aim,
consider the functional $\calF_{\inf}(\cdot,\cdot;\cdot):L^1(\Omega;\R^{m+1})\times\calA(\Omega)\to[0,\infty]$ defined
by
\begin{equation}\label{Fp inf}
\calF_{\inf}(u,v;A):=\int_A \Psi^\qc(\nabla u)\dx + \int_{J_u\cap A}\gpinf([u],\nu_u)\dH^{n-1},
\end{equation}
if $A\in\calA(\Omega)$,
$u\in (GSBV\cap L^1(\Omega))^m$ with $\nabla u\in L^q(\Omega;\R^{m\times n})$,
$v=1$ $\calL^n$-a.e. on $\Omega$, and $\calF_{\inf}(u,v;A):=\infty$ otherwise.

We follow in part the approach developed in \cite[Section~4]{ContiFocardiIurlano2022}.

One important ingredient in the lower bound is the fact that the energy density $\Psi^\qc$ is recovered in the limit thanks to Proposition~\ref{prophdeltaqc}. 
We stress that no extra hypothesis on $\Psi$ is assumed besides \eqref{e:Psi gc} and
\eqref{eqpsipsiinf}.
We are now ready to identify the domain of the eventual $\Gamma$-limit and to prove the lower bound.
\begin{proposition}\label{p:lb GSBV}
 Let $(u_\eps,v_\eps)\to (u,v)$ in $L^1(\Omega;\R^{m+1})$ with
\begin{equation}\label{e:finite enrg}
\liminf_{\eps\to0}\Functeps(u_\eps,v_\eps)<\infty\,,
\end{equation}
then $v=1$ $\calL^n$-a.e. in $\Omega$ and $u\in (GSBV(\Omega)\cap L^1(\Omega))^m$ {with $\nabla u\in L^q(\Omega;\R^{m\times
n})$}.
 Moreover, for all $A\in \calA(\Omega)$
 \begin{equation}\label{e:stima dal basso}
   \calF_{\inf}(u,v;A)\le \Gamma(L^1)\hbox{-}\liminf_{\eps\to0} \Functeps(u, 1;A)\,.
 \end{equation}
\end{proposition}
\begin{proof}
We divide the proof into several steps: in the first we argue as in \cite[Sections~4.3 and 4.4]{ContiFocardiIurlano2022}
 to show that $u\in (GBV(\Omega))^m$, in the second we establish  $u\in (GSBV(\Omega))^m$ and a lower bound inequality
for the diffuse part, in the third
we discuss the lower bound for the surface energy, and in the last step we infer \eqref{e:stima dal basso}.
\smallskip

\noindent{\bf Step~1. We show that  $u\in (GBV(\Omega))^m$ and $v=1$ $\calL^n$-a.e. on $\Omega$.}

We fix $A\in\mathcal{A}(\Omega)$, $\delta\in (0,1)$, and $\eps>0$.
We compute, for any pair $(u,v)\in W^{1,q}(\Omega;\R^{m}\times[0,1])$,
\begin{align}\label{eqFdelta p}
  \Functeps&(u,v;A)\geq\int_{A\cap\{\eps^{q-1} f^q_p(v)>1\}}\Psi(\nabla u)\dx\notag\\
 & +\int_{A\cap\{\eps^{q-1} f^q_p(v)\leq 1\}}\Big(\eps^{q-1} f^q_p(v) \Psi(\nabla u) +\delta^{q'}
\frac{(1-v)^{q'}}{q'q^{\sfrac{q'}{q}}\eps}\Big)\dx\notag\\
 &+ \int_A\Big((1-\delta^{q'})\frac{(1-v)^{q'}}{q'q^{\sfrac{q'}{q}}\eps} + \eps^{q-1} |\nabla v|^q\Big)\dx\notag \\
  &\geq\int_{A\cap\{\eps^{q-1} f^q_p(v)>1\}}\Psi(\nabla u)\dx
  +\delta\int_{A\cap\{\eps^{q-1} f^q_p(v)\leq 1\}}\frac{\ell v}{(1-v)^{p-1}}\Psi^{\sfrac1q}(\nabla u)\dx\notag\\
  &+(1-\delta^{q'})^{\sfrac1{q'}} \int_A|\nabla(\Phi(v))|\dx\notag\\
  &\geq\delta  \int_A\Big(\Psi(\nabla u)\wedge \frac{\ell v}{(1-v)^{p-1}}\Psi^{\sfrac1q}(\nabla u)\Big)\dx
  +(1-\delta^{q'})^{\sfrac1{q'}}\int_A |\nabla(\Phi(v))|\dx\,,
\end{align}
where $\Phi:[0,1]\to[0,\frac12]$ is defined by
\begin{equation}\label{e:Phi}
\Phi(t):=\int_0^t (1-s) \ds = t-\frac12 t^2\,.
\end{equation}
We observe that $\Phi$ is strictly increasing, $\Phi(1)=\frac12$ and, in particular, $\Phi$ is bijective. By the coarea
formula, there is
$\bar t\in (\Phi(\delta^{q'}),\Phi(\delta))$ such that
\begin{equation*}
 (\Phi(\delta)-\Phi(\delta^{q'}))
 \calH^{n-1}({A\cap}\partial^*\{\Phi(v)>\bar t\}) \le \int_A |\nabla (\Phi(v))|\dx\,,
\end{equation*}
thus, if we define $\tilde u:=u\chi_{\{\Phi(v)>\bar t\}} \in {(GSBV(A))^m}$ (not highlighting the dependence on $A$,
$\eps$, and $\delta$)
we obtain from \eqref{eqFdelta p} and the monotonicity of $[0,1)\ni t\mapsto\frac t{(1-t)^{p-1}}$,
\begin{equation}\label{e:lb with hdelta}
  \Functeps(u,v;A)\ge
  {\delta^{q'+1}}\int_A h_\delta(\nabla \tilde u) \dx + \beta_\delta\calH^{n-1}(A\cap J_{\tilde u})
  -
h_{{\delta}}(0) \calL^n(\{v\leq\delta\})\,,
\end{equation}
where we have set $\beta_\delta:=(1-\delta^{q'})^{\sfrac1{q'}}
(\Phi(\delta)-\Phi(\delta^{q'}))$, and $h_\delta:\R^m\to[0,\infty)$ is defined in \eqref{eqdefhdelta}.
We may thus argue as in \cite[Remark~4.7]{ContiFocardiIurlano2022} to infer that
if $(u_\eps,v_\eps)\to(u,v)$ in $L^1(\Omega;\R^{m+1})$ and
$\sup_\eps\Functeps(u_\eps,v_\eps)<\infty$, then necessarily
$v=1$ $\calL^n$-a.e. on $\Omega$ and $u\in (GBV(\Omega))^m$.
\smallskip

\noindent{\bf Step~2. We show that   $u\in (GSBV(\Omega))^m$, and that
for every $A\in\calA(\Omega)$ one has
\begin{equation}\label{e:lb p diffuse}
\int_A \Psi^\qc(\nabla u)\dx \le \liminf_{\eps\to0}\Functeps(u_\eps,v_\eps;A)\,.
\end{equation}}

Let $A\in\calA(\Omega)$ be given.
By taking into account \eqref{e:lb with hdelta}
and the usual averaging and truncation argument (cf. \eqref{e:truncation}),
we infer that for every $N\in\N$ there is $k_N\in\{N+1,\ldots,2N\}$ such that
\begin{align*}
  \Functeps(u_\eps,v_\eps;A)\ge  &{\delta^{q'+1}}\int_A h_\delta(\nabla (\mathcal{T}_{k_N}(\widetilde{u}_\eps))) \dx\\
&+ \beta_\delta\calH^{n-1}(A\cap J_{\mathcal{T}_{k_N}(\widetilde{u}_\eps)}) - {h_\delta(0)} \calL^n(\{v_\eps\leq\delta\})-\frac
c{k_N}\,.
\end{align*}
In addition, it is not restrictive to assume $k_N$ independent from $\eps$, up to passing to a subsequence realizing the
$\liminf$ for $\Functeps(u_\eps,v_\eps{;A})$. Therefore, for every $\delta\in(0,1)$,
using formulas (4.23) and (4.24) in \cite[Proposition~4.2]{ContiFocardiIurlano2022} yields that
\begin{align}\label{e:rough lb p}
\liminf_{\eps\to0}\Functeps(u_\eps,v_\eps;A)\geq&\delta^{q'+1}
\int_A h_\delta^\qc(\nabla (\mathcal{T}_{k_N}(u)))\dx \nonumber \\
&
+\delta^{q'+1}\int_A h_\delta^{\qcinfty}(\dd D^c\mathcal{T}_{k_N}(u))-\frac c{k_N}\,.
\end{align}
Here $h_\delta^{\qcinfty}$ denotes the (linear) recession function of $h_\delta^\qc$
defined by
 \begin{equation}\label{eqdefqcinfty}
  h_\delta^{\qcinfty}(\xi):=\limsup_{t\to\infty} \frac{h_\delta^\qc(t\xi)}{t}.
 \end{equation}
Note that the growth conditions imposed on $\Psi$ in \eqref{e:Psi gc} and the very definition of $h_\delta$ in \eqref{eqdefhdelta} yield that
 $h_\delta$ itself is linear for large $|\xi|$
and, more precisely, that
$h_\delta(\xi)= \frac{\ell}{(1-{\delta^{q'}})^{p-1}}\Psi^{\sfrac1q}(\xi)$
for $|\xi|$ large enough. Hence, by \eqref{eqdefqcinfty} we conclude that
\[
h_\delta^{\qcinfty}(\xi)\geq\frac{\ell}{c^{\sfrac1q}(1-{\delta^{q'}})^{p-1}}|\xi|.
\]
In particular, we conclude that $|D^c \mathcal{T}_{k_N}(u)|(A)=0$ by letting $\delta\uparrow 1$ in \eqref{e:rough lb p}. Thus, $ \mathcal{T}_{k_N}(u)\in (SBV(A))^m$ for every $N\in\N$,
in turn implying $u\in(GSBV(A))^m$.  By taking this into account, \eqref{e:rough lb p} rewrites as
\begin{align*}
\liminf_{\eps\to0}\Functeps(u_\eps,v_\eps;A)\geq{\delta^{q'+1}}\int_A h_\delta^\qc(\nabla
(\mathcal{T}_{k_N}(u)))\dx-\frac c{k_N}\,.
\end{align*}
Recalling that $\delta\mapsto h_\delta(\xi)$, and hence
$\delta\mapsto h_\delta^\qc(\xi)$, are nondecreasing, by Beppo Levi's theorem
and Proposition~\ref{prophdeltaqc}, we obtain
\begin{equation*}
\begin{split}
 \int_A \Psi^\qc(\nabla  (\mathcal{T}_{k_N}(u)))\dx=&
 \lim_{\delta\uparrow 1}  \int_A h_\delta^\qc(\nabla  (\mathcal{T}_{k_N}(u)))\dx\\
=& \lim_{\delta\uparrow 1}{\delta^{q'+1}}
 \int_A h_\delta^\qc(\nabla  (\mathcal{T}_{k_N}(u)))\dx
 \le \liminf_{\eps\to0}\Functeps(u_\eps,v_\eps;A)+\frac c{k_N}\,,
 \end{split}
\end{equation*}
and finally \eqref{e:lb p diffuse} follows at once by letting $N\to\infty$ in the latter inequality.

Since $u\in (GSBV(A))^m$ for every $A\in\mathcal{A}(\Omega)$, we have $u\in (GSBV(\Omega))^m$.
Finally, since the growth condition \eqref{e:Psi gc} holds also for $\Psi^\qc$, we conclude in addition that $\nabla u\in
L^q(A;\R^{m\times n})$
 with a bound uniform  with respect to $A$, so that $\nabla u\in L^q(\Omega;\R^{m\times n})$.
\smallskip

\noindent{\bf Step~3. We show that if $(u_\eps,v_\eps)\to(u,v)$ in $L^1(\Omega;\R^{m+1})$, then for every $A\in\calA(\Omega)$ it holds
\begin{equation}\label{e:lb p surface}
\int_{J_u\cap A} \gpinf([u],\nu_u)\dd\calH^{n-1}\leq
\liminf_{\eps\to0}\Functeps(u_\eps,v_\eps;A)\,.
\end{equation}}

Up to subsequences and with a small abuse of notation, we can assume the previous $\liminf$
to be a limit. Let us define the measures $\mu_\eps\in \calM^+_b(A)$ as
\[
\mu_\eps:=\left(f_{\eps}^q(v_\eps) \Psi(\nabla u_\eps) +
\frac{(1-v_\eps)^{q'}}{q'q^{\sfrac{q'}q}\eps}+\eps^{q-1}|\nabla v_\eps|^q\right)
\calL^n{\LL {A}}\,.
\]
Extracting a further subsequence, we can assume that
\begin{equation}\label{convmeps}
\mu_\eps\rightharpoonup \mu \quad \text{weakly-$*$ in }\calM (A)=(C_c^0(A))'
\end{equation}
as $\eps\to0$, for some $\mu\in \calM^+_b(A)$, so that
\[
\liminf_{\eps\to0}\Functeps(u_\eps,v_\eps;A)\geq \mu(A)\,.
\]
Equation \eqref{e:lb p surface} will follow once we have proved that
\begin{equation}\label{densg}
\frac{\dd \mu}{\dd\calH^{n-1}\res J_u}(x_0)\geq \gpinf([u](x_0),\nu_u(x_0)), \quad  \text{for $\calH^{n-1}$-a.e.\ } x_0\in
J_u\cap A\,.
\end{equation}
We will prove the last inequality for points $x_0\in J_u\cap A$ such that
\[
\frac{\dd \mu}{\dd\calH^{n-1}\res J_u}(x_0)=\lim_{\rho\to0}\frac{\mu(Q^\nu_\rho(x_0))}{\calH^{n-1}(J_u\cap
Q^\nu_\rho(x_0))} \quad \text{exists and is finite}\,,
\]
and
\[
\lim_{\rho\to 0}\frac{\mathcal{H}^{n-1}(J_u\cap Q^\nu_\rho(x_0))}{\rho^{n-1}}=1,
\]
where $\nu:=\nu_u(x_0)$ and $Q^\nu_\rho(x_0){:=x_0+\rho Q^\nu}$ is the cube centred
in $x_0$, with side length $\rho$, and one face orthogonal to $\nu$. We remark that such conditions define a set of full
measure in $J_u\cap A$.
First note that
\begin{align}\label{eq:weakconv}
&\frac{\dd \mu}{\dd\calH^{n-1}\res J_u}(x_0)=\lim_{\rho\to0}\frac{\mu(Q^\nu_\rho(x_0))}{\rho^{n-1}}
=\lim_{\substack{\rho\in I\\\rho\to 0}}\lim_{\eps\to0}\frac{\mu_\eps(Q^\nu_\rho(x_0))}{\rho^{n-1}}\,,
\end{align}
where $I:=\{\rho\in(0,\frac{2}{\sqrt n}\,\mathrm{dist}(x_0,\partial A)):\,
\mu(\partial Q^\nu_\rho(x_0))=0\}$, and we have used \eqref{convmeps} for the second equality.

By \eqref{eqpsipsiinf}, for every $\delta\in(0,1)$ one has
$\Psi(\xi)\ge (1-\delta)\Psiinfty(\xi)$ for $\xi$ sufficiently large. Since $\Psiinfty$ is continuous,
there is $C(\delta)>0$ such that
\begin{equation*}
\Psi (\xi)+ C(\delta) \ge (1-\delta)\Psiinfty(\xi)\hskip5mm \text{ for all }\xi\in\R^{{m\times n}}\,.
\end{equation*}
We choose $\delta_\rho\to0$ such that $\rho C(\delta_\rho)\to0$. As $f_{\eps}^q(v_\eps)\le 1$ we have
\begin{equation*}
f_{\eps,p,q}^q(v_\eps) \Psi(\nabla u_\eps)
\ge (1-\delta_\rho)f_{\eps}^q(v_\eps)\Psiinfty  (\nabla u_\eps) -C(\delta_\rho)\,.
\end{equation*}
As $\rho^{1-n} \calL^n(Q_\rho^\nu) C(\delta_\rho)=\rho C(\delta_\rho)\to0$ as $\rho\to0$,
we conclude by \eqref{eq:weakconv} that
\begin{equation}\label{gamm}
\frac{\dd \mu}{\dd\calH^{n-1}\res J_u}(x_0)
\geq \limsup_{{\scriptsize \begin{array}{c}
\rho\!\!\in\!\! I\\ \!\!\rho\!\!\to\!\!0\end{array}}}\limsup_{\eps\to0} \frac{\calF_\eps^{1}(u_\eps,v_\eps;
Q^\nu_\rho(x_0))}{\rho^{n-1}}\,,
\end{equation}
{where we recall that $\calF_\eps^{1}$ is the functional defined in \eqref{Feps* p}
for $M=1$.}
Setting $y:=(x-x_0)/\rho\in Q^\nu$, changing variables {to
$w^\rho(y):=w(\rho y+x_0)$ for $y\in Q^\nu$,
and setting $\cutoffconst_\rho:=\sfrac1\rho$, we obtain
\begin{equation*}
\calF_\eps^1(u_\eps,v_\eps; Q^\nu_\rho(x_0))=
\rho^{{n-1}}\calF_{\sfrac\eps\rho}^{\cutoffconst_\rho}(u_\eps^\rho,v_\eps^\rho; Q^\nu).
\end{equation*}
Thus, the previous estimate becomes}
\begin{equation*}
\frac{\dd \mu}{\dd\calH^{n-1}\res J_u}(x_0)\geq \limsup_{{\scriptsize \begin{array}{c}
\rho\!\!\in\!\! I\\ \!\!\rho\!\!\to\!\!0\end{array}}}\limsup_{\eps\to0}
\calF_{\sfrac\eps\rho}^{\cutoffconst_\rho}(u_\eps^\rho,v_\eps^\rho;Q^\nu).
\end{equation*}
Recalling that $u_\eps\to u$ in $L^1(\Omega;\R^m)$,
by diagonalization, we can find subsequences $\rho_k\to0$ and $\eps_k\to0$ such that $u_{{\eps_k}}^{\rho_k}\to
[u](x_0)\chi_{\{y\cdot\nu>0\}}+u^-(x_0)$ in $L^1(Q^\nu;\R^m)$,
$\sfrac{\eps_k}{\rho_k}\to0$, and
{\begin{equation*}
\frac{\dd \mu}{\dd\calH^{n-1}\res J_u}(x_0)\geq
\lim_{k\to\infty}\calF_{\sfrac{{\eps_k}}{\rho_k}}^{\cutoffconst_{\rho_k}}(u_{{{\eps_k}}}^{\rho_k},v_{{{\eps_k}}}^{\rho_k
};Q^\nu).
\end{equation*}}
Since $\calF_\eps^{\cutoffconst}$ is invariant under translations of the first argument, we find
\[
\frac{\dd \mu}{\dd\calH^{n-1}\res J_u}(x_0)\geq
\liminf_{k\to\infty}\calF_{\sfrac{{\eps_k}}{\rho_k}}^{\cutoffconst_{\rho_k}}
(u_{{\eps_k}}^{\rho_k},v_{{\eps_k}}^{\rho_k};Q^\nu)\geq {\gpinf}([u](x_0),\nu_u(x_0)),
\]
which is \eqref{densg}, and this concludes the proof.
\smallskip

\noindent{\bf Conclusion.}
The lower bound inequality in \eqref{e:stima dal basso}
is a consequence of \eqref{e:lb p diffuse}, of \eqref{e:lb p surface}, and of
\cite[Proposition~1.16]{Braides},  taking into account that the terms on the left-hand side in \eqref{e:lb p diffuse}
and in \eqref{e:lb p surface} define mutually singular measures.
\end{proof}
 
\subsection{$\Gamma-\limsup$}

The next proposition establishes an upper bound for the $\Gamma$-limits. To this aim,
consider the functional $\calF_{\sup}(\cdot,\cdot;\cdot):L^1(\Omega;\R^{m+1})\times\calA(\Omega)\to[0,\infty]$ defined
by
\begin{equation}\label{Fp sup}
\calF_{\sup}(u,v;A):=\int_{A} \Psi^\qc(\nabla u)\dx + \int_{J_u\cap A}
(\gpsup)_{BV}([u],\nu_u)\dH^{n-1},
\end{equation}
if $A\in\calA(\Omega)$, $u\in (GSBV\cap L^1(\Omega))^m$ with $\nabla u\in L^q(\Omega;\R^{m\times n})$,
$v=1$ $\calL^n$-a.e. on $\Omega$, and $\calF_{\sup}(u,v;A):=\infty$ otherwise.

We recall
that by \eqref{eqpsipsiinf}, for $\delta>0$ there exists $C(\delta)>0$ such that
\begin{equation}\label{psipsiinf}
\Psi(\xi)\leq (1+\delta)\Psi_\infty(\xi)+C(\delta), \text{ for all }\xi\in\R^{m\times n}.
\end{equation}

\begin{proposition}\label{p:ubp}
Let $\Functeps$ be the functional defined in \eqref{functeps p}. Then for all
$(u,v)\in L^1(\Omega;\R^{m+1})$ it holds
\begin{equation}\label{Gammalimsup0}\Gamma(L^1)\text{-}\limsup_{\eps\to0}\Functeps(u,v)\leq \calF_{\sup}(u,v).
\end{equation}
\end{proposition}
\begin{proof}
It is enough to prove \eqref{Gammalimsup0} when $v=1$ $\calL^n$-a.e. and $u$ is such that
{$\calF_{\sup}(u,1)<\infty$}.

Let us assume for the moment that
\begin{equation}\label{Gammalimsup}
\Gamma(L^1{(\Omega)})\text{-}\limsup_{\eps\to0}\Functeps(u,1)\leq {H_1}(u; {\Omega'})
\end{equation}
{for {all} open sets $\Omega'$ with $\overline\Omega\subset\Omega'$ and {all}}
  $u\in PA({\R^n};\R^m)$, where {for $A\in\calB(\R^n)$}
  {\begin{equation*}
  {H_1}(u;{A}):=\begin{cases}
  \displaystyle \int_{{A}} \Psi(\nabla u)\dx+\int_{J_u{\cap{A}}}\gpsup([u],\nu_u)\dH^{n-1}, & \text{if }u\in
  PA({\R^n};\R^m),\\
  \infty, & \text{otherwise.}
  \end{cases}
  \end{equation*}}
{Then, 
\begin{equation}
\Gamma(L^1{(\Omega)})\text{-}\limsup_{\eps\to0}\Functeps(u,1)\leq {H_1(u;\overline\Omega)=H_0(u),}
 \end{equation}
  for any $u\in PA({\R^n};\R^m)$
with $\calH^{n-1}(J_u\cap\partial\Omega)=0$, {where $H_0$ is the functional introduced in \eqref{Hg}, with bulk density $\Psi$ and surface density $\gpsup$.}}
The relaxation Theorem~\ref{relaxation2} and the lower semicontinuity of
$\mathcal{F}''(\cdot):=\Gamma(L^1)$-$\limsup_{k\to\infty}\Functeps(\cdot,1)$ with respect to $L^1(\Omega;\R^m)$ convergence imply that \eqref{Gammalimsup0} holds for all $u\in {(GSBV(\Omega))^m}$.

Hence, we only need to show that \eqref{Gammalimsup} holds.
The construction is based on
Proposition~\ref{periodicity} and on a decomposition of the domain
similar  to the one used in proving
Theorem~\ref{relaxation2}, {for which one covers} a large part of the jump set with cubes of scale $\rho$. The current construction, however, needs to produce Sobolev functions and not functions with jumps; therefore the interfaces need to be regularized. We do this using a mollification on a scale $\eps$ close to the jump set, and an interpolation on an intermediate scale $(1-\lambda)\rho$ around the boundaries of the cubes.
We use the same notation as in Theorem~\ref{relaxation2} to emphasize the similarities, but since there are several differences in the details, we repeat the common parts of the argument for greater clarity.

Fix $u\in PA(\R^n;\R^m)$ and an open set $\Omega'$ with $\Omega\subset\subset\Omega'$.
Let $\{F_j\}_{j\in\N}$ denote the faces  of the simplexes in the decomposition of $\R^n$ associated to $u$.
Possibly splitting the simplexes, we can ensure that $\text{diam}(F_j)\le \frac12 \dist(\Omega,\partial\Omega')$ for all $j$. We
set
\[
F_*:=\bigcup_j F_j. \hskip1cm
\]
Let
\begin{equation}
 M_u:=\|\nabla u\|_{L^\infty(\Omega')} + \calH^{n-1}(F_*\cap\Omega')+\calH^{n-2}({\bigcup_j\partial F_j}\cap\Omega'),
\end{equation}
{where the boundary operator is understood in the $(n-1)$-dimensional sense;}
since $u$ is piecewise affine, we have ${M_u}<\infty$.
We first construct a sequence with bounded energy. We recall that $u\in PA(\R^n;\R^m)$ by assumption, and define
for any $\eps>0$ similarly to \eqref{eqdefUjVj}
\begin{equation}
U_\eps:=u\ast\varphi_{\eps}\,,\hskip1cm
V_\eps:=\chi_{\R^n\setminus (F_*)_{2\eps}}\ast\varphi_{\eps} \,.
\end{equation}
We observe that, since $\varphi_\eps\in C^\infty_c(B_\eps)$ is an even function
and $u$ is affine on each connected component of $\R^n\setminus F_*$, we have $U_\eps=u$ outside $(F_*)_{\eps}$.
At the same time, $V_\eps=0$ in $(F_*)_\eps$ and
 $V_\eps=1$ outside $(F_*)_{3\eps}$,
so that with $f_{\eps}(V_\eps)=0$ in $(F_*)_\eps$
and $f_{\eps}(V_\eps)\le 1$ everywhere we obtain,
for $\eps<\frac13\dist(\Omega,\partial\Omega')$,
\begin{equation}
 \calF_\eps(U_\eps, V_\eps)
 \le \int_{\Omega}
 \Psi(\nabla u) \dx
 +
 \calF_\eps^\infty(0,V_\eps;\Omega\cap (F_*)_{3\eps}),
\end{equation}
where we recall that
\begin{equation}
\calF_\eps^\infty(0,v;A)=
\int_A
 \Big(\frac{(1-v)^{q'}}{q'q^{\sfrac{q'}q}\eps} + \eps^{q-1}|\nabla v|^q\Big)\dx
\end{equation}
only contains the phase-field part of the functional.

Using $V_\eps\in[0,1]$ and $|\nabla V_\eps|\le C/\eps$ leads to
\begin{equation}
 \calF_\eps^\infty(0,V_\eps;\Omega\cap (F_*)_{3\eps})
\le \frac{C}{\eps} \calL^n(\Omega\cap (F_*)_{3\eps}).
\end{equation}
Since $F_*$ is locally a finite union of polygons,
we conclude that
\begin{equation}
 \limsup_{\eps\to0}
 \calF_\eps^\infty(0,V_\eps;\Omega\cap (F_*)_{3\eps})<\infty,
\end{equation}
so that the sequence $(U_\eps,V_\eps)$ is bounded in energy, as claimed.

We
next modify the sequence in a set covering a large part of $(F_*)_{3\eps}$ in order to replace this estimate by one with the optimal constant. To do this, we
shall introduce a parameter $\rho>0$ and cover a large part of each face $F_j$
by cubes on a scale $\rho$,
as in the proof of Theorem~\ref{relaxation2}.
In order to quantify ``a large part'' it is useful first to introduce another parameter $\delta>0$, that will be chosen later.
From now on, it is convenient to restrict to the indices in the finite set
$J:=\{j: F_j\cap\overline\Omega\ne \emptyset\}$.
We first fix relatively open sets $F_j'\subset F_j$ such that $\dist(F_j',{\partial F_j})
>0$ and
\begin{equation}\label{eqchoiceFjp2}
 \sum_{j\in J} \calH^{n-1}(F_j\setminus F_j')\le \delta.
\end{equation}
Let {$A_j:\R^n\to\R^n$} be an affine {isometry} that maps $\R^{n-1}\times \{0\}$ to the $(n-1)$-dimensional affine space containing $F_j$,  {and set $\nu_j:={(\nabla A_j) e_n}$}.
For $\pcube\in \rho\Z^{n-1}$ we let
$Q_{j,\pcube}:=A_j(\pcube+Q_\rho)$.
Let $I_j$ denote the set of $\pcube\in\rho\Z^{n-1}$ such that $Q_{j,\pcube}\cap F_j'\ne\emptyset$.
Since $\dist(F_j',F_l)>0$ for $j\ne l$,
for $\rho$ sufficiently small,
for all $l\ne j\in J$, all $\pcube\in I_j$, and all
$\pcube'\in I_l$, one has
$Q_{j,\pcube}\cap F_l=\emptyset$
and $Q_{j,\pcube}\cap Q_{l,\pcube'}=\emptyset$.

 In each cube $Q_{j,\pcube}$, we intend to apply Proposition~\ref{periodicity}. However, {this} proposition {provides} a construction for limit functions that are constant on the two halves of the cube, whereas $u$ is affine on each of them. Therefore, we first introduce a piecewise constant approximation in each cube, {followed by} an additional interpolation step.
 
 For each $j\in J$ and $\pcube\in I_j$,
let $u_{j,\pcube}^\pm:=u^\pm(A_j(\pcube))$ denote the two traces of $u$ at the center of $Q_{j,\pcube}$ (the traces have point
values,
since $u$ is piecewise affine).
From these two values we construct
a piecewise constant function by
\begin{equation}\label{e:wjy}
w_{j,\pcube}(x):=
 \begin{cases}
 u_{j,\pcube}^+, & \text{ if } x\in
 A_j(\pcube+\R^{n-1}\times[0,\infty)),\\
 u_{j,\pcube}^-, & \text{ if } x\in
 A_j(\pcube+\R^{n-1}\times(-\infty,0)).
 \end{cases}
\end{equation}
We remark that $\|w_{j,\pcube}-u\|_{L^\infty(Q_{j,\pcube})}\le \rho \sqrt n M_u$,
and in particular that
$\|[w_{j,\pcube}]-[u]\|_{L^\infty(J_u\cap Q_{j,\pcube};\calH^{n-1})}\le 2\rho \sqrt n M_u$.

We fix a cutoff function
$\theta_{\delta}\in C^\infty_c(Q_1;[0,1])$ with $\theta_{\delta}=1$ in $Q_{1-\frac{\delta}{2}}$, {for $\delta\in(0,\sfrac12)$.}
For each $j\in J$ and $\pcube\in I_j$, we define
$\theta_{j,\pcube}\in C^\infty_c(Q_{j,\pcube})$ by 
$\theta_{j,\pcube}({A_j(\pcube+\rho x)}):=\theta_{{\delta}}(x)$
and let $q_{j,\pcube}:=A_j(\pcube+Q_{(1-\delta)\rho})\subset\subset Q_{j,\pcube}$.
We set
\begin{equation}\label{e:Wjy}
W_{j,\pcube}:=\left(w_{j,\pcube}\theta_{j,\pcube}+u(1-\theta_{j,\pcube})\right)\ast\varphi_{\eps}.
\end{equation}
Then  for $\eps$ sufficiently small we have
\begin{equation}\label{eqboudnaryWjp}
W_{j,\pcube}=U_\eps \text{ on $\partial Q_{j,\pcube}$, \hskip5mm and }\hskip5mm 
W_{j,\pcube}=w_{j,\pcube}\ast\varphi_\eps \text{ on }\partial q_{j,\pcube}.
\end{equation}

For each $j$ and $\pcube$ we use
Proposition~\ref{periodicity}, with $\nu={\nu_j}$ 
and $z=u^+_{j,\pcube}-u^-_{j,\pcube}$. By
 an elementary translation and scaling argument, we obtain functions
$(u_{\eps,j,\pcube},v_{\eps,j,\pcube})
\in W^{1,q}(q_{j,\pcube};\R^m\times[0,1])$ such that
$u_{\eps,j,\pcube}
\to w_{j,\pcube}$ as $\eps\to0$ in $L^q(q_{j,\pcube};\R^{m})$,
\begin{equation}\label{eqfunctepsqjp}
\lim_{\eps\to0}\Functeps^\infty(u_{\eps,j,\pcube},
v_{\eps,j,\pcube};q_{j,\pcube})={(1-\delta)^{n-1}\rho^{n-1}}\gpsup([u](A_j(\pcube)),\nu_j),
\end{equation}
and
\begin{equation}\label{equbboundaryqjp}
u_{\eps,j,\pcube}=w_{j,\pcube}*\varphi_{\eps}=W_{j,\pcube},\quad  
v_{\eps,j,\pcube}=V_\eps\quad \text{on }\partial {q}_{j,\pcube}. 
\end{equation}
Using \eqref{psipsiinf}, {we have
	\begin{equation*}
	\Functeps({u_{\eps,j,\pcube},v_{\eps,j,\pcube}};q_{j,\pcube})\leq
	(1+\delta)\Functeps^\infty({u_{\eps,j,\pcube},v_{\eps,j,\pcube}};q_{j,\pcube})+C(\delta)\rho^n,
	\end{equation*}
and then}
\begin{equation}\label{eqfunctepsqjpb}
\lim_{\eps\to0}\Functeps(u_{\eps,j,\pcube},
v_{\eps,j,\pcube};q_{j,\pcube})
\le (1+\delta)
\rho^{n-1}\gpsup([u](A_j(\pcube)),\nu_j)
+C_\delta \rho^n.
\end{equation}

Recalling the bounds after the definition of $w_{j,\pcube}$,
{setting $g_0(s):=s^{\frac2{p+1}}\wedge 1$ for every $s\geq 0$,
and using \ref{e:gp growth} and \ref{e:gp subadd} in Lemma~\ref{lemmapropg}, we have}
\begin{equation}\label{eqgpsupwu}
\begin{split}
\rho^{n-1} \gpsup([u](A_j(\pcube)),\nu_j)
=&\int_{J_u\cap Q_{j,\pcube}}
\gpsup([w_{j,\pcube}],\nu_j)
 \dH^{n-1}\\
 \le &
\int_{J_u\cap Q_{j,\pcube}}
\left(\gpsup([u],\nu_j)+{C} g_0(2\rho\sqrt n M_u)\right)
 \dH^{n-1}.
 \end{split}
 \end{equation}

We finally set
\begin{equation}\label{e:urho2}
 \begin{split}
u_\eps^{\delta,\rho}:=&U_\eps 
 +\sum_{j\in J} \sum_{\pcube\in I_j}\left(
 \chi_{q_{j,\pcube}} (u_{\eps,j,\pcube}-W_{j,\pcube})
 +\chi_{Q_{j,\pcube}} (W_{j,\pcube}-U_\eps)\right)
,\\
 v_\eps^{\delta,\rho}:=&V_\eps
 +\sum_{j\in J} \sum_{\pcube\in I_j}
 \chi_{q_{j,\pcube}} (v_{\eps,j,\pcube}-V_\eps).
 \end{split}\end{equation}
The boundary data in \eqref{equbboundaryqjp} show that the characteristic functions do not introduce any jump on 
each $\partial q_{j,\pcube}$; and with the boundary data in \eqref{eqboudnaryWjp} the same holds
on each $\partial Q_{j,\pcube}$. Therefore $(u_\eps^\rho, v_\eps^\rho)\in W^{1,q}(\Omega;\R^{m}\times[0,1])$.

\newcommand\Omegain{\Omega^\text{in}}
\newcommand\Omegamid{\Omega^\text{mid}}
\newcommand\Omegaout{\Omega^\text{out}}

In order to estimate the energy we decompose the domain in three parts, 
\begin{equation}
 \Omegain:=\bigcup_{j\in J}\bigcup_{\pcube\in I_j} q_{j,\pcube}, 
 \hskip5mm
 \Omegamid:=\bigcup_{j\in J}\bigcup_{\pcube\in I_j} Q_{j,\pcube}\setminus q_{j,\pcube}, 
 \hskip5mm
 \Omegaout:=\Omega\setminus \Omegain\setminus\Omegamid.
\end{equation}
As the cubes {$Q_{j,\pcube}$} are disjoint, $\sum_{j\in J}\sum_{\pcube{\in I_j}}\rho^{n-1}\le \calH^{n-1}(\Omega\cap F_*)$ implies that their number is bounded by ${C_u}\rho^{1-n}$, so that {for $\delta\in(0,\sfrac12)$ we have}
\begin{equation}\label{e: Omega volume estimates}
 \calL^n(\Omegain)\le C_u\rho,
 \hskip5mm
 \calL^n(\Omegamid)\le C_u \rho \delta,
\end{equation}
where $C_u$ denotes a constant that may depend on $u$ (and $\Omega$) but not on $\eps$, $\rho$, and $\delta$.

We start from $\Omegain$.
Since in each $q_{j,\pcube}$ we have $(u_\eps^{\delta,\rho},v_\eps^{\delta,\rho})=(u_{\eps,j,\pcube},
v_{\eps,j,\pcube})$,
by \eqref{eqfunctepsqjpb}, \eqref{eqgpsupwu}, {and the first estimate in \eqref{e: Omega volume estimates}}
\begin{equation}
\label{eqFomegain}
\begin{split}
& \limsup_{\eps\to0}\Functeps(u_{\eps}^{\delta,\rho},
v_{\eps}^{\delta,\rho};\Omegain)
\le (1+\delta)\sum_{j,\pcube}\Big(\rho^{n-1}\gpsup([u](A_j(\pcube)),\nu_j)+C_\delta\rho^n\Big)\\
&\le (1+\delta){H_1}(u;J_u\cap \Omegain)
+{C M_u}
g_0(2\rho\sqrt n M_u)
+{C C_\delta} \calL^n(\Omegain)
\\
&\le (1+\delta){H_1}(u;J_u\cap\Omega')
+C{M_u}g_0(2\rho\sqrt n M_u)
+{C_u C_\delta\rho}
,
\end{split}\end{equation} and similarly, using
$u_{\eps,j,\pcube}\to w_{j,\pcube}$ and the definition of $w_{j,\pcube}$ in \eqref{e:wjy} and
that of $W_{j,\pcube}$ in \eqref{e:Wjy} we infer
\begin{equation}\label{eqboundl1inside}
\begin{split}
 \limsup_{\eps\to0} \| u_{\eps}^{\delta,\rho}\|_{L^1(\Omegain)} = 
& \sum_{j,\pcube}\|w_{j,\pcube}\|_{L^1(q_{j,\pcube})}
\le \calL^n(\Omegain) \|u\|_{L^\infty(\Omega')}
\le C_u \rho.
\end{split}\end{equation}

We next address $\Omegamid$.
In $Q_{j,\pcube}\setminus (F_j)_\eps$,
letting $Q'_{j,\pcube}$ be a cube with the same center and twice the side length,
\begin{equation}\begin{split}
 \|\nabla W_{j,\pcube}-\nabla u\|_{L^\infty(Q_{j,\pcube}\setminus (F_j)_\eps)}
 \le &
 \|\nabla ((w_{j,\pcube}-u)\theta_{j,\pcube})\|_{L^\infty(Q'_{j,\pcube})}\\
\le & \|\nabla \theta_{j,\pcube}\|_{{L^\infty(Q'_{j,\pcube})}} \|w_{j,\pcube}-u\|_{L^\infty(Q'_{j,\pcube})}+
 \|\nabla u\|_{L^\infty(\Omega')},
\end{split}\end{equation}
where we used that {$U_\eps=u$ outside $(F_*)_{\eps}$}, $\nabla w_{j,\pcube}=0$, and $\theta_{j,\pcube}\in[0,1]$ pointwise.
Recalling that 
$|\nabla \theta_{j,\pcube}|\le C_\delta/\rho$
and $|\nabla u|\le M_u$, we conclude
\begin{equation}
 \|\nabla W_{j,\pcube}-\nabla u\|_{L^\infty(Q_{j,\pcube}\setminus (F_j)_\eps)}
 \le C_\delta M_u,
\end{equation}
which implies that $\Psi(\nabla W_{j,\pcube})\le C_{\delta}(M_u^q+1)$ pointwise in  $Q_{j,\pcube}\setminus (F_j)_\eps$.

Since ${v_{\eps}^{\delta,\rho}=V_\eps}=0$ in 
$\Omega\cap (F_*)_\eps\setminus \Omegain$,
we conclude
\begin{equation}
\label{eqpsiomegamid}
\begin{split}
\int_{\Omegamid} f^q_{\eps}(v_\eps^{\delta,\rho})\Psi(\nabla u_\eps^{\delta,\rho}) \dx\le&
\int_{\Omegamid\setminus (F_*)_\eps} \Psi(\nabla u_\eps^{\delta,\rho}) \dx \\
\le &
C_{\delta}(M_u^q+1) \calL^n(\Omegamid)
\le 
C_{\delta}(M_u^q+1) C_u\rho.
\end{split}
\end{equation}
Similarly,
\begin{equation}
\| W_{j,\pcube}-{u*\varphi_\eps}\|_{L^1(Q_{j,\pcube})}\le
\| (w_{j,\pcube}-u)\theta_{j,\pcube}\|_{L^1(Q'_{j,\pcube})}\le 
\| w_{j,\pcube}-u\|_{L^1(Q'_{j,\pcube})}\le 
C_u\rho^{{n+1}},
\end{equation}
so that {for $\eps$ sufficiently small, using that $u*\varphi_\eps\to u$ {in} $L^{1}_{\loc}(\R^n;\R^m)$ as $\eps\to0$, we get}
\begin{equation}\label{eqL1omegamid}
\| W_{j,\pcube}-u\|_{L^1(\Omegamid)}\le  C_u{\rho}^2.
\end{equation}
Finally, in $\Omegaout\setminus (F_*)_\eps$ we have $u_\eps^{\delta,\rho}=U_\eps=u$,
so that 
\begin{equation}\label{eqpsiomegaout}
\begin{split}
\int_{\Omegaout} f^q_{\eps}(v_\eps^{\delta,\rho})\Psi(\nabla u_\eps^{\delta,\rho}) \dx\le&
\int_{\Omegaout\setminus (F_*)_\eps} \Psi(\nabla u) \dx
\le \int_\Omega \Psi(\nabla u) \dx.
\end{split}
\end{equation}

The {last two terms of $\calF_\eps$ are} estimated jointly in $\Omegamid$ and $\Omegaout$.
Since $v_\eps^{\delta,\rho}=V_\eps$ in $\Omega\setminus\Omegain$,
 $V_\eps=1$ in $\Omega\setminus (F_*)_{3\eps}$,
 and $|\nabla V_\eps|\le C/\eps$ everywhere,
\begin{equation}\label{eqestATeps1}
 \calF_\eps^\infty(0,
 v_\eps^{\delta,\rho};\Omega\setminus\Omegain)
= \calF_\eps^\infty(0,V_\eps;\Omega\setminus\Omegain)
 \le  \frac C\eps \calL^n(\Omega\cap (F_*)_{3\eps}\setminus\Omegain).
\end{equation}
To compute this limit we observe that, since $\Omega\cap F_*\subseteq\cup_j F_j$, the $F_j$ are contained in planes, and $\Omegain$ does not touch the $(n-{2})$-dimensional boundary of the $F_j$,
\begin{equation}
 \Omega\cap (F_*)_{3\eps}\setminus\Omegain
 = \Omega\cap\bigcup_j (F_j)_{3\eps}\setminus \Omegain
 \subset \bigcup_{{j\in J}} (F_j\setminus \Omegain)_{3\eps}.
\end{equation}
As {each} $F_j\setminus\Omegain$ is a compact subset of a plane, its Minkowski content equals its $\calH^{n-1}$  measure, so that
\begin{equation}
\limsup_{\eps\to0} \frac
{ \calL^n(\Omega\cap
(F_*)_{3\eps}\setminus\Omegain)}
{6\eps}
 \le \sum_{{j\in J}} \Big(\calH^{n-1}(F_j\setminus F_j')
 +C\delta \calH^{n-1}(F_j')\Big).
\end{equation}
Combining this with \eqref{eqestATeps1} and \eqref{eqchoiceFjp2} leads to
\begin{equation}\label{eqATnotin}
 \calF_\eps^\infty(0,v_\eps^{\delta,\rho};\Omega\setminus\Omegain)
\le  C \delta + C_u \delta.
\end{equation}

Putting together 
\eqref{eqFomegain},
\eqref{eqpsiomegamid},
\eqref{eqpsiomegaout},
\eqref{eqATnotin}
shows that for any {$\delta\in (0,\sfrac12)$} there is a sequence with
\begin{equation}
\limsup_{\rho\to0}\limsup_{\eps\to0} \calF_\eps(u_\eps^{\delta,\rho},v_\eps^{\delta,\rho};\Omega)
 \le (1+\delta) {H_1}(u;\Omega')
+C \delta + C_u \delta.
\end{equation}
At the same time, 
from \eqref{eqboundl1inside},
\eqref{eqL1omegamid}, and $ u_\eps^{\delta,\rho}={U_\eps}$
on $\Omegaout$,
\begin{equation}
\limsup_{\rho\to0}\limsup_{\eps\to0} \| u_\eps^{\delta,\rho}-u\|_{L^1(\Omega)}=0.
\end{equation}
Taking the limit $\delta\to0$, and finally taking a diagonal subsequence, concludes the proof.

\end{proof}
{
We are now ready to prove Theorem~\ref{t:finale p}.
\begin{proof}[Proof of Theorem~\ref{t:finale p}]
Propositions~\ref{p:lb GSBV} and~\ref{p:ubp} imply in particular that $\gpinf\leq(\gpsup)_{BV}$, so that by  \eqref{e:gpinf=gpsup} we have $\gpsup\leq(\gpsup)_{BV}$, and then $\gpsup=(\gpsup)_{BV}$. Hence, $\calF_{\inf}=\calF_{\sup}$ and the common value $g=\gpinf=\gpsup$ is $BV$-elliptic.
\end{proof}
}

\section*{Acknowledgements}
SC and FI gratefully
thank the University of Florence for the warm hospitality of the DiMaI ``Ulisse
Dini'', where part of this work was carried out.
This work was partially supported
by the Deutsche Forschungsgemeinschaft (DFG) through project CRC1720 - 539309657.
MF is a member of GNAMPA - INdAM. MF has been supported by the European Union - Next Generation EU, Mission 4 Component 1 CUP B53D23009310006, codice 2022J4FYNJ, PRIN2022 project "Variational
methods for stationary and evolution problems with singularities and interfaces".

\section*{Data Availability Statement} Data sharing is not applicable to this article as no
datasets were generated or analysed during the current study.

\bibliographystyle{alpha-noname}
\bibliography{cfi}

\end{document}